\documentclass[11pt]{article}
\usepackage{geometry}  
\usepackage[french]{babel} 
           
\usepackage{graphicx}
\usepackage{amssymb}
\usepackage{epstopdf}
\newtheorem{remark}{Remark}
\newtheorem{thm}{Theorem}[section]
\newtheorem{prop}[thm]{Proposition}

\newtheorem{cor}[thm]{Corollary}
\newenvironment{proof}{{\textbf Proof}  \/\rm}{$\square$}

\newtheorem{lem}[thm]{Lemma}
\newenvironment{proco}{\textbf{Proof of Corollary} \/\rm}{$\square$}
\newenvironment{proth}{\textbf{Proof of Theorem} \/\rm}{$\square$}

\begin{document}

\centerline{\textbf{\Large 
{On spectral gap properties  and extreme value}}}

\centerline{\textbf{\Large{theory for multivariate  affine stochastic recursions }}}

\vskip 3mm
\centerline{Y. Guivarc'h$^{a,}$\footnote{Corresponding author} and \'E. Le Page$^{b}$}

 Universit\'e de Rennes 1, 263 Av. du G\'en\'eral Leclerc, 35042 Rennes Cedex, France, 
 
 U.M.R CNRS 6625, E-mail : yves.guivarch@univ-rennes1.fr 

Universit\'e  de Bretagne Sud, Campus de Tohannic, 56017 Vannes, France, 

U.M.R CNRS 6205, E-mail : emile.le-page@univ-ubs.fr
\vskip 6mm
\textbf{Abstract}

\noindent We consider a general multivariate affine stochastic recursion and the associated Markov chain on $\mathbb R^{d}$. We assume a natural geometric condition which implies existence of an unbounded stationary solution and we show that the large values of the associated stationary process follow extreme value properties of classical type, with a non trivial extremal index. We develop some explicit consequences such as convergence to  Fr\'echet's law or to an exponential law, as well as convergence to a stable law. The proof is based on a spectral gap property for the action of associated positive operators on spaces of regular functions with slow growth, and on the clustering properties of large values in the recursion.
\vskip 3mm
\section{ Introduction}
 \noindent Let $V=\mathbb R^{d}$ be the $d$-dimensional Euclidean space and let $\lambda$ be a probability on the affine group $H$ of $V$, $\mu$ the projection of $\lambda$ on the linear group $G=GL (V)$. 
Let $(A_{n}, B_{n})$ be a sequence of  $H$-valued i.i.d. random variables distributed according to $\lambda$ and let us consider the affine stochastic recursion on $V$ defined by 

\centerline{$X_{n}=A_{n} X_{n-1}+B_{n}$}

\noindent  for $n\in \mathbb N$. We denote by $P$ the corresponding Markov kernel on $V$ and by $\mathbb P$ the product measure $\lambda^{\otimes \mathbb N}$ on $H^{\mathbb N}$. 
Our geometric hypothesis (c-e) on $\lambda$ involves contraction and expansion properties  and   is robust for the relevant   weak topology if $d>1$ ; it implies   that $P$ has a unique invariant probability $\rho$ on $V$, with unbounded support. In our situation (see \cite{13}), the quantity $\rho\{|v| > t\}$ is asymptotic $(t\rightarrow \infty)$ to $\alpha^{-1}c\  t^{-\alpha}$ with $\alpha>0$, $c>0$. More precisely  the measure $\rho$ is  multivariate regularly varying,  a basic property for the development of extreme value theory i.e. for the study of exceptionally large values of $X_{k} (1\leq k\leq n)$ for $n$ large (see \cite{30}), our main goal in this paper. One non trivial aspect of hypothesis (c-e) is that it implies unboundness of the subsemigroup of $G$ generated by supp$(\mu)$, hence  multiplicity of large values of $|X_{n}|$. Also, from a heuristic point of view, hypothesis (c-e) allows us,  to reduce the $d$-dimensional linear situation to a 1-dimensional setting.  

\noindent In such a situation of weak dependance, spectral gap properties of operators associated to $P$ play  an important role via a  multiple mixing condition  described in \cite{7} for step functions. We observe that the same idea has been used in various closely related situations : limit theorems for the largest coefficient in the continued fraction expansion of a real number   (see \cite{27}, \cite{37}), limit theorems for the ergodic sums $\displaystyle\mathop{\Sigma}_{k=1}^{n} X_{k}$ along the above stochastic recursion (see \cite{15}), geodesic excursions on the modular surface (see \cite{17}, \cite{28}),   "shrinking larget" problem (see \cite{26}, \cite{32}). In our setting, spectral gap properties of operators associated to $P$ allow us to study the path behaviour of the Markov chain defined by $P$.  In the context of geometric ergodicity (see \cite{31}) for the Markov operator $P$ acting on measurable functions, assuming a density condition on the law of $B_{n}$, partial results were  obtained in \cite{22}. However simple examples show (see below) that, in general, $P$ has no spectral gap in $\mathbb L^{2} (\rho)$. Here we go further in this direction replacing geometric ergodicity by condition (c-e).   Condition (c-e) implies that  the operator $P$ has a spectral gap property in the spaces of H\"older functions with polynomial growth considered below, a fact which allow us to deduce convergence with exponential speed on H\"older functions. A typical example of this situation occurs if the support of $\lambda$ is finite and generates a dense subsemigroup of the affine group $H$.

\noindent  If $\mathbb Z_{+}$ is the set of non negative integers, we denote by $\mathbb P_{\rho}$ the Markov probability on $V^{\mathbb Z_{+}}$ defined by the kernel $P$ and the initial probability $\rho$. 
\noindent In this paper we 
 establish  spectral gap properties for the action of $P$  on H\"older functions and we deduce fundamental extreme value statements  for the point processes defined by the $\mathbb P_{\rho}$-stationary sequence  $(X_{k})_{k\geq 0}$. Our results are based on the fact  that the general conditions of multiple  mixing and anticlustering used in extreme value theory of stationary processes (see \cite{2}, \cite{7}) are valid for affine stochastic recursions, under condition (c-e). We observe that, in the context of  Lipschitz functions, the above mixing property is a consequence of the spectral gap properties studied below ; it turns out that the use of advanced point process theory  allows us to extend this mixing property to the  setting of compactly supported  continuous functions considered in\cite{2}. Using these basic results, we give proofs of  a few extreme value properties, following closely \cite{2}.  
 
 \noindent Then, our framework allow us to develop extreme value theory for a large class of natural examples in collective risk theory, including the so-called GARCH process as a very special case (see \cite{9}).  In this context, some of our results have natural interpretations as asymptotics of ruin probabilities or ruin times \cite{6}.
 
 \noindent In order to sketch our results, we recall that Fr\'echet's
 law $\Phi_{\alpha}^{a}$ with positive parameters $\alpha, a$ is the max-stable probability on $\mathbb R_{+}$ defined  by the distribution function $\Phi_{\alpha}^{a}([0,t])=\hbox{\rm exp}(-a t^{-\alpha})$.  Also, we consider the associated stochastic linear recursion $Y_{n}=A_{n} Y_{n-1}$,  we denote by $Q$ the corresponding Markov kernel on $V\setminus \{0\}$ and by $\mathbb Q=\mu^{\otimes \mathbb N}$ the product measure on $\Omega=G^{\mathbb N}$ ; we write $S_{n}=A_{n}\cdots A_{1}$ for the product of random matrices $A_{k} (1\leq k\leq n)$. Extending previous work  of H. Kesten (see \cite{20}), a basic result proved in \cite{13} under condition (c-e) is that for some $\alpha>0$, the probability $\rho$ is $\alpha$-homogeneous at infinity, hence $\rho$ has an asymptotic tail measure $\Lambda\neq 0$ which is a $\alpha$-homogeneous  $Q$-invariant Radon measure on $V\setminus \{0\}$. The multivariate regular variation of $\rho$ is a direct consequence of this fact. Also, it follows that, if $U_{t} \subset V$ is the closed  ball of radius  $t>0$ centered at $0\in V$ and $U'_{t}=V\setminus U_{t}$, then we have $\Lambda (U'_{t})=\alpha^{-1} c t^{-\alpha}$ with $c>0$.
 In particular, $\Lambda(U'_{t})$ is finite and the projection of $\rho$ on $\mathbb R_{+}$, given by the norm map $v\rightarrow |v|$ has the same asymptotic tail as $\Phi_{\alpha}^{c}$. We write $\Lambda_{0}=c^{-1} \alpha \Lambda$ hence $\Lambda_{0}(U'_{1})=1$ and  we denote by $\Lambda_{1}$ the restriction of $\Lambda_{0}$ to $U'_{1} $. If $u_{n}>0$ satisfies $(u_{n})^{\alpha}=\alpha^{-1}cn$, it follows that the mean number of exceedances of $u_{n}$ by  $|X_{k}|$  $(1\leq k\leq n)$  converges to one. It will appear below that $u_{n}$ is an estimate of  $\sup \{|X_{k}|$ ; $1\leq k\leq n\}$ and that normalization by $u_{n}$ reduces extreme values for the process $X_{n}$ to excursions at infinity for the linear random walk $S_{n}(\omega)v$ and the measures $\Lambda_{0}, \Lambda_{1}$.

\noindent Then,  one of our main results is the convergence in law of the $u_{n}$-normalized maximum of the sequence $|X_{1}|, |X_{2}|, \dots, |X_{n}|$ towards Fr\'echet's law $\Phi_{\alpha}^{\theta}$ with $\theta \in ]0,1[$. A closely related point process result is the weak convergence  of the time exceedances process 

\centerline{$N_{n}^{t}=\displaystyle\mathop{\Sigma}_{k=1}^{n} \varepsilon_{n^{-1} k} 1_{\{|X_{k}| > u_{n}\}}$}

\noindent  towards a compound Poisson process with intensity $\theta$ and cluster probabilities depending on the occupation  measure  $\pi_{v}^{\omega}=\displaystyle\mathop{\Sigma}_{i=0}^{\infty} \varepsilon_{S_{i}(\omega)v}$ of the associated linear random walk $S_{i}(\omega)v$  and  on  $\Lambda_{1}$. The expression $\varepsilon_{x}$ denotes the Dirac mass at $x$. The significance of the relation $\theta<1$ is that, in our situation,  values of the sequence $(|X_{k}|)_{0\leq k\leq n}$  larger than $u_{n}$, appear in localized clusters with asymptotic expected cardinality $\theta^{-1}>1$. This reflects the local dependance of large values in the sequence $(X_{k})_{k\geq 0}$ and is in contrast with the well known situation of positive i.i.d. random variables with tail  also given by $\Phi_{\alpha}^{c}$, where the same convergences with $\theta=1$ is satisfied.  If Euclidean norm is replaced by another norm, the value of  the extremal index $\theta$  is changed but  the condition $\theta \in ]0,1[$ remains valid.
For affine stochastic recursions in dimension one, if $A_{n}, B_{n}$ are positive and condition (c-e) is satisfied, convergence to Fr\'echet's law and $\theta \in ]0,1[$ was proved in \cite{18}, using \cite{20}. We observe that  our result  is the natural multivariate extension of this fact.
Here, our proofs use the tools of point processes theory and a remarkable formula (see \cite{2}) for the Laplace functional of a cluster point process $C=\displaystyle\mathop{\Sigma}_{j\in \mathbb Z} \varepsilon_{Z_{j}}$ on $V\setminus \{0\}$,  which describes in small time the large values of $(X_{n})_{n\geq 0}$.  As a consequence of Fr\'echet's law and in the spirit of \cite{28}, we obtain a logarithm law for affine random walk.

\noindent To go further, we consider the linear random walk $S_{n}(\omega) v$ on $V\setminus \{0\}$, we observe that condition (c-e) implies $\displaystyle\mathop{\lim}_{n\rightarrow \infty} S_{n}(\omega) v=0$, $\mathbb Q-$a.e. and we denote by $\mathbb Q_{\Lambda_{0}}$ (\hbox{\rm resp.} $\mathbb Q_{\Lambda_{1}})$ the Markov measure on $(V\setminus \{0\})^{\mathbb Z_{+}}$ defined by the kernel $Q$ and the initial  measure $\Lambda_{0}$ (\hbox{\rm resp.} $\Lambda_{1})$. We show below the weak convergence to a limit process $N$  of the sequence of space-time exceedances processes 

\centerline{$N_{n}=\displaystyle\mathop{\Sigma}_{i=1}^{n} \varepsilon_{(n^{-1}i, u^{-1}_{n}X_{i})}$}

\noindent on $[0,1] \times (V\setminus \{0\})$. In restriction to $[0,1]\times U'_{\delta}$, with $\delta>0$,    $N$ can be expressed in terms of $C$ and of a Poisson component on $[0,1]$  with intensity $\theta \delta^{-\alpha}$ ; the expression of the Laplace functional of $C$ involves the  occupation measure $\pi_{v}^{\omega}$  and  $\mathbb Q_{\Lambda_{1}}$. Using the framework and the results of (\cite{2}, \cite{3}, \cite{7}), we describe a few consequences of this convergence. In particular we consider also, as in (\cite{7}, \cite{8}), the convergence of the normalized partial sums $\displaystyle\mathop{\Sigma}_{i=1} ^{n} X_{i}$  towards  stable laws, if $0<\alpha<2$,  in the framework of extreme value theory.  Also, as observed in \cite{7}, this convergence is closely connected to the convergence of the sequence of space exceedances point processes on $V$
$$N_{n}^{s}=\displaystyle\mathop{\Sigma}_{i=1}^{n} \varepsilon_{u_{n}^{-1}X_{i}},$$
towards a certain infinitely divisible point process $N^{s}$.   Another consequence of this convergence is the description of the asymptotics of the normalized hitting time of any dilated Borel subset of $U'_{1}$ with positive $\Lambda$-measure,  and negligible boundary : weak convergence to a non trivial exponential law is valid. We observe that convergence of normalized hitting times to an exponential law is a well studied property in the context of the "shrinking target" problem. Here we are concerned with a special "random dynamical system" (see \cite{32}) associated to the affine action on $\mathbb R^{d}$ and the hitting times of shrinking neighbourhoods of the sphere at infinity of $\mathbb R^{d}$ considered  as a  weakly attractive target. The point process approach and the use of spectral gap properties allow us to deal with this geometrically  more complex  situation which involves a non trivial extremal index. 

\noindent In these studies we follow closely the approaches previously developed in (\cite{2},  \cite{7}) in the context of extreme value theory for general stationary processes, in particular we make full use of the concepts of tail and cluster processes introduced in \cite{2}.  This allow us to recover, in a natural setting, the  characteristic functions of the above $\alpha$-stable laws, as described  in   \cite{15} if $d=1$ and in  \cite{10} if $d>1$,  completing thereby the  results of (\cite{1}, \cite{2}, \cite{7}).  Furthermore,  in view of our results we observe that, with respect to the $Q$-invariant  measure $\Lambda_{0}$, the potential theory of the associated linear random walk   enters   essentially in the description of the extreme value asymptotics for the affine random walk $X_{n}$, through excursions,  occupation  measures and capacities.

\noindent  For self containment reasons we have developed anew a few arguments of (\cite{2}) in our situation. However we have modified somewhat the scheme of (\cite{2}, \cite{33}) for the construction of the tail process, introducing a shift invariant measure $\widehat{\mathbb Q}_{\Lambda_{0}}$ on $(V\setminus \{0\})^{\mathbb Z}$ which governs the excursions at infinity of the associated linear random walk $S_{n}(\omega) v$. This lead us to an essentially self-contained presentation for the extremal index and the point process asymptotics. It is the use of its  restriction $\widehat{\mathbb Q}_{\Lambda_{1}}$,  which allows one to express geometrically the point processes $N, N^{s}, N^{t}$, as in (\cite{2}, \cite{3}).

\noindent Here is the structure of our  paper. 

\noindent In section 2, we recall a basic result of \cite{13} and we define two processes of probabilistic significance. This allow us to connect the extreme values of the affine random walk to the excursions of its associated linear random walk, as in \cite{2}. In view of the results in section 4, it is essential to deal with $\mathbb Z$-indexed processes, hence to introduce the shift invariant measure $\widehat{\mathbb Q}_{\Lambda_{0}}$. 

\noindent In section 3 we develop the spectral gap properties for affine random walks. In particular we deduce the multiple mixing property which plays an essential role in section 4. 

\noindent In section 4 we describe the new results and give the corresponding proofs.

\noindent  Section 5 describe briefly new proofs of the convergence to stable laws,  in the context of point processes using the approach of \cite{2}, \cite{3}, \cite{7}. 

\noindent In section 6 we show that our basic hypothesis (c-e) is "generic" in the weak topology of measures on the affine group $H$.

\noindent We refer to (\cite{14}, \cite{16}) for surveys of the above results. After submission of this paper the authors became aware of the content of the book \cite{5}, where closely related results are proved by different methods.   We thank the referees for substancial and useful comments. We thank also J.P. Conze for suggesting the use of dynamical methods in the above context.

\section{ The tail process and the cluster process}

\vskip 3mm
 \subsection{\sl { Homogeneity at infinity of the stationary measure}}
 We recall condition (c-e) from \cite{13}, for the probability $\lambda$ on the affine group of $V$. 
 
 \noindent A semigroup $T$ of $GL(V)=G$ is said to satisfy i-p if 

a) $T$ has no invariant finite family of proper subspaces

b) $T$ contains an element with a dominant eigenvalue which is real and unique.

\noindent Condition i-p implies that the action of $T$ on the projective space of $V$ is proximal ; heuristically speaking this means that, $T$ contracts asymptotically two arbitrary given directions to a single one, hence the situation could be compared to a 1-dimensional one. Condition i-p  for $T$ is valid if and only if it is valid for the group which is the Zariski closure of $T$, (see \cite{25}).  We recall that a Zariski closed subset of an algebraic group like $G$ or $H$ is a subset defined as the set of zeros of a family of polynomials in the coordinates. The corresponding topology is much weaker that the usual locally compact topology . In particular the Zariski closure of a subsemigroup is a subgroup which is closed in the usual topology and which  has a finite number of connected components.   Hence condition i-p is valid if $T$ is Zariski dense in $G$ (see \cite{29}) ; also it is valid for $T$  if it is valid for $T^{-1}$. Below we will denote by $T$ the closed subsemigroup generated by supp$(\mu)$, the support of $\mu$.

\noindent For $g \in G$ we write $\gamma (g)=\hbox{\rm sup} (|g|, |g^{-1}|)$ and we assume $\int \hbox{\rm log} \gamma (g) d\mu (g) <\infty$.  For $s\geq 0$ we write $k(s)=\displaystyle\mathop{\lim}_{n\rightarrow \infty}( \int |g|^{s}d\mu^{n} (g))^{1/n}$ where $\mu^{n}$ denotes the $n^{th}$ convolution power of $\mu$ and we write $L(\mu)$ for the dominant Lyapunov exponent of the product $S_{n}(\omega)=A_{n}\cdots A_{1}$ of random matrices $A_{k} (1\leq k \leq n)$ i.e. $L(\mu)=\displaystyle\mathop{\lim}_{n\rightarrow \infty} \frac{1}{n} \int \hbox{\rm log} |g| d\mu^{n} (g)=k'(0)$. We denote by $r(g)$ the spectral radius of $g\in G$. We say that $T$ is non arithmetic if $r(T)$ contains two elements with irrational ratio. Condition (c-e) is the following :

1) supp$(\lambda)$ has no fixed point in $V$.

2) There exists $\alpha>0$ such that $k(\alpha)= \displaystyle\mathop{\lim}_{n\rightarrow \infty}  (\mathbb E |S_{n}|^{\alpha})^{1/n}=1$.

3) There exists $\varepsilon >0$ with $\mathbb E (|A|^{\alpha} \gamma^{\varepsilon} (A)+|B|^{\alpha+\varepsilon})<\infty$.

4) If $d>1$, $T$ satisfies i-p and if $d=1$, $T$ is non arithmetic.

\noindent The above conditions imply in particular that $L(\mu)<0$, $k(s)$ is analytic,  $k(s)<1$ for $s\in ]0,\alpha[$ and
  there exists a unique stationary probability $\rho$ for $\lambda$ acting by convolution  on $V$ ; the support of $\rho$  is unbounded. Property 1 guarantees that $\rho$ has no atom and says that the action of supp$(\lambda)$ is not conjugate to a linear action. Properties 2,3  imply that $T$ is unbounded and are responsible for the $\alpha$-homogeneity at infinity of $\rho$ described below ; if $k(s)$ is finite on $[0,\infty[$ and there exists $g\in T$ with $r(g)>1$, then Property 2 is satisfied. 
 Also  if $d>1$, condition i-p is basic for renewal theory of the linear random walk $S_{n} (\omega) v$ and it implies  that $T$ is non arithmetic. 
 
 \noindent In the appendix we will show that if $d>1$  condition (c-e) is open in the weak topology of probabilities on the affine group, defined by convergence of moments and of values on continuous compactly supported functions.
 
 \noindent Below, we use the decomposition of $V\setminus\{0\}=\mathbb S^{d-1} \times \mathbb R_{>0}$ in polar coordinates, where $\mathbb S^{d-1}$ is the unit sphere of $V$. We consider also the Radon measure $\ell^{\alpha}$ on $\mathbb R_{>0}$ $(\alpha>0)$ given by $\ell^{\alpha} (dt)=t^{-\alpha-1}dt$. We recall (see \cite{20}) that, if $(A_{n}, B_{n})_{n\in \mathbb N}$ is an i.i.d. sequence of $H$-valued random variables with law $\lambda$ and $L(\mu)<0$, then $\rho$ is the law of the $\mathbb P-$a.e. convergent series $X=\displaystyle\mathop{\Sigma}_{0}^{\infty} A_{1}\cdots A_{k} B_{k+1}$. The following is basic for our analysis.
  \vskip 3mm

\begin{thm} ( see \cite{13}, Theorem C) 

\noindent Assume that $\lambda$ satisfies condition (c-e). Then the operator  $P$ has a unique stationary probability $\rho$,   the support of $\rho$ is unbounded and  we have the following vague convergence on $V\setminus \{0\}$ :
$$\displaystyle\mathop{\lim}_{t\rightarrow 0_{+}} t^{-\alpha}(t.\rho)=\Lambda=c(\sigma^{\alpha} \otimes \ell^{\alpha})$$
where $c>0$ and  $\sigma^{\alpha}$ is a probability on $\mathbb S^{d-1}$. Furthermore $\Lambda$ is a $Q$-invariant Radon measure on $V\setminus \{0\}$ with unbounded support.
\end{thm}
\vskip 3mm
\noindent We observe that, for $d>1$, if supp$(\lambda)$ is compact and has no fixed point in $V$, $L(\mu)<0$, $T$ is    Zariski dense in $G$ and is unbounded, then condition (c-e) is satisfied. For $d=1$, if supp$(\lambda)$ is compact, the hypothesis of (\cite{18}, Theorem 1.1) is equivalent to condition (c-e).

\noindent The existence of $\Lambda$ stated in the theorem implies multivariate regular variation of $\rho$.  Since the convergence stated in the theorem is valid we say that $\rho$ is homogeneous at infinity ; below we will make essential use of this property. We note that the fact that supp$(\Lambda)$ is unbounded follows from condition 2 above, hence of the unboundness of $T$.

\noindent  Under condition (c-e), $\Lambda$ gives zero mass to any proper affine subspace and $\sigma^{\alpha}$ has positive dimension.  We observe that, if the sequence $(A_{n}, B_{n})_{n\in \mathbb N}$ is replaced by $(A_{n}, t B_{n})_{n\in \mathbb N}$ with $t\in \mathbb R^{*}$, then the asymptotic tail measure is replaced by $t.\Lambda$, in particular the constant $c$ is replaced by $|t|^{\alpha} c$. In subsection 2.2 below, in view of normalisation, it is natural to replace $\Lambda$ by $\Lambda_{0}=c^{-1}\alpha \Lambda$ ; this takes into account the magnitude of $B_{n}$ and implies $\Lambda_{0}(U'_{1})=1$.   Also, as shown in \cite{13}, the $Q$-invariant Radon measure $\Lambda_{0}$ is extremal or can be decomposed in two extremal measures. Hence, if the action of $T$ on $\mathbb S^{d-1}$ has a unique minimal subset, then $\Lambda_{0}$ is symmetric, supp$(\sigma^{\alpha})$ is equal to this minimal subset and the shift invariant Markov measure $\mathbb Q_{\Lambda_{0}}$ on $(V\setminus \{0\})^{\mathbb Z_{+}}$ is ergodic. Otherwise $\mathbb Q_{\Lambda_{0}}$ decomposes into at most two ergodic measures. Hence $\Lambda_{0}$ depends only of $\mu$, possibly up to one  positive coefficient.

\noindent The following is a consequence of vague convergence.

\begin{cor}
 Let $f$ be a bounded Borel function on $V\setminus \{0\}$  which has a $\Lambda$-negligible discontinuity set and such that \hbox{\rm supp}$(f)$ is bounded away from zero. Then we have 
 
 \centerline{$\displaystyle\mathop{\lim}_{t\rightarrow 0_{+}} t^{-\alpha} (t.\rho) (f)=\Lambda (f)$.}
 \end{cor}

\subsection{\sl  The tail process }

 We denote $\Omega=G^{\mathbb N}$, $\widehat{\Omega}=G^{\mathbb Z}$,  and we endow $\Omega$  with the product probability $\mathbb Q=\mu^{\otimes \mathbb N}$.

\noindent We define the $G$-valued cocycle $S_{n} (\omega)$ where $\omega=(A_{k})_{k\in \mathbb Z} \in \widehat{\Omega}$, $n\in \mathbb Z$ by :

$S_{n}(\widehat{\omega})=A_{n}\cdots A_{1}$ for $n>0$, $S_{n}(\widehat{\omega})=A_{n+1}^{-1}\cdots A_{0}^{-1}$ for $n<0$, $S_{0}(\widehat{\omega})=Id$

\noindent We consider also the random walk $S_{n} (\omega) v$ on $V\setminus \{0\}$, starting from $v\neq 0$, $\omega \in \Omega$ and we denote by $\mathbb Q_{\Lambda_{0}}$  the Markov measure on $(V\setminus \{0\})^{\mathbb Z_{+}}$ for the random walk $S_{n} (\omega)v$ with initial measure $\Lambda_{0}$. This measure is invariant under the  shift on $(V\setminus \{0\})^{\mathbb Z_{+}}$. We recall that the shift $s$ (resp. $\widehat{s})$ on the product space $\mathbb D^{\mathbb Z_{+}}$  (resp. $\mathbb D^{\mathbb Z})$ is the map defined on $\omega=(\omega_{k})_{k\in \mathbb Z_{+}}$ (resp. $\widehat{\omega}=(\widehat{\omega}_{k})_{k\in \mathbb Z})$ by $s(\omega)_{k} =\omega_{k+1}$ (resp. $\widehat{s} (\widehat{\omega})_{k}=\widehat{\omega}_{k+1}$. We recall below Proposition 4 of \cite{12} which extends the construction of the natural extension for a non invertible transformation (see also \cite{34}) ; this construction is standard for finite invariant measures.

\noindent Let $\Delta$ (resp. $\widehat{\Delta})$ be the shift on $\Omega$ (resp. $\widehat{\Omega})$, $\sigma$ (resp. $\widehat{\sigma})$ the map of $\Omega \times (V\setminus \{0\})=E$ into itself (resp. $\widehat{\Omega} \times (V\setminus \{0\})=\widehat{E})$ defined by

$\sigma(\omega,v)=(\Delta \omega, A_{0}(\omega) v)$ (resp. $(\widehat{\sigma} (\widehat{\omega},v)=(\widehat{\Delta} \widehat{\omega}, A_{0}(\widehat{\omega})v)$.

\noindent We denote  by $\tau$ (resp. $\widehat{\tau})$ the shift on $(V\setminus \{0\})^{\mathbb Z_{+}}$ (resp. $(V\setminus \{0\})^{\mathbb Z})$ and  we note that $\mathbb Q\otimes \Lambda_{0}$ is $\sigma$-invariant we observe that the  map $(\omega, v)\rightarrow (S_{k}(\omega) v)_{k\in \mathbb {Z}_{+}}$ defines the dynamical system $((V\setminus \{0\})^{\mathbb {Z}_{+}}, \tau, \mathbb Q_{\Lambda_{0}})$ as a measure preserving factor of $(E, \sigma, \mathbb Q \otimes \Lambda_{0})$. The construction of a natural extension has been detailed in \cite{12} and we state the result.

\begin{prop} (see \cite{12})

\noindent There exists a unique $\widehat{\sigma}$-invariant measure $\widehat{\mathbb Q \otimes \Lambda}$ on $\widehat{E}$ with projection $\mathbb Q \otimes \Lambda$ en $E$.
\end{prop}
Since the map defined by $p(\omega,v)=(S_{k}(\omega)v)_{k\in \mathbb Z_{+}}$ (resp. $\widehat{p} (\widehat{\omega}, v)=(S_{k}(\widehat{\omega}) v)_{k\in \mathbb Z})$ commutes with $\sigma, \tau$ (resp. $\widehat{\sigma}, \widehat{\tau})$ and $\widehat{\mathbb Q \otimes \Lambda_{0}}$ is $\widehat{\sigma}$-invariant with projection $\mathbb Q \otimes \Lambda_{0}$ on $\Omega \times (V\setminus \{0\})$ it follows, that the push-forward measure $\widehat{p} (\widehat{\mathbb Q \otimes \Lambda_{0}})= \widehat{\mathbb Q}_{\Lambda_{0}}$ on $(V\setminus \{0\})^{\mathbb Z}$, is $\widehat{\tau}$-invariant and has projection $\mathbb Q_{\Lambda_{0}}=p (\mathbb Q \otimes \Lambda_{0})$ on $(V\setminus \{0\})^{\mathbb Z_{+}}$. The definition of $\widehat{\mathbb Q}_{\Lambda_{0}}$ is used below in the construction of the tail process (see \cite{2}) of the affine random walk $(X_{k})_{k\in \mathbb Z}$. 

\noindent We define the probability $\widehat{\mathbb Q} _{\Lambda_{1}}$ by $\widehat{\mathbb Q}_{\Lambda_{1}}=  (1_{U'_{1}} \circ \pi) \widehat{\mathbb Q}_{\Lambda_{0}}$ where $\pi$ denotes the projection of $(V\setminus \{0\})^{\mathbb Z}$ on $V\setminus \{0\}$.
 The  restriction of $\Lambda_{0}$  to $U'_{1}$ is denoted $\Lambda_{1}$, hence $\Lambda_{1} (U'_{t})=t^{-\alpha}$ if $t>1$ and we denote by $\mathbb Q_{\Lambda_{1}}$ the Markov measure defined by $Q$ and the initial measure $\Lambda_{1}$.   
 We note that the probability $\mathbb Q_{\Lambda_{1}}$ (resp $\widehat{\mathbb Q}_{\Lambda_{1}})$ extends to $V^{\mathbb Z_{+}}$ (resp $V^{\mathbb Z})$  and its extension will be still denoted $\mathbb Q_{\Lambda_{1}}$ (resp $\widehat{\mathbb Q}_{\Lambda_{1}})$.

\noindent In order to illustrate the above  contruction, we take $d=1$, hence $V\setminus \{0\}=\mathbb R^{*}$, $\Lambda=c|x|^{-\alpha-1} dx$ where $\alpha>0$ and $\mu$ satisfies $\int \hbox{\rm log} |a| d\mu (a)<0$. Since $\int |a|^{\alpha} d\mu (a)=1$, we can denote by $\mu_{\alpha}$ the new probability defined by  $d\mu_{\alpha} (a)=|a|^{-\alpha} d\mu^{*} (a)$ and $\mu^{*}$ is the push forward of $\mu$ by the map $a\rightarrow a^{-1}$ ; it follows $\int \hbox{\rm log} |a| d\mu_{\alpha}(a)<0$. We denote by $Q^{*}$ the Markov kernel defined by convolution with $\mu_{\alpha}$ on $\mathbb R^{*}$, hence $Q^{*}\Lambda=\Lambda$. Then it is easy to verify that $\widehat{\mathbb Q\otimes \Lambda_{0}}=\mathbb Q_{\alpha}^{*}\otimes \Lambda_{0} \otimes \mathbb Q$ where $\mathbb Q_{\alpha}^{*}=\mu_{\alpha}^{\otimes (-\mathbb N)}$ and $\widehat{E}$ is identifixed with $G^{-\mathbb N}\times V\times G^{\mathbb N}$. Also we denote by $\mathbb Q_{v}^{*} \otimes \varepsilon_{v} (\hbox{\rm resp.}\varepsilon_{v}\otimes \mathbb Q_{v})$ the Markov probability on $(V\setminus\{0\})^{\otimes(-\mathbb Z_{+})}(\hbox{\rm resp.\ }(V\setminus \{0\})^{\otimes Z_{+}})$ associated to $Q^{*}$ (resp. $Q$) and the inititial  measure $\varepsilon_{v}$. Then we have 
 $\widehat{\mathbb Q}_{\Lambda_{1}}=\int \mathbb Q_{v}^{*}\otimes \varepsilon_{v} \otimes \mathbb Q_{v} d\Lambda_{1} (v)$. If $d>1$, such formulae remain valid, with $Q^{*}$ equal to the adjoint operator to $Q$ in $\mathbb L^{2}(\Lambda)$.

\noindent We consider the probability $\rho$,  the shift-invariant Markov measure $\mathbb P_{\rho}$ (resp $\widehat{\mathbb P}_{\rho}$) on $V^{\mathbb Z_{+}}$ (resp $V^{\mathbb Z}$), where $\rho$ is the law of $X_{0}$ and $\mathbb P_{\rho}$ is the projection of $\widehat{\mathbb P}_{\rho}$ on $V^{\mathbb Z_{+}}$. Since $\rho(\{0\})=0$, we can replace $V$ by $V\setminus \{0\}$ when working under $\mathbb P_{\rho}$.  For $0< j \leq i$ we  write $S_{j}^{i}=A_{i}\cdots A_{j}$ and $S_{i}^{i+1}=I$. Expectation with respect to $\mathbb P$ or $\mathbb Q$,  will be simply denoted by the symbol $\mathbb E$. If expectation is taken with respect to a Markov measure with initial measure $\nu$, we will  write $\mathbb E_{\nu}$. For a family $Y_{j}(j\in \mathbb Z)$ of $V$-valued random variables and $k, \ell$ in $\mathbb Z$ $\cup\{-\infty, \infty\}$, we denote $M_{k}^{\ell} (Y)= \sup \{|Y_{j}|\ ;\ k\leq j \leq \ell\}$. We observe that, if $t>0$, condition (c-e) implies $\rho\{ |x|>t\}>0$, hence as in \cite{2}, we can consider the new process $(Y^t_{i})_{i\in \mathbb Z}$ deduced from $t^{-1}(X_{i})_{i\in \mathbb Z}$ by conditioning on the set $\{|X_{0}|>t\}$, for $t$ large, under $\widehat{\mathbb P}_{\rho}$. We recall (see \cite{30}) that a sequence of point processes on a separable locally compact space $E$ is said to converge weakly to another point process if there is weak convergence of the corresponding finite dimensional distributions.  The following is a detailed form in our case of the general result  for multivariate jointly regularly varying stationary processes in \cite{2}. The tail process appears here to be closely related to the stationary process (with infinite measure) $\widehat{\mathbb Q}_{\Lambda_{0}}$, which plays therefore an important geometric role.

\vskip 3mm
\begin{prop}  
a) The family of  finite dimensional distributions of the point process $(Y_{i}^{t})_{i\in \mathbb Z}$ converges weakly  $(t\rightarrow \infty)$ to those of the point process $(Y_{i})_{i\in \mathbb Z}$ on $V$ given by $Y_{i}=S_{i} Y_{0}$ where  $(Y_{i})_{i\in \mathbb Z}$ has law $\widehat{\mathbb Q}_{\Lambda_{1}}$.

 b) We have   $\mathbb Q_{\Lambda_{1}} \{M_{1}^{\infty}(Y)\leq 1\}=\widehat{\mathbb Q}_{\Lambda_{1}} \{M^{-1}_{-\infty} (Y)\leq 1\}$,

 $\displaystyle\mathop{\lim}_{n\rightarrow \infty}\ \displaystyle\mathop{\lim}_{t\rightarrow \infty} \mathbb P_{\rho}\{\displaystyle\mathop{\sup}_{1\leq k\leq n} |X_{k}| \leq t/ |X_{0}|> t\} =\mathbb Q_{\Lambda_{1}} \{M_{1}^{\infty} (Y) \leq 1\} :=\theta$

c) In particular $\theta \in ]0,1[$
 \end{prop}
   
\begin{proof} 
 a) We observe that, since for any $i\geq 0$, 
$X_{i}=S_{i} X_{0}+\displaystyle\mathop{\Sigma}_{j=1}^{i} S^{j+1}_{i} B_{j}$ 
and $\displaystyle\mathop{\lim}_{t\rightarrow \infty} \frac{1}{t} \displaystyle\mathop{\Sigma}_{j=1}^{i}S_{i}^{j+1}$ $B_{j}=0$, $\mathbb P_{\rho}-$a.e,  the random vectors  $(t^{-1}X_{i})_{0\leq i\leq p+q}$ and $(t^{-1} S_{i} X_{0})_{0\leq i\leq p+q}$ have the same asymptotic behaviour in $\mathbb P_{\rho}$-law, conditionally on $|X_{0}|> t$. Also by stationarity of $\widehat{\mathbb P}_{\rho}$, for $f$ continuous and bounded on $V^{p+q+1}$ we have 

\noindent $\mathbb E_{\rho}\{f(t^{-1} X_{-q}, \cdots, t^{-1} X_{p}) /|X_{0}|>t\}= \mathbb E_{\rho} \{f (t^{-1} X_{0}, t^{-1} X_{1}, \cdots, t^{-1} X_{p+q})\  / \   |X_{q}|>t\}$.

\noindent From above,   using Corollary 2.2, $\Lambda\{|x|=1\}=0$ and the formula $\Lambda \{|x|>1\}=\alpha^{-1} c$, we see that the right hand side converges to :

$c^{-1} \alpha  \int \mathbb E \{f(x, S_{1}x, \cdots, S_{p+q}x) 1_{\{|S_{q}x|>1\}}\}d\Lambda (x)=\widehat{\mathbb Q}_{\Lambda_{0}}(f 1_{\{S^{-1}_{q} (U'_{1})\}})$

\noindent We observe that Corollary 2.2 can be used here for fixed $\omega\in \Omega$ since the condition $|X_{q}|>t$ reduces the transformed expression under $\mathbb E_{\rho}$ to a bounded function supported in $S_{q}^{-1} (U'_{1})$.

\noindent Hence, using stationarity of  $\widehat{\mathbb Q}_{\Lambda}$, we get the weak convergence of the process $(Y_{i}^{t})_{i\in \mathbb Z}$ to $(Y_{i})_{i\in \mathbb Z}$ as stated in a).

b) In view of a) and Corollary 2.2, since the discontinuity sets of the functions $1_{]0,1]} (M_{1}^{n} (Y))$ and $1_{[1,\infty[} (Y_{0})$ on $V^{n}$ are $\mathbb Q_{\Lambda_{1}}
$-negligible, we have $\displaystyle\mathop{\lim}_{t\rightarrow \infty} \mathbb P_{\rho} \{\displaystyle\mathop{\sup}_{1\leq k\leq n} t^{-1} |X_{k}| \leq 1 / t^{-1} |X_{0}|  >1\}=\mathbb Q_{\Lambda_{1}} \{M_{1}^{n} (Y)\leq 1\}$.

\noindent  Hence  $\theta=\displaystyle\mathop{\lim}_{n\rightarrow \infty} \displaystyle\mathop{\lim}_{t\rightarrow \infty} \mathbb P_{\rho}\{\displaystyle\mathop{\sup}_{1\leq k\leq n} |X_{k}|\leq t / X_{0}>t\}=\mathbb Q_{\Lambda_{1}} \{\displaystyle\mathop{\sup}_{k\geq 1} |Y_{k}|\leq 1\}$.

\noindent  We write $\widehat{\mathbb Q}_{\Lambda_{1}} \{M^{-1}_{-\infty} (Y)\leq 1\}=1-\widehat{\mathbb Q}_{\Lambda_{1}} \{M^{-1}_{-\infty}(Y) >1\}$ 
 and we define the random time $T$ by $T= \hbox{\rm inf} \{k\geq 1$ ; $|Y_{-k}|>1\}$ if there exists $k\geq 1$ with $|Y_{-k}|>1$ ; if such a $k$ does not exist we take $T=\infty$. We have by definition of $T$ :

$\widehat{\mathbb Q}_{\Lambda_{1}} \{M_{-\infty}^{-1} (Y)>1\}= \displaystyle\mathop{\Sigma}_{k= 1}^{\infty} \widehat{\mathbb Q}_{\Lambda_{1}} \{T=k\}$, 

 $\widehat{\mathbb Q}_{\Lambda_{1}} \{T=k\}=\widehat{\mathbb Q}_{\Lambda_{1}} \{|Y_{-1}|\leq 1$, $|Y_{-2}|\leq 1, \cdots , |Y_{-k+1}|\leq 1 \ ;\ |Y_{-k}|>1\}$, 

 \noindent Using $\widehat{\tau}$-invariance of $\widehat{\mathbb Q}_{\Lambda}$, the definition of $\widehat{\mathbb Q}_{\Lambda_{1}}$ and a), we get
 
\noindent $\widehat{\mathbb Q}_{\Lambda_{1}} \{T=k\}=\mathbb Q_{\Lambda_{1}}\{|Y_{1}|\leq1,\cdots, |Y_{k-1}|\leq 1\ ;\ |Y_{k}|>1\}$,

\noindent $\widehat{\mathbb Q}_{\Lambda_{1}}\{M_{-\infty}^{-1} (Y)>1\}=\displaystyle\mathop{\Sigma}_{k=1}^{\infty} \mathbb Q_{\Lambda_{1}} \{|Y_{1}|\leq 1,\cdots, |Y_{k-1} |\leq 1\ ;\ |Y_{k}|>1\}=\mathbb Q_{\Lambda_{1}}\{M_{1}^{\infty}(Y)>1\}$.

\noindent The formula $\widehat{\mathbb Q}_{\Lambda_{1}} \{M_{-\infty}^{-1} (Y)\leq 1\}=\mathbb Q_{\Lambda_{1}} \{ M_{1}^{\infty}(Y)\leq 1\}$ follows. 

c)The formula $\theta=\mathbb Q_{\Lambda_{1}} \{ M_{1}^{\infty} (Y) \leq 1\}$ and the form of $Y_{i} (i\geq 0)$ given in a) imply $\theta=\mathbb E (\int 1_{\{\displaystyle\mathop{\sup}_{i\geq 1} |S_{i} x|\leq 1\}}  d\Lambda_{1}(x))\leq 1$. The condition $\theta=1$ would imply for any $i\geq 1 : |S_{i} x|\leq 1$, $\mathbb Q \otimes \Lambda_{1}-$a.e., hence $\{0\} \not \subset \hbox{\rm supp} (S_{i} \Lambda_{1})\subset U_{1}$. This  would contradict the fact that $\hbox{\rm supp} (\Lambda_{1})$ is unbounded, hence we have $\theta<1$. The condition $\theta=0$ implies $\mathbb Q_{\Lambda_{1}}\{M_{1}^{\infty} (Y) \leq 1\}=0$ and, since $Y_{j}=S_{j}Y_{0}$, we have $|S_{j}y|>1$ $\mathbb Q_{\Lambda_{1}}-$a.e for some $j>1$.  Since $\Lambda$ is $Q$-invariant, the corresponding hitting law of supp$(\Lambda_{1})$ is absolutely continuous with respect to $\Lambda_{1}$. Then, using Markov property we see that $S_{k}(\omega) y$ returns to supp$(\Lambda_{1})$ infinitely often $\mathbb Q_{\Lambda_{1}}$-a.e. Since condition (c-e)implies $\displaystyle\mathop{\lim}_{j\rightarrow \infty} |S_{j} y|=0$, $\mathbb Q-$a.e. for any $y\neq 0$, this is a contradiction.
\end{proof}

\vskip 3mm
\subsection{\sl  Anticlustering property}
We will show in section 2.4 that, asymptotically, the set of large values of $X_{k}$  consists of a sequence of localized elementary clusters. An  important sufficient condition for localization (see \cite{7}) is proved in Proposition 2.5 below and  will allow us to show the existence of a cluster process, following \cite{2}. It is called anticlustering and is used in section 4 to decompose the set of values of $X_{k} (1\leq k\leq n)$ into successive quasi-independent blocks. For $k\leq \ell$ in $\mathbb Z$ we write 

\centerline{$M_{k}^{\ell}=\displaystyle\mathop{\sup}_{k\leq i\leq \ell} |X_{i}|\ \ ,\ \ R_{k}^{\ell}=\displaystyle\mathop{\Sigma}_{i=k}^{\ell} \widehat{\mathbb P}_{\rho}\{|X_{i}|>u_{n}/|X_{0}|>u_{n}\},$}

\noindent where $u_{n}=(\alpha^{-1} c n)^{1/\alpha}$.
  For $k>0$, we write also $M_{k}=M_{1}^{k}$.

\noindent   We observe that $M_{k}^{\ell} \leq \displaystyle\mathop{\Sigma}_{k}^{\ell} |X_{i}|$. Let $r_{n}$ be any sequence of integers with $r_{n}=o (n)$, $\displaystyle\mathop{\lim}_{n\rightarrow \infty} r_{n}=\infty$. Then we have $\mathbb P_{\rho} \{M_{1}^{r_{n}}> u_{n}\}\leq r_{n} \mathbb P_{\rho} \{|X_{0}|>u_{n}\}$, hence the homogeneity of $\rho$ at infinity gives $\displaystyle\mathop{\lim}_{n\rightarrow \infty} \mathbb P_{\rho} \{ M_{1}^{r_{n}}> u_{n}\}=0$. The condition $r_{n}=o(n)$ allows us to localize the influence of one large value of $X_{k} (1\leq k\leq n)$.  It follows that the event $\{M_{1}^{r_{n}}>u_{n}\}$ can be considered as ''rare''.   The homogeneity at infinity of $\rho$ and the arbitrariness of $r_{n}$ allow us to restrict the study to the sequence $u_{n}$ instead of $t u_{n} (t>0)$. 

\noindent The following is based on the homogeneity at infinity of $\rho$, the inequality  $0<k(s)<1$ if $0<s<\alpha$.
\vskip 3mm

\begin{prop}
Assume $r_{n} \leq [n^{s}]$ with $0<s<1$, $\displaystyle\mathop{\lim}_{n\rightarrow \infty} r_{n}=\infty$. Then $\displaystyle\mathop{\lim}_{m\rightarrow \infty} \displaystyle\mathop{\lim}_{n\rightarrow \infty} R_{m}^{r_{n}}=0$. In particular $\displaystyle\mathop{\lim}_{m\rightarrow \infty} \displaystyle\mathop{\lim}_{n\rightarrow \infty} \widehat{\mathbb P}_{\rho}\{\sup (M_{-r_{n}}^{-m},\ M_{m}^{r_{n}})>u_{n}/|X_{0}|>u_{n}\}=0$, hence the random walk $X_{n}$ satisfies anticlustering. Furthermore we have $\widehat{\mathbb Q}_{\Lambda_{1}}\{\displaystyle\mathop{\lim}_{|t|\rightarrow \infty} |Y_{t}|=0\}=1$.
\end{prop}

\begin{proof}
 We observe that :

$\widehat{\mathbb P}_{\rho}\{M_{m}^{r_{n}}>u_{n}/ |X_{0}|>u_{n}\}\leq R_{m}^{r_{n}}$, \ \ 
$\widehat{\mathbb P}_{\rho}\{M_{-r_{n}}^{-m}>u_{n}/|X_{0}|>u_{n}\}\leq R_{-r_{n}}^{-m}=R_{m}^{r_{n}}$

\noindent where we have used stationarity of $X_{k}$ in the last equality. Hence it suffices to show $\displaystyle\mathop{\lim}_{m\rightarrow \infty} \displaystyle\mathop{\lim}_{n\rightarrow \infty} R_{m}^{r_{n}}=0$. For $i\geq 0$ we have $X_{i}=S_{i} X_{0}+\displaystyle\mathop{\Sigma}_{j=1}^{i} S_{i}^{j+1}B_{j}$ where $S_{i}$, $X_{0}$ are independent, as well as $X_{0}$, $\displaystyle\mathop{\Sigma}_{j=1}^{i} |S_{i}^{j+1} B_{j}|$. We write 
$I_{n}^{i}=\widehat{\mathbb P}_{\rho}\{|X_{i}|>u_{n}/ |X_{0}|> u_{n}\}$,

 $J_{n}^{i}=\widehat{\mathbb P}_{\rho}\{|S_{i} X_{0}|>2^{-1} u_{n}/|X_{0}|>u_{n}\}$,  
$K_{n}^{i}=\widehat{\mathbb P}_{\rho}\{\displaystyle\mathop{\Sigma}_{j=1}^{i} |S_{i}^{j+1} B_{j}|>2^{-1}u_{n}/|X_{0}|>u_{n}\}$, 

\noindent hence  $R_{m}^{r_{n}}=\displaystyle\mathop{\Sigma}_{i=m}^{r_{n}} I_{n}^{i}\leq \displaystyle\mathop{\Sigma}_{i=m}^{r_{n}} J_{n}^{i}+\displaystyle\mathop{\Sigma}_{i=m}^{r_{n}}K_{n}^{i}$. 

\noindent We are going to show $\displaystyle\mathop{\lim}_{m\rightarrow \infty} \displaystyle\mathop{\lim}_{n\rightarrow \infty} \displaystyle\mathop{\Sigma}_{i=m}^{r_{n}} J_{n}^{i}=\displaystyle\mathop{\lim}_{m\rightarrow \infty} \displaystyle\mathop{\lim}_{n\rightarrow \infty} \displaystyle\mathop{\Sigma}_{i=m}^{r_{n}} K_{n}^{i}=0$.

\noindent We apply Chebyshev's inequality to the $\chi$-moments of $X_{n}$ with $\chi\in ]0,\alpha[$. We have : 

 $J_{n}^{i}\leq (2 u_{n}^{-1})^{\chi} \mathbb E_{\rho}(|S_{i} X_{0}|^{\chi}/|X_{0}|>u_{n}\}\leq (2 u_{n}^{-1})^{\chi} \mathbb E(|S_{i}|^{\chi}) \mathbb E_{\rho} (|X_{0}|^{\chi}/|X_{0}|>u_{n})$, 
 
 \noindent where independance of $S_{i}$ and $X_{0}$ has been used in the last formula. Since the law of $X_{0}$ is $\alpha$-homogeneous at infinity we have :

$\displaystyle\mathop{\lim}_{x\rightarrow \infty} x^{-\chi} \mathbb E_{\rho} (|X_{0}|^{\chi}/|X_{0}|>x)=\alpha(\alpha-\chi)^{-1},\ \  \displaystyle\mathop{\lim\sup}_{n\rightarrow \infty} J_{n}^{i}\leq 2^{\chi} (\mathbb E |S_{i}|^{\chi})\alpha(\alpha-\chi)^{-1}$.

\noindent It follows 
$\displaystyle\mathop{\lim\sup}_{n\rightarrow \infty} (\displaystyle\mathop{\Sigma}_{i=m}^{r_{n}} J_{n}^{i})\leq 2^{\chi} \alpha(\alpha-\chi)^{-1}\mathbb E (\displaystyle\mathop{\Sigma}_{i=m}^{\infty} |S_{i}|^{\chi})$, hence $\displaystyle\mathop{\lim}_{m\rightarrow \infty} \displaystyle\mathop{\lim}_{n\rightarrow \infty} \displaystyle\mathop{\Sigma}_{i=m}^{r_{n}} J_{n}^{i}=0$,

\noindent since $\displaystyle\mathop{\lim}_{i\rightarrow \infty} (\mathbb E|S_{i}|^{\chi})^{1/i}=k(\chi)<1$.

\noindent Also, using independence of $X_{0}$ and $\displaystyle\mathop{\Sigma}_{j=1}^{i} |S_{i}^{j+1}B_{j}|$ :

$K_{n}^{i}= \mathbb P_{\rho}\{\displaystyle\mathop{\Sigma}_{j=1}^{i} |S_{i}^{j+1} B_{j}|>2^{-1}u_{n}\}
\leq (2u_{n}^{-1})^{\chi} \mathbb E (\displaystyle\mathop{\Sigma}_{j=1}^{i} |S_{i}^{j+1} B_{j}|)^{\chi}\leq (2 u_{n}
^{-1})^{\chi} \mathbb E (\displaystyle\mathop{\Sigma}_{j=1}^{\infty} |S_{j-1} B_{j}|)^{\chi}$.

\noindent From above using independance of $S_{j-1}$ and $B_{j}$we know that $R_{0}=\displaystyle\mathop{\Sigma}_{1}^{\infty} |S_{j-1} B_{j}|$ has finite $\chi$-moment if $\chi\leq 1$. Then by Chebyshev's inequality : 
$\displaystyle\mathop{\Sigma}_{i=m}^{r_{n}} K_{n}^{i} \leq (2 u_{n}^{-1})^{\chi} r_{n}\mathbb E (R_{0}^{\chi})$.

\noindent Also if $\chi \in [1,\alpha[$ we can use Minkowski's inequality and independance in order to show the finiteness or $\mathbb E(| R_{0}|^{\chi})$.   
Since $0<s<1$, we can choose $\chi \in ]0,\alpha[$ such that $\alpha^{-1}\chi>s$, hence $\displaystyle\mathop{\lim}_{n\rightarrow \infty} r_{n} u_{n}^{-\chi}=0$. Then, for any fixed $m : \displaystyle\mathop{\lim}_{n\rightarrow \infty} \displaystyle\mathop{\Sigma}_{i=m}^{r_{n}} K_{n}^{i}=0$. 

\noindent From \cite{13} we know that, since $0<\chi<\alpha$,  we have $k(\chi)< 1$, hence $\mathbb E(|S_{i}|^{\chi})$ decreases exponentially fast to zero ; then the series $\mathbb E(\displaystyle\mathop{\Sigma}_{i=1}^{\infty} |S_{i}|^{\chi})$ converges and 

$\displaystyle\mathop{\lim}_{m\rightarrow \infty} \mathbb E (\displaystyle\mathop{\Sigma}_{i=m}^{\infty} |S_{i}|^{\chi})=0$, $\displaystyle\mathop{\lim}_{m\rightarrow \infty} \displaystyle\mathop{\lim}_{n\rightarrow \infty} (\displaystyle\mathop{\Sigma}_{i=m}^{r_{n}} J_{n}^{i})=0$.

\noindent Hence $\displaystyle\mathop{\lim}_{m\rightarrow \infty} \displaystyle\mathop{\lim}_{n\rightarrow \infty} \displaystyle\mathop{\Sigma}_{i=m}^{r_{n}} I_{n}^{i}=0$ and the  result follows.

\noindent The anticlustering property $\displaystyle\mathop{\lim}_{m\rightarrow \infty} \displaystyle\mathop{\lim}_{n\rightarrow \infty} \widehat{\mathbb P}_{\rho} \{\sup (M^{-m}_{-r_{n}}, M_{m}^{r_{n}})>u_{n}/|X_{0}|>u_{n}\}=0$ shows that, given $\varepsilon>0$, $u>0$ and using Proposition 2.4,  there exists $m\in \mathbb N$ such that for $r\geq m$ we have $\widehat{\mathbb Q}_{\Lambda_{1}}\{\displaystyle\mathop{\sup}_{m\leq |j|\leq r}|Y_{j}|\geq u\}\leq \varepsilon$. Hence $\widehat{\mathbb Q}_{\Lambda_{1}} \{\displaystyle\mathop{\lim}_{|j|\rightarrow \infty} |Y_{j}|=0\}=1$.

\end{proof}
\vskip 3mm

\subsection{\sl The cluster process }
\noindent In general, for a stationary $V$-valued point process with an associated tail process, the  properties of anticlustering  and  positivity of the extremal index $\theta$ for a sequence $r_{n}=o(n)$ with $\displaystyle\mathop{\lim}_{n\rightarrow \infty} r_{n}=\infty$,  imply the existence of the  cluster process (see \cite{2}). For self containment reasons we give in Proposition 2.6 below  a  proof of this fact, using  arguments of \cite{2} ;  this gives us also the convergence of $\theta_{n}$ defined by $\theta_{n}^{-1}=\mathbb E_{\rho} \{\displaystyle\mathop{\Sigma}_{1}^{r_{n}} 1_{[u_{n},\infty[}(|X_{k}|)/M_{r_{n}}>u_{n}\}$ to $\theta$ defined in Proposition 2.4. We note that the condition $\displaystyle\mathop{\Sigma}_{1}^{r_{n}} 1_{[u_{n},\infty[} (X_{k})>0$ implies $M_{r_{n}}>u_{n}$, hence we have $\theta_{n}^{-1}= r_{n}(\mathbb P_{\rho} \{|X_{0}|>u_{n}\}) (\mathbb P_{\rho}\{M_{r_{n}}>u_{n}\})^{-1}$. For later use we include also in the statement the formula of (\cite{2}, Theorem 4.3) giving the Laplace functional for the cluster process restricted to $U'_{1}$. We recall that the Laplace functional of a random measure $\eta$ on a locally  compact separable metric space $E$, where the space $M_{+}(E)$ of positive Radon measures on $E$ is  endowed with a probability $m$,  is given by 

\centerline{$\psi_{\eta}(f)= \int \hbox{\rm exp}(-\nu (f)) dm (\nu)$}

\noindent with $f$  continuous  non negative and $\hbox{\rm supp} (f)$ compact. We recall also that weak convergence of a sequence of point processes is equivalent to convergence of their Laplace functionals.

\noindent We denote by $r_{n}$ a  sequence as  above and we consider the sequence of point processes $C_{n}=\displaystyle\mathop{\Sigma}_{i=1}^{r_{n}} \varepsilon_{u_{n}^{-1} X_{i}}$, on $E=V\setminus \{0\}$  under $\mathbb P_{\rho}$ and conditionally on $M_{r_{n}}=M_{1}^{r_{n}}>u_{n}$. In particular $\theta_{n}^{-1}=\mathbb E_{\rho} (C_{n} (U'_{1}))$. Using the tail process  $(Y_{n})_{n\in \mathbb Z}$ defined in Proposition 2.4 above,  we show  that $C_{n}$ converges weakly to the point process $C$ ; $C$ is a basic quantity for the asymptotics of $X_{n}$ and is called the cluster process of $X_{n}$. As shown in Proposition 2.6 below, the law of $C$ can be expressed in terms of $\widehat{\mathbb Q}_{\Lambda_{1}}$ and  depends only of $\mu, \Lambda_{1}$. By definition, $C$  selects the large values of the process $X_{n}$ and describes the local multiplicity of large values in a typical cluster for the process $X_{n}$. We denote $\widehat{F}=(V\setminus \{0\})^{\mathbb Z}$, $\widehat{F}_{-}=\{v\in \widehat{F} ; \displaystyle\mathop{\sup}_{k\leq -1} |v_{k}|\leq 1\}$. Since $ \displaystyle\mathop{\lim}_{|i|\rightarrow \infty} Y_{i}=0$, Proposition 2.5 implies $\widehat{\mathbb Q}_{\Lambda_{1}} (\widehat{F}_{-})>0$, hence we can define the conditional probability $\overline{\mathbb Q}_{\Lambda_{1}}$ on $\widehat{F}_{-}$ by $\overline{\mathbb Q}_{\Lambda_{1}}=(\widehat{\mathbb Q}_{\Lambda_{1}} (\widehat{F}_{-}))^{-1} (1_{\widehat{F}_{-}} \widehat{\mathbb Q}_{\Lambda_{1}})$. Proposition 2.6 below implies that the law of $C$ is  $\overline{\mathbb Q}_{\Lambda_{1}}$, hence we can define a natural version of $C$ as $C= \displaystyle\mathop{\Sigma}_{j\in \mathbb Z} \varepsilon_{Z_{j}}$ where $Z_{j}(v)$ is defined as the $j$-projection of $v=(v_{k})_{k\in \mathbb Z} \in \widehat{F}_{-}$, where $\widehat{F}_{-}$ is endowed with the probability $\overline{\mathbb Q}_{\Lambda_{1}}$.

\noindent We denote by $\pi_{v}^{\omega}$ the   occupation measure of the random walk $S_{n}(\omega) v$ on $V\setminus \{0\}$, given by $\pi_{v}^{\omega}=\displaystyle\mathop{\Sigma}_{0}
^{\infty} \varepsilon_{S_{i}(\omega)v}$. For $v$ fixed, the mean measure of  $\pi_{v}^{\omega}$ is the potential measure $\displaystyle\mathop{\Sigma}_{0}^{\infty} Q^{i} (v,.)$ of the Markov kernel $Q$ ; if $L(\mu)<0$ the asymptotics $(|v|\rightarrow \infty)$ of this Radon measure are described in \cite{13}.  The formula below for the Laplace functional of $C$ involves the   occupation measure $\pi_{v}^{\omega}$ of the linear random walk $S_{n}(\omega) v$ and $\Lambda_{1}$ ; it plays an essential role below.

\noindent With the above notations we have the.

\vskip 3mm

\begin{prop}
 Under $\mathbb P_{\rho}$, the sequence of point processes  $C_{n}$  converges weakly  to a point process $C$. The law of the point process $C$ is equal to  the $\widehat{\mathbb Q}_{\Lambda_{1}}-$ law of the point process $\displaystyle\mathop{\Sigma}_{j\in \mathbb Z} \varepsilon_{S_{j}x}$ conditional on $\displaystyle\mathop{\sup}_{j\leq-1} |S_{j} x|\leq 1$. In particular we have for $C= \displaystyle\mathop{\Sigma}_{j\in \mathbb Z} \varepsilon_{Z_{j}}$ with $Z_{j}$ as above 

\centerline{$\overline{\mathbb Q}_{\Lambda_{1}} \{\displaystyle\mathop{\lim}_{|i|\rightarrow \infty} |Z_{i}|=0\}=1$,  $\overline{\mathbb Q}_{\Lambda_{1}} \{\displaystyle\mathop{\sup}_{i\geq 1} |Z_{i}|\geq 1\}=1$. }

Furthermore the sequence $\theta_{n}$ defined above converges to  the positive number $\theta$ and :

 $\displaystyle\mathop{\sup}_{n\in \mathbb N} \mathbb E_{\rho} \{(\displaystyle\mathop{\Sigma}_{1}^{r_{n}} 1_{[u_{n},\infty[} |X_{i}|)^{2} / M_{r_{n}}>u_{n}\}<\infty$,\ \  $\theta^{-1}=\overline{\mathbb E}_{\Lambda_{1}} (\displaystyle\mathop{\Sigma}_{j\in \mathbb Z} 1_{U'_{1}} (Z_{j}))<\infty$.

\noindent If \hbox{\rm supp}$(f) \subset U'_{1}$, the Laplace functional of $C$  on $f$ is given by

\centerline{$1-\theta^{-1} \mathbb E_{\Lambda_{1}} [(\hbox{\rm exp} f(v)-1) \hbox{\rm exp}(-\pi_{v}^{\omega} (f))]$.
}
\end{prop}

\begin{proof}  
  Let $f$ be a non negative and continuous  function on $V\setminus \{0\}$ which is compactly supported, hence $f(x)=0$ if $|x| \leq \delta$ with $\delta>0$.

\noindent We write for $k\leq \ell$ with $k, \ell \in \mathbb Z \cup \{\pm\infty\}$, $M_{k}^{\ell} (Y)=\displaystyle\mathop{\sup}_{k\leq j\leq \ell} |Y_{j}|$ with $Y_{j}= S_{j} Y_{0}$. For $k, \ell, f$ as above we write $C_{k}
^{\ell}=\hbox{\rm exp}(-\displaystyle\mathop{\Sigma}_{k}^{\ell} f (u_{n}^{-1}X_{j}))$, $C_{k}^{\ell} (Y)=\hbox{\rm exp}(-\displaystyle\mathop{\Sigma}_{k}^{\ell} f(Y_{j}))$ and we observe that $C_{k}^{\ell}\leq 1$. We fix $m>0$ and we take $n$ so large that the sequence $r_{n}$ of the above proposition satisfies $r_{n}>2 m+1$. When convenient we write $r_{n}=r$, hence $\mathbb E_{\rho} (C_{1}^{r} ; M_{1}^{r}>u_{n})=\displaystyle\mathop{\Sigma}_{1}^{r} \mathbb E_{\rho} (C_{1}^{r}; M_{1}^{j-1}\leq u_{n}<X_{j})$. We observe that, for $r-m \geq j>m+1$, we have $C_{1}^{r}=C_{j-m}^{j+m}$ except if $\sup (M_{1}^{j-m-1}, M_{j+m+1}^{r})> u_{n} \delta$. We are going to compare $\mathbb E_{\rho} (C_{1}^{r} ; M_{1}^{r}>u_{n})$ and $(r-2m) \widehat{\mathbb E}_{\rho} (C_{-m}^{m} ; M_{-m-1}^{-1} \leq u_{n}<|X_{0}|)$ using those $j's$ which satisfy $m+1<j\leq r-m$ and we denote by $\Delta_{n,m}$ their difference.

\noindent If we write

$\Delta_{n,m} (j)=\mathbb E_{\rho} (C_{1}^{r} ; M_{1}^{j-1}\leq u_{n}<|X_{j}|)-\mathbb E_{\rho} (C_{j-m}^{j+m} ; M_{j-m-1}^{j-1}\leq u_{n}<|X_{j}|)$,

\noindent then we have using stationarity  and $C_{k}^{\ell}\leq 1$ :

$|\Delta_{n,m}|\leq \displaystyle\mathop{\Sigma}_{m+1}^{r-m} |\Delta_{n,m} (j)|+2 m \ \mathbb P_{\rho}\{|X_{0}|>u_{n}\}$.

\noindent Using stationarity of $X_{n}$  with respect to $\widehat{\mathbb P}_{\rho}$ and the above observation we have 

$|\Delta_{n,m}(j)|\leq \widehat{\mathbb P}_{\rho}\{\sup(M_{-r}^{-m-1} , M_{m+1}^{r})>u_{n} \delta ; |X_{0}|>u_{n}\}$.

\noindent Also using  the formula $\theta_{n}=(r_{n}\mathbb P_{\rho} \{|X_{0}
|>u_{n}\})^{-1} \mathbb P_{\rho} \{|M_{1}^{r}|>u_{n}\}$, we have $\theta_{n} \mathbb E_{\rho} (C_{1}^{r} / M_{1}^{r} > u_{n})=r_{n}^{-1} \mathbb E_{\rho} (C_{1}^{r} ; M_{1}^{r} >u_{n}) \mathbb P_{\rho} \{ |X_{0}|>u_{n}\})^{-1}$. Then the use of  stationarity and the above estimations for $\Delta_{n,m} (j)$ and $\Delta_{n,m}$ give the basic relation,
 
$|\theta_{n} \mathbb E_{\rho} (C_{1}^{r}/ M_{1}^{r}> u_{n})- r^{-1} (r-2m) \widehat{\mathbb E}_{\rho}(C_{-m}^{m} ; M_{-m-1}^{-1} \leq u_{n}/|X_{0}|>u_{n})|\leq$ 

$\widehat{\mathbb P}_{\rho}\{\sup (M_{-r}^{-m-1} , M_{m+1}^{r})>u_{n}\delta / |X_{0}|>u_{n}\}+2m \ r_{n}^{-1}$.

\noindent Using Proposition 2.4, we see that the discontinuity set of the function $1_{]0,1]} (M_{-m-1}^{-1} (Y))$ is $\widehat{\mathbb Q}_{\Lambda_{1}}$-negligible, hence using again Proposition 2.4,

$\displaystyle\mathop{\lim}_{n\rightarrow \infty} \widehat{\mathbb E}_{\rho} (C_{-m}^{m} ; M_{-m-1}^{-1} \leq u_{n}/|X_{0}|>u_{n})=\widehat{\mathbb E}_{\Lambda_{1}} (C_{-m}^{m}  (Y) ; M_{-m-1}^{-1}(Y)\leq 1)$. 

\noindent Also $\displaystyle\mathop{\lim}_{n\rightarrow \infty} r_{n}^{-1} (r_{n}-2m)=1$ since $\displaystyle\mathop{\lim}_{n\rightarrow \infty} r_{n}=\infty$. We observe that, by definition of $\theta_{n}$ and $C_{1}^{r}\leq 1$, we have $\theta_{n} \mathbb E_{\rho} (C_{1}^{r}/M_{1}^{r}>u_{n})\leq \theta_{n}\leq 1$. The  anticlustering property of $X_{n}$ implies that the limiting values $(n\rightarrow \infty)$ of $\widehat{\mathbb P}_{\rho}\{\sup (M_{-r}^{-m-1} , M_{m+1}^{r})> \delta u_{n}/ |X_{0}|>u_{n}\}$ are bounded by $\varepsilon_{m}>0$ with $\displaystyle\mathop{\lim}_{m\rightarrow \infty} \varepsilon_{m}=0$. Then the above inequality implies  

$\displaystyle\mathop{\lim\sup}_{n\rightarrow \infty} |\theta_{n} \widehat{\mathbb E}_{\rho} (C_{1}^{r} / M_{1}^{r}>u_{n}) -\widehat{\mathbb E}_{\Lambda_{1}}(C_{-m }^{m} (Y) ; M_{-m-1}^{-1} (Y)\leq 1) |\leq \varepsilon_{m}+2 m\ r_{n}^{-1}$.

\noindent Since $\displaystyle\mathop{\lim}_{m\rightarrow \infty} \widehat{\mathbb E}_{\Lambda_{1}} (C_{-m }^{m} (Y) ; M_{-m-1}^{-1} (Y) \leq 1)=\widehat{\mathbb E}_{\Lambda_{1}} (\hbox{\rm exp}(-\displaystyle\mathop{\Sigma}_{-\infty}^{\infty} f(Y_{j}))\ ;\  M_{-\infty}^{-1} (Y)\leq 1):=I$, we have $\displaystyle\mathop{\lim}_{n\rightarrow \infty}  \theta_{n} \mathbb E_{\rho} (C_{1}^{r}/ M_{1}^{r}>u_{n})=I$. In particular with $f=0$ and using Proposition 2.4, we get $\displaystyle\mathop{\lim}_{n\rightarrow \infty} \theta_{n}=\widehat{\mathbb Q}_{\Lambda_{1}}\{M_{-\infty}^{-1}(Y)\leq 1\}=\theta>0$. 

\noindent Then  we get $\displaystyle\mathop{\lim}_{n\rightarrow \infty} \mathbb E_{\rho} (C_{1}^{r_{n}}/M_{1}^{r_{n}}>u_{n})=\theta^{-1} I= \widehat{\mathbb E}_{\Lambda_{1}} (\hbox{\rm exp}(-\displaystyle\mathop{\Sigma}_{- \infty}^{\infty} f(Y_{j}))/M_{-\infty}^{-1}(Y)\leq 1)$ hence the first assertion, using Proposition 2.4. The expression of $(Z_{j})_{j\in \mathbb N}$ in terms of $\widehat{F}_{-}$ and  $\overline{\mathbb Q}_{\Lambda_{1}}$ explained above and the relation $\displaystyle\mathop{\lim}_{|n|\rightarrow \infty} Y_{n}=0$, $\widehat{\mathbb Q}_{\Lambda_{1}}-$a.e. stated in Proposition 2.5 gives 

\centerline{$\overline{\mathbb Q}_{\Lambda_{1}}\{\displaystyle\mathop{\lim}_{i\rightarrow \infty} Z_{i}=0\}=1$.}

\noindent  Since the discontinuity set of $1_{U'_{1}}$ is $\Lambda_{1}$-negligible,  using the weak convergence of $C_{n}$ to $C$, the continuous mapping theorem (see \cite{30}) and the convergence of $\theta^{-1}_{n}$ to $\theta^{-1}$, we get the formula $\theta^{-1}\geq\overline{\mathbb E}_{\Lambda_{1}} (\displaystyle\mathop{\Sigma}_{j\in \mathbb Z} 1_{U'_{1}} (Z_{j}))$. To go further we write $W_{n}=\displaystyle\mathop{\Sigma}_{1}^{r_{n}} 1_{[u_{n},\infty[} (X_{i})$ and we observe that the stated formula for $\displaystyle\mathop{\lim}_{n\rightarrow \infty} \theta_{n}^{-1}$ is a consequence of the uniform boundedness of $\mathbb E_{\rho} \{W_{n}^{2}/M_{r_{n}}>u_{n}\}$. We have 

$\mathbb E_{\rho}\{W_{n}^{2}/M_{r_{n}}>u_{n}\}=\theta_{n}^{-1}+2 \displaystyle\mathop{\Sigma}_{1\leq i<j\leq r_{n}} \mathbb P_{\rho} \{|X_{i}|>u_{n}, |X_{j}| > u_{n}/ M_{r_{n}}>u_{n}\}$,

\noindent hence using the convergence of $\theta_{n}^{-1}$ to $\theta^{-1}$, we see that it  suffices to bound the second term in the above formula. Writing $j-i=p$, stationarity gives

$\mathbb P_{\rho}\{|X_{i}|>u_{n}, |X_{j}|>u_{n}/M_{r_{n}}>u_{n}\}=\mathbb P_{\rho} \{|X_{0}|>u_{n}, |X_{p}|>u_{n}\} (\mathbb P_{\rho} \{M_{r_{n}} > u_{n}\})^{-1}.$

\noindent We know already that $\theta^{-1}=\displaystyle\mathop{\lim}_{n\rightarrow \infty}(r_{n} \mathbb P_{\rho}\{|X_{1}|>u_{n}\}) (|\mathbb P_{\rho} (M_{r_{n}}>u_{n}))^{-1}$ and we will now use a calculation similar to the one  in the proof of Proposition 2.5. 

\noindent In particular we have :

\noindent $\mathbb P_{\rho} \{|X_{0}|>u_{n}, |X_{p}|>u_{n}\}\leq \mathbb P_{\rho} \{|X_{0}|>u_{n}, |S_{p} X_{0}|>2^{-1}u_{n} \}+\mathbb P_{\rho}\{|X_{0}|>u_{n}, 
 \displaystyle\mathop{\Sigma}_{1}^{p} |S_{p}^{j+1} B_{j}|>2^{-1}u_{n}\}$.

  With $0<\chi<\alpha$  we have using independance :

$\mathbb P_{\rho}\{|X_{0}|>u_{n}, |S_{p} X_{0}|>2^{-1} u_{n}\} \leq 2^{\chi} u_{n}^{-\chi} \mathbb E (|S_{p}|^{\chi}) \mathbb E_{\rho} (|X_{0}|^{\chi} 1_{\{|X_{0}|>u_{n}\}})$.

\noindent Since the law of $X_{0}$ is homogeneous at infinity, the right hand side is bounded by 

$C(\chi) \mathbb P_{\rho} \{|X_{0}|>u_{n}\} \mathbb E (|S_{p}|^{\chi})$.

On the other hand we have 

 $\mathbb P_{\rho} \{|X_{0}|>u_{n}, |S_{p}|>2^{-1} u_{n}\}=\mathbb P_{\rho} \{|X_{0}|>u_{n}\} \mathbb P_{\rho} \{|S_{p}|>2^{-1}
u_{n}\}\leq$

$(2  u_{n}^{-1})^{\chi} (\mathbb E |S_{p}|^{\chi}) (\mathbb P_{\rho}\{|X_{0}|>u_{n}\})\leq C'(\chi) u_{n}^{-\chi}
\mathbb P_{\rho} \{|X_{0}|> u_{n}\}$

\noindent It follows that $\mathbb P_{\rho}\{|X_{p}|>u_{n}/|X_{0}|>u_{n}\}$ is bounded by $C''(\chi) (u_{n}^{-\chi}+\mathbb E(|S_{p}|^{\chi}))$, hence $\displaystyle\mathop{\Sigma}_{1\leq i<j\leq r_{n}} \mathbb P_{\rho} \{|X_{i}>u_{n}, |X_{j}|>u_{n}/M_{r>n}\}$ is bounded by 

\noindent $D(\chi) r_{n}^{2} (\displaystyle\mathop{\Sigma}_{1}^{\infty} (\mathbb E |S_{p}|^{\chi} + u_{n}^{-\chi} ) (\mathbb P_{\rho} \{|X_{0}|>u_{n}\}) (\mathbb P_{\rho} \{M_{r_{n}}>u_{n}\})^{-1}\leq$ 
 $D'(\chi) r_{n} (\displaystyle\mathop{\Sigma}_{1}^{\infty} \mathbb E (|S_{p}|^{\chi})$,   

\noindent since $r_{n}(\mathbb P_{\rho} \{|X_{0}|>u_{n}\}) (\mathbb P_{\rho} \{M_{r_{n}}>u_{n}\})^{-1}$ converges to $\theta^{-1}$.

\noindent Then the uniform boundedness of $\mathbb E_{\rho}(W_{n}^{2} /M_{r_{n}}> u_{n})$ will follow if $r_{n} u_{n}^{-\chi}$ is bounded. In view of the form of $r_{n}, u_{n}$, this amounts to the boundedness of $n^{s}  n^{-\chi/\alpha}$ with $s<1$. Hence it suffices to choose $\chi$ with $\alpha s \leq \chi<\alpha$ in order to get the result.

\noindent The last formula  is  proved in (\cite{2}, Theorem 4.1).  A different proof is sketched in section 5.
\end{proof}
\vskip 4mm

\section{ A spectral gap property and  multiple mixing }

We denote $X_{k}^{x} (k\in \mathbb N)$  the affine random walk on $V$ governed by $\lambda$, starting from $x\in V$ and we write $P \varphi (x)= \int \varphi (hx) d\lambda (h)=\mathbb E(\varphi (X_{1}^{x}))$. 

\noindent In this section we use a spectral gap property for a family of operators associated to the process $X_{k} (1\leq k\leq n)$, in order to show the quasi-independance of its successive blocks of length $r_{n}$, where $r_{n}$ is defined in subsection 2.3.
\vskip 3mm

\subsection{\sl Spectral gap property}
It was proved in (\cite{10}, Theorem 1)  that, given  a probability $\lambda$ on $H$ which satisfies condition (c-e), the corresponding convolution operator $P$ on $V$ satisfies a  ''Doeblin-Fortet''  inequality (see \cite{19}) for suitable Banach spaces $\mathcal{C}_{\chi}$ and $\mathcal{H}_{\chi,\varepsilon,\kappa}$ defined below. In particular, it will be essential here  to use  that  the operator $P$ on $\mathcal{H}_{\chi,\varepsilon,\kappa}$ is the direct sum of a 1-dimensional projection $\pi$ and a contraction $U$ where $\pi$ and $U$ commute, hence we give  also a short proof of this fact  below. In order to obtain the relevant multiple mixing property, we show a global Doeblin-Fortet inequality for a family of operators closely related to $P$. For $\chi, \kappa \geq0$, we consider the weights $\omega, \eta$ on $V$ defined by $\omega(x)=(1+|x|)^{-\chi}, \eta(x)=(1+|x|)^{-\kappa}$. The space $\mathcal{C}_{\chi}$ is the space of continuous functions $\varphi$ on $V$ such that $\varphi(x) \omega(x)$ are bounded and we write $|\varphi|_{\chi}=\displaystyle\mathop{\sup}_{x\in V} |\varphi (x)| \omega(x)$. 

\noindent For $\varepsilon \in ]0,1]$  we write :

\centerline{$[\varphi]_{\varepsilon,\kappa}=\displaystyle\mathop{\sup}_{x\neq y}|x-y|^{-\varepsilon} \eta (x) \eta(y)| \varphi (x)-\varphi (y)|,\ \  \|\varphi\|=|\varphi|_{\chi}+[\varphi]_{\varepsilon, \kappa}$,}

\noindent and we denote by $\mathcal{H}_{\chi, \varepsilon, \kappa}$ the space of functions $\varphi$ on $V$ such that $\|\varphi\|<\infty$. We observe that $\mathcal{C}_{\chi}$ and $\mathcal{H}_{\chi, \varepsilon, \kappa}$ are Banach spaces with respect to the  norms $|.|_{\chi}$ and $\|.\|$ defined above. Also $\mathcal{H}_{\chi,\varepsilon,\kappa}\subset \mathcal{C}_{\chi}$ with compact injection if $\kappa+\varepsilon<\chi$.   We observe that   the operator $P$ acts continuously on $\mathcal{C}_{\chi}$ and $\mathcal{H}_{\chi,\varepsilon,\kappa}$. For a Lipschitz function $f$ on $V$ with non negative real part we define the Fourier-Laplace operator $P^{f}$ by $P^{f} \varphi(x)=P(\varphi \hbox{\rm exp}(-f))$. In \cite{10}, spectral gap properties for Fourier operators were studied for $f(v)=i<x,v>, x\in V$. Here the calculations are analogous but $f$ will be Lipschitz and bounded. We observe that for functions $f_{k}(1\leq k \leq n)$ and $\varphi$ as above we have : 

\centerline{$P^{f_{1}} P^{f_{2}} \cdots P^{f_{n}} \varphi (x) = \mathbb E\{\varphi (X_{n}^{x}) \hbox{\rm exp}(-\displaystyle\mathop{\Sigma}_{k=1}^{n} f_{k} (X_{k}^{x}))\}$}

\noindent 
Also we note that, for $f$ bounded, with $k(f)=\displaystyle\mathop{\sup}_{x\neq y} |x-y|^{-1} |f(x)-f(y)|<\infty$ 

 $|x-y|^{-\varepsilon} |f(x)-f(y)|\leq \displaystyle\mathop{\inf}_{x\neq y} (2|f|_{\infty} |x-y|^{-\varepsilon}, k(f) |x-y|^{1-\varepsilon})\leq 2|f|_{\infty}+k(f):=k_{1}(f)$,
 
 \noindent For $u,v$ with non negative real parts we have
$|\hbox{\rm exp}(-u)-\hbox{\rm exp}(-v)|\leq |u-v|$.
In particular, for $f$ as above,  $|\hbox{\rm exp}(-f(x))-\hbox{\rm exp}(-f(y))|\leq k_{1}(f)) |x-y|^{\varepsilon}$.

\noindent It follows that  multiplication by $\hbox{\rm exp}(-f)$ acts continuously on $\mathcal{C}_{\chi}$, $\mathcal{H}_{\chi,\varepsilon,\kappa}$, hence $P^{f}$ is a bounded operator on $\mathcal{C}_{\chi}$ and $\mathcal{H}_{\chi,\varepsilon,\kappa}$. For $m,\gamma>0$ we denote by $O(m,\gamma)$ the set of operators $P^{f}$ such that $|f|_{\infty}\leq m$ and $k(f)\leq \gamma$, hence $k_{1}(f) \leq 2m+\gamma$. For $p\in \mathbb N$ let $O^{p}(m,\gamma)$ be the set  of products of $p$ elements in $O(m,\gamma)$ and $\widehat{O}(m,\gamma)=\displaystyle\mathop{\cup}_{p>0} O^{p}(m,\gamma)$. We will  endow $\widehat{O}(m,\gamma)$ with the natural norm from $End(\mathcal{H}_{\chi,\varepsilon,\kappa})$. Then we have the 
\vskip 3mm
\begin{thm} 
With the above notations and $0\leq\chi< 2\kappa < 2\kappa+\varepsilon<\alpha$, there exists $C(m,\gamma)\geq 1$ such that for any $Q\in O(m,\gamma)$ the norm $\|Q\|$ of $Q$ on $\mathcal{H}_{\chi,\varepsilon,\kappa}$  is bounded by $C(m,\gamma)$.  
 Furthermore there exists $r\in [0,1[$, $p\in \mathbb N$, $D>0$ such that for any $Q\in O^{p}(m,\gamma)$, $\varphi \in \mathcal{H}_{\chi,\varepsilon,\kappa}$ :
$$\|Q \varphi\|\leq r\|\varphi\|+D |\varphi|_{\chi} .$$

In particular $\widehat{O}(m,\gamma)$ is a bounded subset of $End(\mathcal{H}_{\chi,\varepsilon,\kappa})$ and $C(m,\gamma), r, D$ depend only of $m,\gamma$.
\end{thm}
\vskip 3mm
\noindent The proof depends on the two lemmas given  below, and of calculations  analogous to those of \cite{10} for Fourier operators.
\vskip 3mm

\begin{lem}

\noindent $\widehat{O}(m,\gamma)$ is a bounded subset of $End(\mathcal C_{\chi})$.
\end{lem}

\begin{proof}
 Since $Re (f)\geq 0$ we have for $Q\in O^{\ell} (m,\gamma)$ with $\ell \in \mathbb N$ , $\varphi\in \mathbb C_{\chi} : |Q\varphi|_{\chi}\leq |P^{\ell} |\varphi||_{\chi}$, hence  it suffices to show that the set $\{P^{\ell} ; \ell \in \mathbb N\}$ is bounded in $End(\mathcal C_{\chi})$. We have for $\varphi\geq 0$, with $M=P^{\ell}$ : 

$\omega (x) M\varphi (x)=\omega (x) \mathbb E (\varphi (X_{\ell}^{x})) \leq |\varphi|_{\chi} \mathbb E[\omega (x) \omega^{-1}(X_{\ell}^{x})]$.

\noindent If $\chi\leq 1$, using independance and the expression of $X_{\ell}^{x}$ we get 

$\omega(x) M\varphi (x) \leq |\varphi|_{\chi}(1+\mathbb E |S_{\ell}|^{\chi}+\displaystyle\mathop{\Sigma}_{1}^{\ell} (\mathbb E | S_{\ell}^{k+1}|^{\chi}) |\   (\mathbb E (|B_{k}|^{\chi})$, 

\noindent hence $\displaystyle\mathop{\sup}_{x\in V} \omega (x) M \varphi (x) \leq |\varphi|_{\chi} (1+\displaystyle\mathop{\sup}_{\ell\geq 1} \mathbb E |S_{\ell}|^{\chi}+(\mathbb E(|B_{1}|^{\chi}) (\displaystyle\mathop{\Sigma}_{1}^{\infty} \mathbb E |S_{\ell}|^{\chi})$.

\noindent Since $\chi<\alpha$, we have $\displaystyle\mathop{\lim}_{\ell \rightarrow \infty} (\mathbb E |S_{\ell}|^{\chi})^{1/\ell}=k(\chi)<1$, hence $\displaystyle\mathop{\sup}_{x\in V} \omega (x) |M\varphi (x)|$ is bounded by $C_{\chi}|\varphi|_{\chi}$ with $C_{\chi}<\infty$.

\noindent If $\chi>1$, we use Minkowski's inequality in $\mathbb L^{\chi}$ and write  :

$\omega(x)  |M\varphi (x)|\leq |\varphi|_{\chi} (1+(\mathbb E  |S_{\ell}|^{\chi})^{1/\chi}+\displaystyle\mathop{\Sigma}_{1}^{\ell}  \mathbb E (|S_{\ell}^{k+1}|^{\chi})^{1/\chi} (\mathbb E(|B_{k}|^{\chi})^{1/\chi})$

\noindent As above we get 

$\displaystyle\mathop{\sup}_{x\in V} \omega (x) |M\varphi (x)|\leq C_{\chi}|\varphi|_{\chi}$ with $C_{\chi}<\infty$.
\end{proof}
\vskip 3mm
\begin{lem}

\noindent a) For $\beta\in [0,\alpha[$ we have $\displaystyle\mathop{\sup}_{n} \mathbb E |X_{n}^{0}|^{\beta}<\infty$.

\noindent b) For $\beta_{1}, \beta>0$ and $\beta+\beta_{1}<\alpha$, we have $\displaystyle\mathop{\lim}_{n\rightarrow \infty} (\mathbb E (|S_{n}|^{\beta_{1}} |X_{n}^{0}|^{\beta}))^{1/n}<1$. 

\noindent c) If $\chi+\varepsilon<\alpha$ the quantity $\widetilde C_{n}=\mathbb E (\displaystyle\mathop{\Sigma}_{1}^{n} |S_{i}|^{\varepsilon} (1+|S_{n}|+|X_{n}^{0}|)^{\chi})$ is bounded. Furthermore, if $2\kappa+\varepsilon<\alpha$, then
$\widetilde D_{n}=\mathbb E(|S_{n}|^{\varepsilon}(1+|S_{n}|+|X_{n}^{0}|)^{2\kappa})$ satisfies $\displaystyle\mathop{\lim}_{n\rightarrow \infty} (\widetilde D_{n})^{1/n}<1$.
\end{lem}

\begin{proof}  
 a) We write $|X_{n}^{0}|^{\beta}=|\displaystyle\mathop{\Sigma}_{1}^{n}  S_{n}^{k+1} B_{k}|^{\beta}$. If $\beta\leq 1$ we get :

$\mathbb E(|X_{n}^{0}|^{\beta})\leq \displaystyle\mathop{\Sigma}_{1}^{n}  (\mathbb E |S_{n}^{k+1}|^{\beta}) (\mathbb E  |B_{k}|^{\beta})=(\mathbb E |B_{1}|^{\beta}) (\displaystyle\mathop{\Sigma}_{0}^{n-1} \mathbb E  |S_{j}|^{\beta})$

\noindent Since  $\displaystyle\mathop{\lim}_{j\rightarrow \infty} (\mathbb E |S_{j}|^{\beta})^{1/j}<1$ if $\beta<\alpha$ we get $\displaystyle\mathop{\sup}_{n\geq 0} \mathbb E |X_{n}^{0}|^{\beta} \leq (\mathbb E  |B_{1}|^{\beta}) (\displaystyle\mathop{\Sigma}_{0}^{\infty} \mathbb E |S_{j}|^{\beta})<\infty$.

\noindent If $\beta>1$, we use Minkowski's inequality in $\mathbb L^{\beta}$ as in the proof of Lemma 3.2.

\noindent b) Using H\"older's inequality we have

$\mathbb E(|S_{n}
|^{\beta_{1}} |X_{n}^{0}|^{\beta}) \leq (\mathbb E |S_{n}|^{\beta+\beta_{1}})^{\beta_{1}/\beta+\beta_{1}}) \mathbb E |X_{n}^{0}|^{\beta+\beta_{1}})^{\beta/\beta+\beta_{1}})$,

\noindent hence the result follows from a) and the fact that $\displaystyle\mathop{\lim}_{n\rightarrow \infty} (\mathbb E |S_{n}|^{\beta+\beta_{1}})^{1/n}<1$ since $\beta+\beta_{1}<\alpha$.

\noindent c) The assertions follows from easy estimations as in b) and the conditions $\chi+\varepsilon<\alpha$, $2\kappa+\varepsilon<\alpha$.
\end{proof}
\vskip 3mm

\begin{proth} \textbf{3.1}
 We start with a basic observation. For $n>0$ we have $X_{n}^{x}=h_{n}\cdots h_{1}x=S_{n} x+\displaystyle\mathop{\Sigma}_{1}^{n} S_{n}^{k+1} B_{k}$, hence $|X_{n}^{x}-X_{n}^{y}|=|S_{n}(x-y)|\leq |S_{n}| |x-y|$. It follows for $k(f)\leq \gamma$, $x$ and $y$ in $V$ :
$$|f(X_{n}^{x})-f(X_{n}^{y})|\leq \gamma |S_{n}| |x-y|.$$
\noindent We write $M=T_{1} T_{2}\cdots T_{n}$ with $T_{i}=P^{f_{i}}\in O(m,\gamma)$ $1\leq i\leq n$. We have using Markov property, 

$M \varphi (x)-M\varphi (y)=I_{n} (x,y)+J_{n}(x,y)$ with 

$I_{n}(x,y)=\mathbb E ( [\hbox{\rm exp}( -\displaystyle\mathop{\Sigma}_{1}^{n} f_{i} (X_{i}^{x}))-\hbox{\rm exp}( -\displaystyle\mathop{\Sigma}_{1}^{n}  f_{i} (X_{i}^{y})) \varphi (X_{n}^{x})])$

$J_{n} (x,y)=\mathbb E ((\hbox{\rm exp}(-\displaystyle\mathop{\Sigma}_{1}^{n}  f_{i}(X_{i}^{y}))) (\varphi(X_{n}^{x})-\varphi (X_{n}^{y})))$

\noindent  Since $Re (f)\geq0$ we have :

$|\hbox{\rm exp}(-\displaystyle\mathop{\Sigma}_{1}^{n}  f_{i}(X_{i}^{x}))-\hbox{\rm exp}(-\displaystyle\mathop{\Sigma}_{1}^{n}  f_{i}(X_{i}^{y}))| \leq \displaystyle\mathop{\Sigma}_{1}^{n}  | f_{i}(X_{i}^{x})-f_{i}(X_{i}
^{y})|\leq (2m+\gamma) \displaystyle\mathop{\Sigma}_{1}^{n} |X_{i}^{x}-X_{i}^{y}|^{\varepsilon}$.  

\noindent The basic observation gives :

$I_{n}(x,y)\leq (2m+\gamma) |\varphi|_{\chi} |x-y|^{\varepsilon} C_{n}(x)$ with $C_{n}(x)=\mathbb E (\displaystyle\mathop{\Sigma}_{1}^{n} |S_{i}|^{\varepsilon} (1+|X_{n}^{x}|)^{\chi}$

$J_{n}(x,y)\leq \mathbb E |\varphi (X_{n}^{x})-\varphi (X_{n}^{y})| \leq [\varphi]_{\varepsilon,\kappa} |x-y|^{\varepsilon} D_{n}(x,y)$, 

with $D_{n}(x,y)=\mathbb E (|S_{n}|^{\varepsilon} (1+|X_{n}
^{x}|)^{\kappa} (1+|X_{n}^{y}|)^{\kappa})$. 

\noindent Using symmetry of $|M\varphi (x)-M\varphi (y)|$, $\chi \leq 2\kappa$ and $|X_{n}^{x}| \leq |S_{n}| |x|+|X_{n}^{0}|$,  we get $[M \varphi]_{\varepsilon,\kappa} \leq (2m+\gamma) |\varphi|_{\chi} \widetilde C_{n}+[\varphi]_{\varepsilon,\kappa} \widetilde D_{n}$ where $\widetilde C_{n}, \widetilde D_{n}$ are as in Lemma 3.3. 

\noindent Using Lemma 3.3 we can choose $p\in \mathbb N$ such that $r=\widetilde D_{p} <1$, hence for $M\in O^{p} (m,\gamma)$, $[M\varphi]_{\varepsilon,\kappa} \leq k_{1} (f) \widetilde C_{p} |\varphi|_{\chi}+r [\varphi]_{\varepsilon,\kappa}$.

\noindent Using Lemma 3.2 we see that there exists $C_{\chi}\geq 1$ such that $|M\varphi|_{\chi} \leq C_{\chi}|\varphi|_{\chi}$ for $M\in \widehat{O}(m,\gamma)$,  $\varphi\in \mathcal C_{\chi}$.  Then for $M\in O^{p} (m,\gamma)$, $\varphi \in \mathcal{H}_{\chi,\varepsilon,\kappa}$ and $p$ as above :

$\|M \varphi\| \leq r \|\varphi\|+(C_{\chi}+2m+\gamma) \widetilde C_{p} |\varphi|_{\chi}=r \|\varphi\|+D |\varphi|_{\chi}$ with $D>0$.

\noindent For the last assertion, assume $M\in O^{n} (m,\gamma)$ and write $n=p n_{1}+n_{0}$ with $n_{1}\in \mathbb N$, $0\leq n_{0}<p$. We have $M=Q_{1}\cdots Q_{n_{1}} R_{1}\cdots R_{n_{0}}$ with $Q_{i}\in O^{p} (m,\gamma)$ $(1\leq i \leq p)$ and $R_{j}\in O(m,\gamma)$ $(0\leq j\leq n_{0})$, hence $\|R_{j}\| \leq C_{\chi} (m,\gamma)$. Finally we get

$\|M\varphi\|\leq C_{\chi}(m,\gamma)^{n_{0}} \left[r^{n_{1}} \|\varphi\|+D|\varphi|_{\chi} (r^{n_{1}-1}+C_{\chi} \displaystyle\mathop{\Sigma}_{0}^{n_{1}-2} r^{k})\right]$,

$\|M\|\leq C_{\chi} (m,\gamma)^{p} \left[1+D(1+C_{\chi}(1-r)^{-1})\right] := C(m,\gamma)$, which gives the result. \qquad 
\end{proth}
\vskip 3mm
\noindent For $\chi\in ]0,\alpha[$ we consider the function $W^{\chi}$ on $V$ defined by $W^{\chi}(x)=|x|^{\chi}$. In Proposition 3.4 below we show that, due to the inequality $0<k(\chi)<1$ for $\chi \in ]0,\alpha[$,  $P$ satisfies a drift condition (see \cite{23}) with respect to $W^{\chi}$. The same inequality implies also a spectral gap property in the Banach space $\mathcal H_{\chi,\varepsilon,\kappa}$ considered in Proposition 3.4 below. For reader's convenience we recall 
 the Doeblin-Fortet spectral gap theorem (see \cite{19}). 
 
 \noindent Let $(F, |.|)$ be a Banach space, $(L, \|.\|)$ another Banach space with a continuous injection $L\rightarrow F$. Let $P$ be a bounded operator on $F$, which preserves $L$ and satisfies the following conditions 

 1) The sequence of operator norms $|P^{n}|$ in is bounded.

2) The injection $L\rightarrow F$ is compact.

3) There exists an integer $k$ and $r\in [0,1[$, $D>0$ such that for any $v$ in $L$ :

\centerline{$\|P^{k} v\| \leq r \|v\| + D|v|$}

4) If $v_{n}\in L$ is a sequence and $v\in F$ are such that $\|v_{n}\|\leq 1$ and $\displaystyle\mathop{\lim}_{n\rightarrow \infty} |v-v_{n}|=0$, there $v\in L$ and $\|v\|\leq 1$

\noindent Then in restriction to $L$, $P$ is the commuting direct sum of a finite dimensional operator $\pi$ with unimodular spectral values and a bounded operator $U$ with spectral radius $r(U)<1$.

\noindent We observe that, frequently the norm $\|.\|$ on $L$ is given as a sum of a semi-norm $[.]$ and the norm $|.|$ ; then the inequality in condition 3 can be replaced by 

\centerline{$[P^{k} v] \leq r [v]+D|v|$}

\noindent such an inequality is called Doeblin-Fortet's inequality. 

\noindent Our substitute for  the strong mixing property (see \cite{30}) uses regularity of functions and is the following.
\vskip 3mm

\begin{prop}
For any $\beta \in ]0,1]$ there exists $\ell \in \mathbb N$ and $b\geq 0$ such that $P^{\ell} W^{\chi}\leq \beta W^{\chi}+b$ for $n\geq \ell$. In particular the sequence of norms $|P^{n}|_{\chi}$ is bounded. Furthermore, if $0<\kappa+\varepsilon<\chi<2\kappa<2\kappa+\varepsilon<\alpha$, the injection of  $\mathcal H_{\chi, \varepsilon,\kappa}$   into  $\mathcal C_{\chi}$  is compact and on $\mathcal H_{\chi,\epsilon,\kappa}$, the Markov operator $P$ satisfies the direct sum decomposition 

\centerline{$P=\rho \otimes 1+U$}

\noindent where $r(U)<1$ and $U(\rho \otimes 1)=(\rho\otimes 1) U=0$

\noindent If $\alpha=1$ and $0<\varepsilon<\chi<1$, $\kappa=0$, the same result is valid.
\end{prop}

\begin{proof}
 We verify successively the four above conditions. First we observe that for any $x\in V$,

\centerline{$|X_{n}^{x}-X_{n}^{0}| \leq |S_{n}| |x|,\ |X_{n}^{x}|\leq |X_{n}^{0}|+|S_{n}| |x|$.}
\noindent If $\chi \leq 1$, it follows

\centerline{$\mathbb E |X_{n}^{x}|^{\chi} \leq \mathbb E |X_{n}^{0}|^{\chi}+ (\mathbb E |S_{n}|^{\chi}) |x|^{\chi}$.}

\noindent Using the expression of $X_{n}^{0}$ and independence we get $\mathbb E |X_{n}^{0}|^{\chi}\leq (\mathbb E |B_{1}|^{\chi}) (\displaystyle\mathop{\Sigma}_{0}^{\infty} \mathbb E |S_{k}|^{\chi})$. Since $\chi<\alpha$, we have $\mathbb E (|X_{n}^{0}|^{\chi})\leq b<\infty$. On the other hand we have $\displaystyle\mathop{\lim}_{n\rightarrow \infty} (\mathbb E (|S_{n}|^{\chi}))^{1/n}=k(\chi)<1$, hence for some $\varepsilon>0$ $k(\chi)+\varepsilon<1$, and for $n\geq \ell$, $|S_{n}|^{\chi} \leq \beta'\leq (k(\chi)+\varepsilon)^{n}$. It follows, for $n\geq \ell$ :

\centerline{$P^{n} W^{\chi} (x)=\mathbb E |X_{n}^{x}|^{\chi} \leq \beta' W^{\chi} (x)+b$}

\noindent If $\chi >1$we use Minkowski's inequality, hence :

\centerline{$\mathbb E |X_{n}^{x}|^{\chi}\leq 2^{\chi} (\mathbb E |X_{n}^{0}|^{\chi}+\mathbb E |S_{n}|^{\chi} |x|^{\chi})$}

\noindent As above, using $k(\chi)+\varepsilon <1$ and $n\geq \ell$ we get 

$\mathbb E (|X_{n}^{x}|^{\chi}) \leq 2^{\chi} b+2^{\chi} (k(\chi)+\varepsilon)^{n}|x|^{\chi},\ P^{n} W^{\chi} \leq \beta'' W^{\chi}+b'$ with $\beta''<1$, $b'<\infty$. 

\noindent We take $\beta=\beta'$ or $\beta''$ depending on $\chi \leq 1$ or $\chi>1$. This allow us now to show that $|P^{n}|_{\chi}$ is bounded. We observe that $|\varphi (x)|\leq (1+W(x))^{\chi} |\varphi|_{\chi}$, hence the positivity of $P$ and $P1=1$ implies for $n\in \mathbb N$, 
$$|P^{n} \varphi | (x) \leq |\varphi |_{\chi} P^{n} (2^{\chi} +2^{\chi} W^{\chi} (x))=|\varphi |_{\chi} (2^{\chi}+2^{\chi} P^{n} W^{\chi} (x)).$$
\noindent From above we get
$$|P^{n} \varphi| (x) \leq |\varphi|_{\chi} [2^{\chi}+2^{\chi} (b+\beta W^{\chi} (x))].$$
\noindent Then the definition of $|P^{n}|_{\chi}$ gives $|P^{n}|_{\chi} \leq 2^{\chi} (1+b+\beta)$, hence the boundedness of $|P^{n}|_{\chi}$.
\vskip 2mm
\noindent In order to show that if $\kappa+\varepsilon<\chi$, the injection of $\mathcal H_{\kappa,\varepsilon,\chi}$ in $F=\mathcal C_{\chi}$ is compact, we use Ascoli's argument and consider  a large ball $U_{t}$ with $t>0$. We consider $\varphi_{n}\in \mathcal H_{\kappa,\varepsilon, \chi}$ with $\|\varphi_{n}||<1$. The definition on $\|\varphi_{n}\|$ implies for any $x, y \in U_{t}$

\centerline{$|\varphi_{n}(x)|\leq(1+t)^{\chi}, |\varphi_{n}(x)-\varphi_{n}(y)|\leq (1+t)^{2\kappa} |x-y|^{\varepsilon}$}

\noindent Hence, the restrictions of $\varphi_{n}$ to $U_{t}$ are equicontinuous and we can find a convergent subsequence $\varphi_{n_{k}}$. Using  the diagonal procedure and a sequence $t_{i}$ with $\displaystyle\mathop{\lim}_{i\rightarrow \infty} |t_{i}|=\infty$, we get a convergent subsequence $\varphi_{n_{j}}\in \mathcal H_{\kappa,\varepsilon,\chi}$ with limit a continuous function $\varphi$ on $V$. From above we have $|\varphi_{n_{j}}(x)-\varphi_{n_{j}}(0)| \leq (1+|x|)^{\kappa} |x|^{\varepsilon}$. hence for some $A, B>0$, since $\kappa+\varepsilon<\chi$ 
$$|\varphi_{n_{j}}(x)-\varphi_{n_{j}}(0)|\leq (1+|x|)^{\kappa+\varepsilon}, |\varphi(x)|\leq A+B(1+|x|)^{\chi}.$$ 

\noindent It follows that $\varphi\in \mathcal C_{\chi}$. 
 The above inequalities for $\varphi_{n_{j}}$ imply 
 $$|(\varphi_{n_{j}}(x)-\varphi_{n_{j}}(0))-(\varphi (x)-\varphi (0))|\leq 2 (1+|x||^{\kappa+\varepsilon}.$$ 
Then the convergence of $\varphi_{n_{j}}$ to $\varphi$, implies with $\varepsilon_{n_{j}}= |\varphi_{n_{j}}(0)-\varphi (0)|$, 
 \vskip 2mm
\noindent $|\varphi_{n_{j}} (x)-\varphi(x)|\leq \varepsilon_{n_{j}}+2(1+|x|)^{\kappa+\varepsilon}$, $(1+|x|)^{-\chi} |\varphi_{n_{j}}(x) -\varphi (x)|\leq \varepsilon_{n_{j}}+2(1+|x|)^{\kappa+\varepsilon-\chi}$
 \vskip 2mm
 \noindent with $\displaystyle\mathop{\lim}_{j\rightarrow \infty} \varepsilon_{n_{j}}=0$.
Also for $t$ sufficiently large, and $|x|\geq t$, since $\kappa+\varepsilon<\chi$ we have $(1+|x|)^{\kappa+\varepsilon-\chi} \leq \varepsilon_{n_{j}}$. Furthermore, the uniform convergence of $\varphi_{n_{j}}$ to, $\varphi$ on  $U_{t}$ implies $\displaystyle\mathop{\lim}_{j\rightarrow \infty} (\sup \{|\varphi_{n_{j}}(x)-\varphi (x)|\ ;\ |x|\leq t\})=0$. The convergence of $|\varphi_{n_{j}}-\varphi|_{\chi}$ to zero follows. 

\noindent The convergence of $\varphi_{n_{j}}(x)$ to $\varphi (x)$ for any $x\in V$ and the definition of $\|\varphi_{n_{j}}\|$, implies $\|\varphi\| \leq \displaystyle\mathop{\lim}_{j\rightarrow \infty} \|\varphi_{n_{j}}\|\leq 1$, hence $\varphi\in L$, and condition 4 is satisfied.
\vskip 3mm

\noindent With $f=0$ in Theorem 3.1  we have $P^{f}=P$. In particular there exists $k>0$ such that $\|P^{k}\varphi\| \leq r\|\varphi\|+D |\varphi|_{\chi}$ if $\varphi \in \mathcal H_{\chi,\varepsilon,\kappa}$. Hence from \cite{19}, we know that the  above conditions imply that $P$ is the direct sum of a finite rank operator and a bounded operator $U$ which satisfies $r(U)<1$. Now it suffices to show that the equation $P\varphi=z\varphi$ with $|z|=1,\ \varphi\in \mathcal H_{\chi,\varepsilon,\kappa}$ implies that $\varphi$ is constant and $z=1$. From the convergence in law of $X^{x}_{n}$ to $\rho$ we know that for any $x\in V$, the sequence of measures $P^{n} (x,.)$ converges weakly to $\rho$. Also we have $|\varphi| \in \mathcal H_{\chi,\varepsilon,\kappa}$ and the sequence $n^{-1} \displaystyle\mathop{\Sigma}_{1}^{n} P^{k} |\varphi|$ converges to $\rho(|\varphi|)$. Since $|\varphi (x)|=|z^{n}\varphi (x)|\leq P^{n} (x, |\varphi|)$ we get $|\varphi (x)|\leq \rho (|\varphi|)$, hence $|\varphi|$ is bounded. Since $z^{n}\varphi (x)=\mathbb E(\varphi (X_{n}^{x}))$ and $X_{n}^{x}$ converges in law to $\rho$, we get $\displaystyle\mathop{\lim}_{n\rightarrow \infty} z^{n} \varphi (x)=\rho(\varphi)$. This implies $z=1$ and $\varphi(x)=\rho(\varphi)$ for any $x\in V$. 

\noindent For the last assertion, in view of the above, we have only to verify the contraction condition. We write $[\varphi]_{\varepsilon}=\displaystyle\mathop{\sup}_{x\neq y} |x-y|^{-\varepsilon} |\varphi(x)-\varphi (y)|$. Then we have 
$$\mathbb E (|\varphi(X_{n}^{x})-\varphi (X_{n}^{y})|)\leq [\varphi]_{\varepsilon} |X_{n}^{x} -X_{n}^{y}|^{\varepsilon} \leq [\varphi]_{\varepsilon} |x-y|^{\varepsilon} \mathbb E (|S_{n}|^{\varepsilon}).$$
Since $\varepsilon<\alpha$, we have $0<k(\varepsilon)< r<1$ for some $r$, hence $[P^{n} \varphi]_{\varepsilon}\leq r [\varphi]_{\varepsilon}$ for $n$ large. 
\end{proof}

\vskip 3mm
\subsection{\sl  A  mixing property with speed for the system  $(V^{\mathbb Z_{+}}, \tau, \mathbb P_{\rho})$.}
 In general, if the law of $B_{n}$ has no density with respect to Lebesgue measure, the operator $P$ on $\mathbb L^{2}(\rho)$ doesn't satisfy spectral gap properties,  hence  the stationary process $X_{n}$ is not strongly mixing in the sense of \cite{31} but Proposition 3.4 above shows that it is still ergodic. A simple example is as follows. Let $V=\mathbb R$ and let $P$ be the operator defined on $\mathbb L^{2}(\rho)$ by the formula $P \varphi (x)=\frac{1}{2} [\varphi (\frac{x}{x})+\varphi (\frac{x+1}{2})]$. Then $P$ preserves $[0,1]$, $\rho$ is uniform measure on $[0,1]$  and the adjoint $P^{*}$ in $\mathbb L^{2}(\rho)$ of $P$ can be identified with the map $x\rightarrow \{2x\}$ on $[0,1]$ endowed with Lebesgue measure. Then the spectrum of $P^{*}$ in $\mathbb L^{2}(\rho)$ is contained in $\{|z|=1\}$ (and is in fact absolutely continuous). Hence $P^{*}$ has no spectral gap in $\mathbb L^{2}(\rho)$ ; by duality this is true also of $P$.

 \noindent Then, using Theorem 3.1 and Proposition 3.4, it is shown below that the system $(V^{\mathbb Z_{+}}, \tau, \mathbb P_{\rho})$  satisfies a multiple mixing condition with respect to Lipschitz functions.  For a study of extreme value properties for random walks on some classes of homogeneous spaces, using $\mathbb L^{2}$-spectral gap methods, we refer to \cite{21}. Since, using Proposition 2.4, the stationary process $(X_{n})_{n\in \mathbb N}$ satisfies also anticlustering, we see below that extreme value theory can be developed for $(X_{n})_{n\in \mathbb N}$ following the arguments of (\cite{2} , \cite{3}) which were developed under mixing conditions involving continuous functions. 

\noindent However it turns out that the mixing property $\mathcal A' (u_{n})$ of \cite{3} for  continuous functions can be proved, as a consequence of the corresponding convergences involving Lipschitz functions and point process theory. This conditon plays an essential role in the study of space-time convergence (see \cite{3}).

\noindent Let $f$ be a bounded continuous function with non negative real part on $[0,1] \times (V\setminus \{0\})$. Let $r_{n}$ be an integer valued  sequence with $\displaystyle\mathop{\lim}_{n\rightarrow \infty} r_{n}=\infty$, $r_{n}=o (n)$ and $k_{n}=[r_{n}^{-1} n]$. For $0\leq i \leq n,\ 0\leq j\leq n$, $x \in V\setminus \{0\}$, $\omega\in V^{\mathbb Z_{+}}$ we write : 

$\overline{f}^{j}_{n} (x)=f(n^{-1}j, u_{n}^{-1} x), f_{i,n}(\omega)=\overline{f}^{i}_{n} (X_{i}),\   f_{i,n}^{j} (\omega)=\overline{f}_{n}^{j}(X_{i})$. 

\noindent In view of heavy notations, in some formulae we will write   $r_{n}=r$, $k_{n}=k$, $\ell_{n}=\ell$.  For $f$ Lipschitz we denote by $k(f)$ the Lipschitz constant of $f$, and assume supp$(f) \subset [0,1]\times U'_{\delta}$ with $\delta>0$. We consider below the quantity $\mathbb E_{\rho} (\hbox{\rm exp}(-\displaystyle\mathop{\Sigma}_{i=1}^{n}f_{i,n}))$ which is the Laplace functional of the point process $\displaystyle\mathop{\Sigma}_{i=1}^{n} \varepsilon_{(n^{-1}i, u_{n}^{-1} X_{i})}$. For its analysis we use the classical Bernstein method of gaps, i.e. we decompose the interval $[1,n]$ into large subintervals separated by smaller but still large ones.
\vskip 3mm
\begin{prop}
 Let $f$ be a compactly supported Lipschitz function on $[0,1] \times (V\setminus \{0\})$ with $Ref\geq 0$. Assume that the sequence $r_{n}\in \mathbb N$ satisfies $r_{n}=o (n)$, $\displaystyle\mathop{\lim}_{n\rightarrow \infty} (\hbox{\rm log} n)^{-1} r_{n}=\infty$ and write $|f|_{\infty}=m, k(f)=\gamma$, \hbox{\rm supp}$(f) \subset [0,1] \times U'_{\delta} , \delta>0$. Then, with the above notations there exists $C(\delta,m,\gamma)<\infty$ such that, 

$I_{n} (f) := |\mathbb E_{\rho} \{\hbox{\rm exp}(-\displaystyle\mathop{\Sigma}_{i=1}^{n} f_{i,n})\}  - \displaystyle\mathop{\Pi}_{j=1}^{k_{n}} \mathbb E_{\rho} \{ \hbox{\rm exp} (-\displaystyle\mathop{\Sigma}_{(j-1) r_{n}+1}^{jr_{n}}f_{i,n}^{j r_{n}})\}|\leq C(\delta,m,\gamma) \sup (r_{n}^{-1}, n^{-1} r_{n})$.

\noindent In particular with $r_{n}=[n^{1/2}]$ we get $\sup (n^{-1}r_{n}, r_{n}^{-1})\leq 2n^{-1/2}$ 
\end{prop}

\begin{proof}
We write $[0,n]=[0, k_{n} r_{n}] \cup ]k_{n} r_{n}, n]$, we decompose the interval $[0,k_{n} r_{n}]$ into $k_{n}$ intervals $J_{j}=[j r_{n}, (j+1) r_{n}[$ and we distinguish in $J_{j}$ the subinterval of length $\ell_{n}$ $J'_{j}=[(j+1) r_{n}-\ell_{n}, (j+1) r_{n}[$ ; the large integer $\ell_{n}$ will be specified below. 

\noindent We write for $f$ fixed, $I(n)=|\mathbb E_{\rho}(\hbox{\rm exp}( -\displaystyle\mathop{\Sigma}_{1}^{n} f_{i,n})) - \displaystyle\mathop{\Pi}_{j=1}^{k} \mathbb E_{\rho} (\hbox{\rm exp}( -\displaystyle\mathop{\Sigma}_{i=(j-1)r+1}^{jr} f_{i,n}^{jr}))|$. 

\noindent Then the triangular inequality gives $I(n)\leq I_{1}(n)+I_{2}(n)+I_{3}(n)+I_{4} (n)$ with

$I_{1}(n)=|\mathbb E_{\rho}(\hbox{\rm exp} (-\displaystyle\mathop{\Sigma}_{1}^{n} f_{i,n}))-\mathbb E_{\rho}(\hbox{\rm exp}(- \displaystyle\mathop{\Sigma}_{1}^{kr} f_{i,n}))|$

$I_{2}(n)=|\mathbb E_{\rho}(\hbox{\rm exp}( -\displaystyle\mathop{\Sigma}_{1}^{kr} f_{i,n}))-\mathbb E_{\rho}(\hbox{\rm exp}( -\displaystyle\mathop{\Sigma}_{j=1}^{k}\ \  \displaystyle\mathop{\Sigma}_{i=(j-1)r+1}^{jr-\ell} f_{i,n}))|$

$I_{3 }(n)=|\mathbb E_{\rho}(\hbox{\rm exp} (-\displaystyle\mathop{\Sigma}_{j=1}^{k}\ \  \displaystyle\mathop{\Sigma}_{i=(j-1) r+1}^{jr-\ell} f_{i,n}))- \displaystyle\mathop{\Pi}_{j=1}^{k}\mathbb E_{\rho}(\hbox{\rm exp}(- \displaystyle\mathop{\Sigma}_{i=1}^{r-\ell} f_{i,n}^{jr}))|$

$I_{4}(n)=|\displaystyle\mathop{\Pi}_{j=1}^{k}\mathbb E_{\rho}(\hbox{\rm exp}( -\displaystyle\mathop{\Sigma}_{i=1}^{r-\ell} f_{i,n}^{jr}))-\displaystyle\mathop{\Pi}_{j=1}^{k}\mathbb E_{\rho}(\hbox{\rm exp}(- \displaystyle\mathop{\Sigma}_{i=1}^{r} f_{i,n}^{jr}))|$

\noindent where stationarity of $\mathbb P_{\rho}$ has been used in the expressions of $I_{3} (n), I_{4}(n)$. The quantities $I_{1},\ I_{2}, \ I_{4}$ are boundary terms ; their estimation below is based only on the fact that $r_{n}$ (resp. $\ell_{n}$) is small with respect to $n$ (resp. $r_{n}$), the form of $u_{n}$, and $f$ has non negative real part. On the other hand, estimation of $I_{3}$ depends on Theorem 3.1 and Proposition 3.4.  

\noindent Using  the inequality $|\hbox{\rm exp}(-x)-\hbox{\rm exp}(-y)|\leq |x-y|$ for $x,y$  with non negative real parts we get $I_{1}(n)\leq \displaystyle\mathop{\Sigma}_{kr+1}^{n} \mathbb E_{\rho} (f_{i,n})$. Let $\delta>0$ be as above such that $f(t,x)=0$ for $t\in [0,1]$,  $|x|<\delta$, and observe that $n-kr<r$. Then the above bound for $I_{1}(n)$ gives :

$I_{1}(n)\leq r_{n} |f|_{\infty} \mathbb P_{\rho} \{u_{n}^{-1}|X_{1}|\geq \delta\}$.

\noindent Since $\displaystyle\mathop{\lim}_{n\rightarrow \infty} n^{-1} r_{n}=0$, the definition of $u_{n}$ and Theorem 2.1 give $\displaystyle\mathop{\lim}_{n\rightarrow \infty} I_{1}(n)=0$. Also $I_{1}(n)$ is bounded by $n^{-1} r_{n}$, up to a coefficient depending only on $m,\delta$. For $I_{2}(n)$, a similar argument involving each interval $J_{j}$ and the subinterval $J'_{j}$ gives :

$I_{2}(n)\leq k_{n} \ell_{n} |f|_{\infty} \mathbb P_{\rho} \{u_{n}^{-1} |X_{1}|\geq \delta\}$.

\noindent Using $k_{n} r_{n}\leq n$ we get $\displaystyle\mathop{\lim}_{n\rightarrow \infty} n^{-1} k_{n}\ell_{n}\leq  \displaystyle\mathop{\lim}_{n\rightarrow \infty} r_{n}^{-1} \ell_{n}$, i.e. $\displaystyle\mathop{\lim}_{n\rightarrow \infty} I_{2}(n)=0$ if $\displaystyle\mathop{\lim}_{n\rightarrow \infty} r_{n}^{-1} \ell_{n}=0$. 

\noindent Also we can bound $I_{2} (n)$ by $r_{n}^{-1} \ell_{n}$, up to a coefficient depending only on $m,\delta$.

\noindent For $I_{4}(n)$, we use the inequality $|\displaystyle\mathop{\Pi}_{1}^{n} z_{j}-\displaystyle\mathop{\Pi}_{1}^{n} w_{j}|\leq \displaystyle\mathop{\Sigma}_{1}^{n}|z_{j}-w_{j}|$ if $|z_{j}|$ and $|w_{j}|$ are less than 1. Hence : 

$I_{4}(n) \leq \displaystyle\mathop{\Sigma}_{j=1}^{k} |\mathbb E_{\rho}(\hbox{\rm exp}(-\displaystyle\mathop{\Sigma}_{1}^{r-\ell} f_{i,n}^{jr}))-\mathbb E_{\rho}(\hbox{\rm exp}(-\displaystyle\mathop{\Sigma}_{1}^{r} f_{i,n}^{jr}))|\leq |f|_{\infty} k_{n} \ell_{n} \mathbb P_{\rho} \{|X_{1}|>\delta u_{n}\}$

\noindent As above we get $\displaystyle\mathop{\lim}_{n\rightarrow \infty} I_{4}(n)=0$ if $\displaystyle\mathop{\lim}_{n\rightarrow \infty} r_{n}^{-1} \ell_{n}=0$, and a bound for $I_{4}(n)$ of the same form as for $I_{2} (n)$.

\noindent The estimation of $I_{3} (n)$ is more delicate and depends on  Lemma  3.6 below. We begin with the inequality : $I_{3} (n)\leq D(n)+I_{5}(n)+I_{3}(n-r_{n})$ where 

$D(n)=|\mathbb E_{\rho} (\hbox{\rm exp}(-\displaystyle\mathop{\Sigma}_{j=1}^{k}\ \ \displaystyle\mathop{\Sigma}_{(j-1) r+1}^{jr-\ell} f_{i,n}))-\mathbb E_{\rho} (\hbox{\rm exp}(-\displaystyle\mathop{\Sigma}_{i=1}^{r-\ell} f_{i,n}))
\mathbb E_{\rho}(\hbox{\rm exp}(-\displaystyle\mathop{\Sigma}_{j=2}^{k}\ \ \displaystyle\mathop{\Sigma}_{(j-1) r+1}^{jr-\ell} f_{i,n}))|$,

$I_{5}(n)=|\mathbb E_{\rho}(\hbox{\rm exp}(-\displaystyle\mathop{\Sigma}_{1}^{r-\ell} f_{i,n})\mathbb E_{\rho} (\hbox{\rm exp}(-\displaystyle\mathop{\Sigma}_{j=2}^{k}\ \ \displaystyle\mathop{\Sigma}_{(j-1) r+1}^{jr-\ell} f_{i,n}))-\mathbb E_{\rho} (\hbox{\rm exp}(-\displaystyle\mathop{\Sigma}_{1}^{r-\ell} f_{i,n}^{r}))$

$ \mathbb E_{\rho}(\hbox{\rm exp}(-\displaystyle\mathop{\Sigma}_{j=2}^{k}\ \ \displaystyle\mathop{\Sigma}_{(j-1) r+1}^{jr-\ell} f_{i,n}))|$,

$I_{3}(n-r)=|\mathbb E_{\rho}(\hbox{\rm exp}(-\displaystyle\mathop{\Sigma}_{j=2}^{k}\ \ \displaystyle\mathop{\Sigma}_{(j-r) r+1}^{jr-\ell} f_{i,n}))-\displaystyle\mathop{\Pi}_{j=2}^{k}\mathbb E_{\rho} (\hbox{\rm exp}(- \displaystyle\mathop{\Sigma}_{1}^{r-\ell} f_{i,n}^{jr}))|$.

\noindent Using as above the inequality $|\hbox{\rm exp}(-x)-\hbox{\rm exp}(-y)|\leq |x-y|$, and $ Re (f)\geq 0$ we get :

$I_{5}(n) \leq |\mathbb E_{\rho}(\displaystyle\mathop{\Sigma}_{1}^{r-\ell} f_{i,n})-\mathbb E_{\rho}(\displaystyle\mathop{\Sigma}_{1}^{r-\ell} f_{i,n}^{r})|$.

\noindent Since $f$ is Lipschitz we have, for $t', t''$ in $[0,1]$, $x\in V\setminus \{0\}$ :
$|f(t',x)-f(t'',x)|\leq k(f) |t'-t''|$.

\noindent Since $|n^{-1}i-n^{-1} r_{n}|\leq  n^{-1} r_{n}$ we have 

$I_{5}(n)\leq (r_{n}-\ell_{n}) n^{-1} r_{n} k(f) \mathbb P_{\rho} \{u_{n}^{-1}|X_{1}| \geq \delta\}\leq r^{2}_{n} n^{-1} k(f) \mathbb P_{\rho}\{|X_{1}| \geq \delta u_{n}\}$.

\noindent Using Theorem 2.1 we get $ I_{5} (n)\leq C n^{-2} r_{n}^{2}$ with a constant $C$ depending on $k(f)$ and $\delta$.

\vskip 3mm
\noindent In order to estimate $D(n)$ we consider the family of operators $P_{i,n}$ on the space $\mathcal H_{\chi,\varepsilon,\kappa}$ with $\chi, \varepsilon,\kappa$ as in Proposition 3.4, defined by $P_{i,n} \varphi (x)=\mathbb E ((\hbox{\rm exp}(-f_{i,n} (\omega))) \varphi (X_{i}^{x}))$ and the function $\psi_{n}$ defined by $\psi_{n}(y)=\mathbb E \{\hbox{\rm exp}(-\displaystyle\mathop{\Sigma}_{j=2}^{k}\ \ \displaystyle\mathop{\Sigma}_{i=(j-1) r+1}^{jr-\ell} f_{i,n}^{i+r}) / X_{r}^{x}=y\}$. Since, $u_{n}\geq 1$, for $n$ large with $m=|f|_{\infty}$, $\gamma=k(f)$, the functions $f_{i,n}$ satisfy $|f_{i,n}|_{\infty}\leq m$, $k(f_{i,n})\leq \gamma$, hence the operators $P_{i,n}$ belong to $O(m,\gamma)\subset End \mathcal H_{\chi,\varepsilon,\kappa}$. With the above notations, the products of operators $P_{i,n}$ belong to $\widehat{O}(m,\gamma)$. Also, using Proposition 3.4 we know  that on $\mathcal H_{\chi,\varepsilon,\kappa}$ we can write $P=\rho\otimes 1+U$ where $U$ has spectral radius $r(U)$ less then 1 and $U$ commutes with the projection $\rho \otimes 1$. We note also that for $f$ as above and $\psi \in \mathcal H_{\chi,\varepsilon,\kappa}$ we have :

\centerline{$|\rho (P^{f} \psi)| \leq \rho (P|\psi|)= \rho (|\psi|)\leq\|\psi\|$}

\noindent Then Lemma 3.6 below implies the convergence of $D(n)$ to zero with speed.

\vskip 3mm

\noindent Now, in order to prove the proposition, we are left to show $\displaystyle\mathop{\lim}_{n\rightarrow \infty} I_{3} (n)=0$. We iterate $k_{n}$ times the inequality : $I_{3}(n)\leq D(n)+I_{5}(n)+I_{3}(n-r_{n})$. We get,  using Lemma 3.6 :
 
\noindent  $I_{3}(n)\leq I_{3}(n-r_{n})+ C'(f) (n^{-2} r_{n}^{2}+r^{\ell_{n}}_{1} (U))\leq C'(f)  ( k_{n} r_{1}^{\ell_{n}}(U)+n^{-1}r_{n})$, 

\noindent with $C'(f)\geq 1$, depending on $m,\gamma$. Since $r_{n}=o(n)$, it remains to choose $\ell_{n}$ such that $\ell_{n}=o(r_{n})$ 
with $\displaystyle\mathop{\lim}_{n\rightarrow \infty} k_{n}r^{\ell_{n}}_{1}(U)=0$. 
 These conditions can be written as 
 
 $\displaystyle\mathop{\lim}_{n\rightarrow \infty} r_{n}^{-1} \ell_{n}=0$, $\displaystyle\mathop{\lim}_{n\rightarrow \infty} r_{n}^{-1} n r^{\ell_{n}}_{1}(U)=0$. 
 
 \noindent  The choice of $\ell_{n}$ with the above properties is possible since :
 
 $r_{1}(U)<1 , \displaystyle\mathop{\lim}_{n\rightarrow \infty} n^{-1} r_{n}=0$ and $\displaystyle\mathop{\lim}_{n\rightarrow \infty} (\hbox{\rm log} n)^{-1}r_{n}=\infty$.  
 
 \noindent One can take $\ell_{n}<r_{n}$ with $(\hbox{\rm log} n)^{-1}\ell_{n}=\infty$. The above estimations of $I_{1}, I_{2}, I_{3}, I_{4}, I_{5}$ give bounds by $\sup(n^{-1} r_{n}, r_{n}^{-1})$, up to a coefficient depending on $\delta,m,\gamma$ only.
 \end{proof}

\begin{lem} 
 There exist positive numbers $C_{1}(U)$, $r_{1}(U)\in ]r(U),1[$ and $C(f)$ depending only of $m,\gamma$ such that, for $n\in \mathbb N$ and $\ell_{n} <r_{n}$, $D(n)$, as above : 

$D(n)=|\rho (P_{1,n} \cdots P_{ r_{n} -\ell_{n},n}
 U^{\ell_{n}} \psi_{n})|\leq C_{1}(U) C(f) (r_{1}(U))^{\ell_{n}}$.
\end{lem}

\begin{proof} We observe that Markov's property implies $\mathbb E (e^{-f (X_{1}^{x})} g(\omega))=P^{f}(\mathbb E (g(\omega)))$ where $f$  is as above,  $g(\omega)$ is a  function depending  on $\omega$ throught the random variables $X_{k}^{x}  (k\geq 1)$  and  $\mathbb E (g(\omega))$ is a function of $x$. We apply this property to $\mathcal H_{\chi,\varepsilon,\kappa}$ with $f=f_{i,n}$ $(1\leq i \leq r-\ell)$ or $f=0$, $g=\psi_{n}$ as above, hence writing $P^{\ell}=\rho \otimes 1+U^{\ell}$ and 

$D(n)=|\rho (P_{1,n}\cdots P_{r-\ell,n} P^{\ell} \psi_{n})-\rho (P_{1,n}\cdots P_{r-\ell,n}1) \rho (\psi_{n})|=|\rho (P_{1,n}\cdots P_{r-\ell,n}U^{\ell}
 \psi_{n})|$,
 
 \noindent  Proposition 3.4 implies the existence of $C_{1}(U) <\infty, r_{1}(U) \in ]r(U), 1[$ with 
 
\centerline {$\|U^{\ell} \psi_{n}\|\leq C_{1}(U) r_{1}^{\ell} (U) \|\psi_{n}\|$. }
 
 \noindent On the other hand, since  $\psi_{n}$ is of the form $\psi_{n}=M1$ with $M\in \widehat{O}(m,\gamma)$ we have, using   Theorem 3.1,  $\|\psi_{n}\|\leq C(f)$ with $C(f)$ depending on $m,\gamma$. It follows $D(n)\leq C_{1}(U) C(f) (r_{1}(U))^{\ell_{n}}$.
 \end{proof}
 \vskip 4mm
\section{Asymptotics of exceedances processes}

\subsection{\sl  Statements of results}

 Let $E$ be a complete separable metric space which is locally compact, $M_{+}(E)$ the space of positive Radon measures on $E$, $M_{p} (E)$  its subspace of point measures, $\mathcal C_{+}^{c} (E)$ (resp $\mathcal L_{+}^{c} (E))$ the space of non negative and compactly supported continuous (resp Lipschitz) functions. Then it is well known that the vague topology on $M_{+} (E)$ is given by a metric and with respect to this metric,  $M_{+}
  (E)$ is a complete separable metric space. Furthermore this metric is constructed (see \cite{30} Lemma 3.11, Proposition 3.17) using a countable family  $(h_{i})_{i\in I}$ of functions in $\mathcal L_{+}^{c} (E)$ and $M_{p}(E)$ is a closed subset of $M_{+} (E)$. It follows that, in various situations with respect to weak convergence of random measures, $\mathcal C_{+}^{c} (E)$ can be replaced by $\mathcal L_{+}^{c} (E)$.
  
 \noindent Below,  assuming condition (c-e), we describe the asymptotics of the space-time exceedances process $N_{n}=\displaystyle\mathop{\Sigma}_{i=1}^{n} \varepsilon_{(n^{-1}i, u_{n}^{-1} X_{i})}$ under the probability $\mathbb P_{\rho}$ and we state a few corollaries. The results are formally analogous to  results for stationary processes proved in  (\cite{2}, \cite{3}) under general conditions. Here however, corresponding conditions have been proved in sections 2, 3 for  the affine random walk hence the results described below are new for  affine random walks but the scheme of the proofs is given in \cite{2}, \cite{3}.
 
 \noindent It is convenient to express the Laplace formulae below in terms of the  occupation measure $\pi_{v}^{\omega}=\displaystyle\mathop{\Sigma}_{0}^{\infty} \varepsilon_{S_{n}(\omega)v}$ of the linear random walk $S_{n}(\omega) v$ on $V\setminus \{0\}$. We observe that these  formulae depend only of the linear part of the affine random walk ; this is a consequence of the choice of the normalization by $u_{n}$.
 
 \noindent We denote by $\displaystyle\mathop{\Sigma}_{i\geq 0} \varepsilon_{T_{i}^{\delta}}$  the homogeneous point Poisson process on $[0,1]$ with intensity $p(\delta)=\theta \delta^{-\alpha}$ and  by $\displaystyle\mathop{\Sigma}_{j\in \mathbb Z} \varepsilon_{Z_{ij}}$ $(i\geq 0)$  an i.i.d. collection of copies of the cluster process $C=\displaystyle\mathop{\Sigma}_{j\in \mathbb Z} \varepsilon_{Z_{j}}$  described in Proposition 2.6, independant  of  $\displaystyle\mathop{\Sigma}_{i\geq 0} \varepsilon_{T_{i}^{\delta}}$.  Since we have $|X_{n}^{x}-X_{n}^{y}|\leq |S_{n}| |x-y|$ and $\displaystyle\mathop{\lim}_{n\rightarrow \infty} |S_{n}|=0$, $\mathbb P-$a.e. it is possible to replace $\mathbb P_{\rho}$ by $\mathbb P$ and $X_{n}$ by $X_{n}^{x}$ with $x$ fixed, in the statements.  We give the corresponding proof for the logarithm law only. 
 \vskip 3mm
\begin{thm}
 The sequence of normalized space-time point processes $N_{n}=\displaystyle\mathop{\Sigma}_{i=1}^{n} \varepsilon_{(n^{-1} i, u_{n}^{-1}X_{i})}$ on the space $[0,1]\times (V\setminus \{0\})$ converges weakly to a point process $N$. For any $\delta>0$, the law of the restriction of $N$ to $[0,1]\times U'_{\delta}$ is the same as the law of the point process on $[0,1]\times U'_{\delta}$ given by :
 $$\displaystyle\mathop{\Sigma}_{i\geq 0}\ \displaystyle\mathop{\Sigma}_{j\in \mathbb Z} \varepsilon_{(T_{i}^{\delta}, \delta Z_{ij})}\ 1_{\{|Z_{ij}|>1\}}.$$
 If $\eta$ denotes the law of $N$, $f\in \mathcal C_{+}^{c} ([0,1]\times U'_{\delta})$ and $\psi_{\eta}$ is the Laplace functional of $\eta$, then $-\hbox{\rm log} \psi_{\eta} (f)$ is equal to
  
\noindent $\theta \delta^{-\alpha}\displaystyle{\int}_{0}^{1} \widehat{\mathbb E}_{\Lambda_{1}} (1-\hbox{\rm exp}(-\displaystyle\mathop{\Sigma}_{j\in \mathbb Z} f(t, \delta Z_{j})))dt=  \displaystyle{\int}_{0}^{1} \mathbb E_{\Lambda_{0}} (\hbox{\rm exp}\  f_{t}(v)-1) \hbox{\rm exp}(-\pi_{v}^{\omega} (f_{t})) dt$
 
 \noindent where $f_{t} (x)=f (t,x)$
 \end{thm}
 \vskip 4mm
 \noindent Assuming the mixing and anticlustering conditions for compactly supported continuous functions, this statement was proved in \cite{3}. Here we will use Propositions 2.6, 3.4 and point process theory. We observe that, due to the normalization by $u_{n}$ the law of $N$  depends only of $\mu, \Lambda_{0}$.
 
\noindent Now as a consequence of Theorem 4.1, the mixing property stated in Proposition 3.5 for Lipschitz functions can be extended to compactly supported continuous functions.  Then, in particular, the  mixing conditions  $\mathcal A (u_{n})$ and $\mathcal A'(u_{n})$ of (\cite{2}, \cite{3})  are valid  here and the basic conditions of extreme value theory (see \cite{8}) are satisfied in our context.
 
 \vskip 3mm
\begin{cor}
 With the notation of Proposition 3.5, assume $f$ is a continuous compactly supported function on $[0,1]\times (V\setminus\{0\})$. Then we have the convergence $\displaystyle\mathop{\lim}_{n\rightarrow \infty} I_{n}(f)=0$.
 \end{cor}
\vskip 3mm
 \noindent Since the space exceedances process $N_{n}^{s}=\displaystyle\mathop{\Sigma}_{i=1}^{n} \varepsilon_{u_{n}^{-1}X_{i}}$ is the projection of $N_{n}$ on $V\setminus \{0\}$ we have the 
\vskip 3mm

\begin{cor}
 The normalized space exceedances process $N_{n}^{s}$ converges weakly to a point process $N^{s}$. The law of the restriction of $N^{s}$ to $U'_{\delta}$ is the same as the law of the point process 

$$\overline{Q}^{\delta}=\displaystyle\mathop{\Sigma}_{i=0}^{T^{\delta}} \displaystyle\mathop{\Sigma}_{j\in \mathbb Z}\varepsilon_{\delta Z_{ij}} 1_{\{|Z_{ij}|>1\}}$$
where $T^{\delta}$ is a Poisson random variable with mean $p(\delta)=\theta \delta^{-\alpha}$, independant of $Z_{ij}$ for $i\geq0$, $j\in \mathbb Z$. 

\noindent The Laplace functional of $N^{s}$ is given in logarithmic form by 
$$-\mathbb E_{\Lambda_{0}} [(\hbox{\rm exp} f(v)-1) \hbox{\rm exp}(-\pi_{v}^{\omega} (f))].$$
\end{cor}
Assuming the mixing and anticlustering conditions for continuous functions, this statement  was proved in \cite{2}, using the formula for Laplace functionals in Proposition 2.6.
\vskip 3mm
\noindent  We consider the $\mathbb N$-valued random variable $\zeta=\pi_{v}^{\omega} (U'_{1})$ and we write $\zeta_{k}=\mathbb Q_{\Lambda_{1}} \{\zeta=k\}$  for $k\geq 1$ ; in particular we have $\zeta_{1}=\theta$, $\zeta_{k}\geq \zeta_{k+1}$.
\vskip 3mm

\begin{cor}
 The  sequence of normalized time exceedances process $N_{n}^{t}=\displaystyle\mathop{\Sigma}_{i=1}^{n} \varepsilon_{ n^{-1 } i} 1_{\{|X_{i}|>u_{n}\}}$ converges weakly $(n\rightarrow \infty)$ to the homogeneous compound Poisson process $N^{t}$ on $[0,1]$ with intensity $\theta$, and cluster probabilities $\nu_{k} (k\geq 1)$ where $\nu_{k}=\theta^{-1} (\zeta_{k}-\zeta_{k+1})$. 
 \end{cor}
\vskip 3mm
\noindent Under special hypotheses, including density of the law of $B_{n}$ with respect to Lebesgue measure, this statement was proved in \cite{22}. 
\vskip 1mm
\noindent Fr\'echet's law for $M_{n}^{x}=\sup \{|X_{k}^{x}| ; 1\leq k\leq n\}$ is a simple consequence of Corollary 4.4 as follows.
\vskip 3mm

\begin{cor}

\noindent For any $x\in V$ and $t>0$ we have the convergence in law of $u_{n}^{-1} M_{n}^{x}$ to Fr\'echet's law $\Phi_{\alpha}^{\theta}$,

\centerline{$\displaystyle\mathop{\lim}_{n\rightarrow \infty} \mathbb P \{u_{n}^{-1}M_{n}^{x}<t\}=\hbox{\rm exp}(-\theta t^{-\alpha})=\Phi_{\alpha}^{\theta} ([0,t])$}

\noindent with $\theta=\mathbb Q_{\Lambda_{1}}\{\displaystyle\mathop{\sup}_{n\geq1} |S_{n}(\omega)v|\leq 1\}$. Furthermore the  law of the normalized hitting time $t^{-\alpha} \tau_{t}^{x}$ of $U'_{t}$ by the process $|X_{n}^{x}|$  converges to the exponential law with parameter $c\theta$, i.e.

\centerline{$\displaystyle\mathop{\lim}_{t\rightarrow \infty} \mathbb P\{t^{-\alpha} \tau_{t}^{x} >u\}=\hbox{\rm exp}(-c\theta u)$.}
\end{cor}
\noindent For $d=1$, another proofs of Laplace formulae in 4.3, 4.5 were given in \cite{5} (see 3.1.1, 3.2.1)

\vskip 3mm
\noindent It was observed in \cite{28} that Sullivan's logarithm law for excursions of geodesics around the cusps of hyperbolic manifolds (see \cite{36}), in the case of the modular surface, is a consequence of Fr\'echet's law for the continuous fraction expansion of a real number uniformly distributed in $[0,1]$ (see \cite{27}).   Here, in this vein, we have the following logarithm law.
\vskip 3mm
e
\begin{cor}
For any $x\in V$, we have the $\mathbb P-$a.e. convergence
\vskip 2mm
$\displaystyle\mathop{\lim\sup}_{n\rightarrow \infty} \frac{\hbox{\rm log} |X_{n}^{x}|}{\hbox{\rm log} n}=\frac{1}{\alpha}=\displaystyle\mathop{\lim\sup}_{n\rightarrow \infty} \frac{\hbox{\rm log}\ M_{n}^{x}}{\hbox{\rm log}\ n}$.
\end{cor}
\vskip 3mm
\noindent If $x$ is random, we observe that a logarithm law and a modified Fr\'echet law have been obtained in \cite{21} for random walks on some homogeneous spaces of arithmetic character, using $\mathbb L^{2}$-spectral gap methods.

\noindent Given a Borel subset $A$ of $U'_{1}$ and a real number $t>1$, we can also consider the hitting time $\tau_{tA}^{x}$ of the dilated set $tA$ under by the process $X_{n}^{x}$ (see \cite{35} p. 290). In the context of collective risk theory this  hitting time can be interpreted as a ruin time associated to the entrance of $X_{n}^{x}$ in  the set $t A$  (see \cite{6}). We observe that, in contrast to   the associated linear random walk, the event of ruin (in infinite time) occurs here  with probability 1, due to the finiteness of the stationary measure $\rho$. However, in law, the ruin scenario is the same : excursion at infinity of the associated linear random walk. 

\noindent Then the convergence to the point process $N^{s}$ stated in Corollary 4.3 gives the 
\begin{cor} Let $A$ be a Borel subset of $V$ such that $\overline{A} \subset U'_{1}$,   $\Lambda (A)>0$  $\Lambda (\partial A)=0$. Then for any $x\in V$, the normalized hitting time $t^{-\alpha} \tau_{t A}^{x}$ of the set $t A$ converges in law to the exponential law with parameter $c \theta (A)$ with 

$\theta(A)=(\Lambda_{0} (A))^{-1} \mathbb Q_{\Lambda_{0}} \{\displaystyle\mathop{\Sigma}_{1}^{\infty} 1_{A} (S_{i}y)=0, \ y\in A\} \in ]0,\theta]$.
\end{cor}

\noindent We observe that, in the notation of \cite{35}, the quantity $\mathbb Q\{\displaystyle\mathop{\Sigma}_{1}^{\infty} 1_{A} (S_{i} y)=0\}$ is, for $y\in A$, the escape probability from $A$ for the linear random walk $S_{i} y$ on $V\setminus \{0\}$. Hence the number $\gamma(A)=\mathbb Q_{\Lambda_{0}} \{\displaystyle\mathop{\Sigma}_{1}^{\infty} 1_{A} (S_{i} y)=0$, $y\in A\} \in ]0, \Lambda_{0} (A)]$ is the corresponding capacity with respect to the $Q$-invariant measure $\Lambda_{0}$. The positivity of $\theta(A)$ is a simple consequence of the formula for $\theta(A)$ and of the dynamical argument used in the proof of Proposition 2.4 for the inequality $\theta>0$.

\noindent The inequality $\tau^{x}_{tA} \geq \tau_{t}^{x}$ implies $\theta (A) \leq \theta <1$. We observe that, in the context of statistics of hitting times for hyperbolic dynamical systems, convergence to an exponential law is also valid, but since almost every point is repulsive, the corresponding condition $\theta=1$ is then generically satisfied (see for example \cite{26}, \cite{32}). Here the property $\theta(A) \in ]0,1[$ is a consequence of the contraction-expansion property, which follows from the unboundness of the semigroup generated by supp$(\mu)$. Heuristically speaking, the sphere at infinity of $V$ is weakly attractive for the affine random walk.

 \vskip 3mm
\subsection{\sl Proofs of point process convergences}
 The proof of Theorem 4.1 will follow of three lemmas. 

\noindent We denote by $(X_{k,j})_{k\in \mathbb N}$ an i.i.d. sequence of copies of the process $(X_{j})_{j\in \mathbb N}$ and we write

\centerline{$\widetilde{N}_{k,n}=\displaystyle\mathop{\Sigma}_{j=1+(k-1)r_{n}}^{k r_{n}} \varepsilon_{(n^{-1} k r_{n}, u^{-1}_{n}X_{k,j})}, \widetilde{N}_{n}=\displaystyle\mathop{\Sigma}_{k=1}^{k_{n}} \widetilde{N}_{k,n}$,}

\noindent where $r_{n},\ k_{n}$ are as in section 2. 

For $k_{n}>0$ we denote by $\mathbb E_{\rho}^{(k_{n})}$ the expectation corresponding to the product probability of $k_{n}$ copies of $\mathbb P_{\rho}$.

\noindent If $f$ is a non negative and compactly supported Lipschitz function on $[0,1] \times V\setminus \{0\}$, we have, using independance :

$\mathbb E_{\rho}^{(k_{n})} (\hbox{\rm exp}(-\widetilde{N}_{n} (f)))=\displaystyle\mathop{\Pi}_{k=1}^{k_{n}} \mathbb E_{\rho} (\hbox{\rm exp}(-\displaystyle\mathop{\Sigma}_{j=1+(k-1) r_{n}}^{kr_{n}} f(n^{-1}k r_{n}, u^{-1}_{n} X_{k,j})))$.

\noindent This relation and the multiple mixing property in Proposition 3.5  show that, on functions $f$ as above, the asymptotic behaviour of the Laplace functionals of $N_{n}$ under $\mathbb E_{\rho} $, and $\widetilde{N}_{n}$ under $\mathbb E_{\rho}^{(k_{n})}$,  are the same. We begin by considering the convergence of $\mathbb E_{\rho}^{(k_{n})}(\hbox{\rm exp}( - \widetilde{N}_{n} (f)))$.

\noindent Lemma 4.8 below is a general statement giving the weak convergence of a sequence of random measures, using only the convergence of the values of the Laplace functionals on Lipschitz functions. 
 Lemmas 4.9, 4.10 are reformulations of parts of the proof of Theorem 2.3 in   \cite{3}, which was considered in a general setting. 

\vskip 3mm
\begin{lem}
 Let $E$ be a locally compact separable metric space. Let $\eta_{n}$ be a sequence of random measures on $E$  and, for $f$ non negative Lipschitz and compactly supported, assume  that the sequence of Laplace functionals $\psi_{\eta_{n}} (f)$ converges to $\psi (f)$ and $\psi(sf)$ is continuous at $s=0$ ; then the sequence $\eta_{n}$ converges weakly. A random measure $\eta$ on $E$,  is well defined by the values of its Laplace functional on Lipschitz functions.
 \end{lem}

\begin{proof}
 We begin by the last assertion and we use the family of Lipschitz functions $(h_{i})_{i\in I}$ considered in the above subsection. If the random measures $\eta, \eta'$   satisfy $\psi_{\eta} (f)=\psi_{\eta'} (f)$ for any $f\in \mathcal L_{+}^{c}(E)$ and $\lambda_{1}, \lambda_{2}, \cdots, \lambda_{p}$ are non negative numbers then we have $\psi_{\eta}(\displaystyle\mathop{\Sigma}_{i=1}^{i=p} \lambda_{i} h_{i})=\psi_{\eta'} (\displaystyle\mathop{\Sigma}_{i=1}^{i=p} \lambda_{i} h_{i})$. It follows that the random vectors $(\eta (h_{1}),\cdots , \eta (h_{p}))$ and $(\eta' (h_{1}), \cdots , \eta'(h_{p}))$ have the same Laplace transforms, hence the same laws. Furthermore, for rational numbers $r_{j}< r'_{j}$ the finite intersections of sets of the form $\{\mu \in  M_{+} (E), \mu (h_{i}) \in ]r_{j}, r'_{j}[)\}$ define a countable basic $\mathcal U$ of open subsets in $M_{+} (E)$ stable under finite intersection, hence a $\pi$-system (see \cite{30}). Then from above, $\eta$, $\eta'$ are equal on $\mathcal U$ ; since the sigma-field generated by $\mathcal U$ coincides with the Borel sigma-field, one has $\eta=\eta'$.

\noindent We observe that, if a sequence of random measures $\eta_{n}$ is such that for any $f\in \mathcal L_{+} (E)$ the sequence of real random variables $\eta_{n} (f)$ is tight, then the sequence $\eta_{n}$ itself is tight. This follows for a corresponding result in \cite{30} for $f\in \mathcal C_{+}^{c} (E)$ since any such $f$ is dominated by an element of $\mathcal L_{+}^{c} (E)$.

\noindent Assuming the convergence of $\psi_{\eta_{n}} (f)$ to $\psi_{\eta} (f)$ for any $f\in \mathcal L_{+}^{c} (E)$ and the continuity at $s=0$ of $\psi_{\eta}(sf)$, we get that $\psi_{\eta} (sf)$ is the Laplace transform of the real random variable $\eta(f)$, hence the convergence of  the sequence $\eta_{n} (f)$ to $\eta (f)$ for any $f\in \mathcal L_{+}^{c} (E)$. From above and the continuity hypothesis of $\psi(sf)$ at $s=0$, we get that the sequence $\eta_{n}$ is tight. If $\eta_{n_{i}}$ is a subsequence converging weakly to the random measure $\eta$ we have the convergence of the corresponding Laplace functionals for any $f\in \mathcal L_{+}^{c} (E)$. Since such a limit is independant of the subsequence, we get from above that two possible weak limits of random measures are equal. Hence the sequence $\eta_{n}$ converges weakly to $\eta$.
\end{proof}

\begin{lem}
Let $f$ be a non negative and compactly supported continuous function on $[0,1] \times U'_{\delta}$ and let $\displaystyle\mathop{\Sigma}_{j\in \mathbb Z} \varepsilon_{Z_{j}}$ be the cluster process for the affine random walk $(X_{k})_{k\in \mathbb N}$. Then :\\

a) $\displaystyle\mathop{\lim}_{n\rightarrow \infty}[\hbox{\rm log} \mathbb E_{\rho}^{(k_{n})}(\hbox{\rm exp}( -\widetilde{N}_{n} (f)))+\displaystyle\mathop{\Sigma}_{k=1}^{k_{n}} (1-\mathbb E_{\rho} (\hbox{\rm exp}(- \widetilde{N}_{k,n} (f))))]=0$.

b) $\displaystyle\mathop{\lim}_{n\rightarrow \infty} \displaystyle\mathop{\Sigma}_{k=1}^{k_{n}} (1-\mathbb E_{\rho} (\hbox{\rm exp}(- \widetilde{N}_{k,n} (f))))= \theta \delta^{-\alpha} \int_{0}^{1} \widehat{\mathbb E}_{\Lambda_{1}}(1-\hbox{\rm exp}( -\displaystyle\mathop{\Sigma}_{j\in \mathbb Z} f (t, \delta Z_{j}))) dt$.
\end{lem}

\begin{lem}
 Let $\displaystyle\mathop{\Sigma}_{i\geq 0} \varepsilon_{T_{i}^{\delta}}$ be a homogeneous Poisson process of intensity $p(\delta)>0$ on $[0,1]$, which is independant of the sequence of cluster processes $\displaystyle\mathop{\Sigma}_{j\in \mathbb Z} \varepsilon_{Z_{ij}}$.

\noindent Then for any non negative and compactly supported continuous function $f$ on $[0,1]\times U'_{\delta}$, the Laplace functional of the point process 
$Q^{\delta}=\displaystyle\mathop{\Sigma}_{i\geq0} \displaystyle\mathop{\Sigma}_{j\in \mathbb Z} \varepsilon_{(T_{i}^{\delta}, \delta Z_{ij})} 1_{\{|Z_{ij}|>1\}}$ restricted to 
$[0,1]\times U'_{\delta}$ satisfies :

\centerline{$\hbox{\rm log} \psi^{\delta} (f)= -p(\delta) \int_{0}^{1} \widehat{\mathbb E}_{\Lambda_{1}} (1-\hbox{\rm exp}(- \displaystyle\mathop{\Sigma}_{j\in \mathbb Z} f (t,\delta Z_{j})))dt$}
\end{lem}

\begin{proth} \textbf{4.1} 
 Let $f$ be a non negative and compactly supported Lipschitz function on $[0,1] \times U'_{\delta}$. 
 Using Proposition 3.5, Lemma 4.9  implies that, on such functions the Laplace functionals of $N_{n}$ and $\widetilde{N}_{n}$ have the same limit, namely

\centerline{$\psi^{\delta} (f)=\hbox{\rm exp}[-p(\delta) \int_{0}^{1} \mathbb E_{\Lambda_{1}} (1-\hbox{\rm exp}- \displaystyle\mathop{\Sigma}_{j\in \mathbb Z} f (t, \delta Z_{j}))dt].$}

\noindent We observe that, for fixed $f$ as above, the function $s\rightarrow \psi^{\delta} (sf)$ is continuous at $s=0$.  Since the function $s\rightarrow \psi_{n} (sf)=\mathbb E_{\rho} (\hbox{\rm exp}(-s N_{n} (f)))$ is the Laplace transform of the non negative random variable $N_{n} (f)$, the continuity theorem for Laplace transforms implies that the sequence $N_{n} (f)$ converges in law to some random variable. Since  the  sequence of Laplace functionals $\psi_{n} (f)$ converges to $\psi^{\delta} (f)$, Lemma 4.8 implies that there exists a unique point process $N$ on $[0,1]\times (V\setminus \{0\})$ such that the sequence $N_{n}$ converges weakly to $N$. As stated in Lemma  4.10, for $f$ as above the restriction of $N$ to $[0,1]\times U'_{\delta}$ is given by the point process formula in the theorem.  A density argument shows that the Laplace functional of $N$ on the function $f\in \mathcal C_{+}^{c} ([0,1]\times U'_{\delta}$ is equal to 

\centerline{$\psi^{\delta}(f)=\hbox{\rm exp}[-p (\delta) \int_{0}^{1} \widehat{\mathbb E}_{\Lambda_{1}} [1-\hbox{\rm exp}(-\displaystyle\mathop{\Sigma}_{j\in \mathbb Z} f (t,\delta Z_{j})) dt]$.}

\noindent The point process formula for $N$ follows, as well as the first part of the formula giving the Laplace functional of $N$. The second part is a consequence of the last formula in Proposition 2.6 applied to the function $v\rightarrow f(t, \delta v)$ and  of the $\alpha$-homogeneity of $\Lambda_{0}$. 
\end{proth}
\vskip 3mm

\begin{proco} \textbf{4.2}   
 The first term $\mathbb E_{\rho} (\hbox{\rm exp}(-\displaystyle\mathop{\Sigma}_{1}^{n} f_{i,n}))$ in $I_{n}(f)$ is the value of the Laplace functional of $N_{n}$ on the continuous function $f$. Hence  the weak convergence in Theorem 4.1 implies its convergence to the Laplace functional of $N$ on $f$. The same remark is valid for the second term in $I_{n}(f)$, if $N_{n}$ is replaced by $\widetilde{N}_{n}$ : the limit of $\widetilde{N}_{n}$ is also $N$, using Lemma 4.8 and Proposition 3.5. Then for any $f$ in $\mathcal C_{+}^{c} ([0,1]\times (V\setminus \{0\})$ we have : $\displaystyle\mathop{\lim}_{n\rightarrow \infty} I_{n}(f)=\displaystyle\mathop{\lim}_{n\rightarrow \infty} |\mathbb E_{\rho} (\hbox{\rm exp}( -N_{n}(f)))-\mathbb E_{\rho}^{(k_{n})}(\hbox{\rm exp}(-\widetilde{N}_{n} (f)))|=0$
 \end{proco}
 \vskip 3mm
\begin{proco} \textbf{4.3} 
 The point process $N_{n}^{s}$ is the projection of $N_{n}$ on $V\setminus \{0\}$. Since $[0,1]$ is compact and the projection is continuous, the continuous mapping theorem implies the required convergence, using the first part of Theorem 4.1. The formula for the Laplace functional of $N^{s}$ is a direct consequence of the second part in Theorem 4.1 applied to a function independent of $t$. 
 \end{proco}
 \vskip 3mm
\begin{proco} \textbf{4.4}  
 For $\varphi \in C_{+}^{c} ([0,1])$ we have $N_{n}^{t}(\varphi)=N_{n}(\varphi \otimes 1_{U'_{1}})$. Since the discontinuity set of $1_{U'_{1}}$ is $\Lambda$-negligible, Theorem 4.1 gives the convergence of $N_{n}^{t} (\varphi)$ to $N^{t} (\varphi)$. With $f=\varphi \otimes 1_{U'_{1}}$, the formula for the Laplace functional $\psi_{\eta} (f)$ of $N$ gives the Laplace functional $\psi_{\eta^{t}} (\varphi)$ of $N^{t}$ in the logarithmic form
 
 \centerline{$\hbox{\rm log} \ \psi_{\eta^{t}}  (\varphi)=-\theta \int_{0}^{1} \widehat{\mathbb E}_{\Lambda_{1}} [1-\hbox{\rm exp} - \varphi (x) \gamma] dt$.}
 
 \noindent The expression of the generating function of the random variable $\gamma=\displaystyle\mathop{\Sigma}_{j\in \mathbb Z} 1_{U'_{1}} (Z_{j})$ follows from the last formula in Proposition 2.6 :
 
 \centerline{$\displaystyle\mathop{\Sigma}_{1}^{\infty} e^{-sk}\nu_{k}=1-(e^{s}-1) \theta^{-1}\mathbb E_{\Lambda_{1}} [\hbox{\rm exp}(-s\ \pi_{v}^{\omega} (U'_{1}))]$.}
 
 \noindent Hence $\nu_{k}=\theta^{-1}(\zeta_{k}-\zeta_{k+1})$
 
 \noindent In view of Theorem 4.1, the point process $N^{t}$ can be written as $N^{t}=\displaystyle\mathop{\Sigma}_{k\geq 0} \gamma_{k} \varepsilon_{T_{k}^{1}}$, where the random variables $\gamma_{k}$ are i.i.d. with the same law as $\gamma$, hence $N^{t}$ coincides with the compound Poisson process described in the statement. 
 \end{proco}
 \vskip 3mm
\begin{proco} \textbf{4.5}     
 Replacing $u_{n}$ by $\delta u_{n}$ $(\delta>0)$ in Corollary 4.4, we see that the point process on $[0,1]$ given by $N^{t}_{n,\delta}=\displaystyle\mathop{\Sigma}_{k=1}^{n} \varepsilon_{n^{-1}k} 1_{\{|X_{k}|>\delta u_{n}\}}$ converges to $N_{\delta}^{t}=\displaystyle\mathop{\Sigma}_{k\geq 0} \gamma_{k} \varepsilon_{T_{k}^{\delta}}$ where $\displaystyle\mathop{\Sigma}_{k\geq 0} \varepsilon_{T_{k}^{\delta}}$ is the Poisson process on $[0,1]$ with intensity $\theta \ \delta^{-\alpha}$ and the $\gamma_{k}$ are i.i.d. random variables as in the proof of Corollary 4.4. It follows that for any $\delta>0$, 
 
 \centerline{$\displaystyle\mathop{\lim}_{n\rightarrow \infty} \mathbb P_{\rho} \{ N_{n,\delta}^{t} (1)=0\}=\hbox{\rm exp}(-\theta \ \delta^{-\alpha})$.}
 
 \noindent Since $\mathbb P_{\rho}\{N_{n,\delta}^{t} (1)=0\}=\mathbb P_{\rho} \{M_{n}\leq u_{n} \delta\}$, the convergence of $u_{n}^{-1} M_{n}$ to Fr\'echet's law follows.
 
 \noindent If $M_{n}$ is replaced by $M_{n}^{x}$ with  $x\in V$, the same proof as the one given below for the logarithm law remains valid. The last assertion in the corollary is a direct consequence of Fr\'echet's law. 
 \end{proco}
\vskip 3mm

\subsection{\sl  Proof of  logarithm's law}

 The proof of  logarithm's law is based on Fr\'echet's law and depends on two lemmas as follows.

\begin{lem} 
 We have $\mathbb P_{\rho}-$a.e. :

$\displaystyle\mathop{\lim\sup}_{n\rightarrow \infty} \frac{\hbox{\rm log} |X_{n}|}{\hbox{\rm log} n} \leq \displaystyle\mathop{\lim\sup}_{n\rightarrow \infty} \frac{\hbox{\rm log} M_{n}}{\hbox{\rm log} n} \leq \frac{1}{\alpha}$. Furthermore, for any $\varepsilon>0$, and for $n$ large we have $|X_{n}|\leq n^{1/\alpha+\varepsilon}$.
\end{lem}

\begin{proof} 
 Let $\varepsilon>0$, $\alpha_{n}(\varepsilon)=\{|X_{n}|\geq n^{1/\alpha+\varepsilon}\} \subset V^{\mathbb Z_{+}}$, $\alpha'_{n}(\varepsilon)=V^{\mathbb Z_{+}} \setminus \alpha_{n}(\varepsilon)$. Stationarity of $X_{n}$ implies $\mathbb P_{\rho} \{\alpha_{n}(\varepsilon)\}=\mathbb P_{\rho} \{|X_{0}|\geq n^{1/\alpha+\varepsilon}\}$. Since $\displaystyle\mathop{\lim}_{n\rightarrow \infty} n^{1+\alpha \varepsilon} \ell^{\alpha} (n^{1/\alpha+\varepsilon},\infty)=1$, with $\ell^{\alpha}(dt)=t^{-\alpha-1}dt$, Corollary 2.2 gives $\displaystyle\mathop{\Sigma}_{1}^{\infty} \mathbb P_{\rho} \{\alpha_{n} (\varepsilon)\} <\infty$. Then Borel-Cantelli's lemma implies that $\mathbb P_{\rho}\{\displaystyle\mathop{\cup}_{1}^{\infty} \displaystyle\mathop{\cap}_{j\geq n} \alpha'_{j} (\varepsilon)\}=1$, hence $\mathbb P_{\rho}-$a.e. there exists $n_{0} (\omega)$ such that for $n\geq n_{0}(\omega)$, $|X_{n} (\omega)|\leq n^{1/\alpha+\varepsilon}$. 

\noindent Then we deduce that $\mathbb P_{\rho}-$a.e. : $\displaystyle\mathop{\lim\sup}_{n\rightarrow \infty} \frac{\hbox{\rm log} M_{n}}{\hbox{\rm log} n} \leq \frac{1}{\alpha}+\varepsilon$. 
Since $\varepsilon$ is arbitrary we get : $\displaystyle\mathop{\lim\sup}_{n\rightarrow \infty} \frac{\hbox{\rm log} M_{n}}{\hbox{\rm log} n}\leq \frac{1}{\alpha}$.
\end{proof}
\vskip 3mm
\begin{lem} 
 We have $\mathbb P_{\rho}-$a.e. : $\displaystyle\mathop{\lim\sup}_{n\rightarrow \infty} \frac{\hbox{\rm log} |X_{n}|}{\hbox{\rm log} n} \geq \frac{1}{\alpha}$.
\end{lem}

\begin{proof} Let $\varepsilon \in ]0,1/\alpha[$, $\beta(\varepsilon)=\left\{\displaystyle\mathop{\lim\sup}_{n\rightarrow \infty} \frac{\hbox{\rm log} |X_{n}|}{\hbox{\rm log} n}\leq \frac{1}{\alpha}-\varepsilon \right\}$, $\beta_{n} (\varepsilon)=\left\{\displaystyle\mathop{\sup}_{j\geq n} \frac{\hbox{\rm log} |X_{j}|}{\hbox{\rm log} j}\leq \frac{1}{\alpha}-\frac{\varepsilon}{2}\right\}$. The sequence $\beta_{n}(\varepsilon)$ is increasing and $\beta(\varepsilon) \subset \displaystyle\mathop{\cup}_{2}^{\infty} \beta_{n}(\varepsilon)$. We are going to show $\mathbb P_{\rho}\{\beta_{n}(\varepsilon)\}=0$. For $p\geq n\geq 2$, $p\in \mathbb N$, we define $\beta_{n,p}(\varepsilon)=\{\displaystyle\mathop{\sup}_{n\leq j\leq p} |X_{j}| \leq p^{1/ \alpha-\varepsilon/2}\}$, hence $\beta_{n}(\varepsilon) \subset \beta_{n,p} (\varepsilon)$. Using stationarity we get $\mathbb P_{\rho} \{\beta_{n,p}(\varepsilon) \}\leq\mathbb P_{\rho}\{p^{-1/\alpha} M_{p-n+1}\leq p^{-\varepsilon/2}\}$. Also, using Corollary 4.5, we have $\displaystyle\mathop{\lim}_{n\rightarrow \infty}(\displaystyle\mathop{\sup}_{t>0} |\mathbb P_{\rho}\{n^{-1/\alpha} M_{n}\leq t\}-e^{- c\theta t^{-\alpha}}|)=0$ which gives $\displaystyle\mathop{\lim}_{p\rightarrow \infty} \mathbb P_{\rho}\{\beta_{n,p} (\varepsilon)\}=0$. 

\noindent Since  $\beta_{n}(\varepsilon)= \displaystyle\mathop{\cap}_{p\geq n} \beta_{n,p}(\varepsilon)$ we have for $n\geq 2$ : 

$\mathbb P_{\rho} \{\beta_{n}(\varepsilon) \}\leq \displaystyle\mathop{\lim}_{p\rightarrow \infty} \mathbb P_{\rho} \{\beta_{n,p} (\varepsilon)\}=0$, i.e. $\mathbb P_{\rho} \{\beta (\varepsilon)\}=0$.

\noindent We see that $\mathbb P_{\rho}-$a.e., $\displaystyle\mathop{\lim\sup}_{n\rightarrow \infty} \frac{\hbox{\rm log} |X_{n}|}{\hbox{\rm log} n} \geq \frac{1}{\alpha}-\varepsilon$, and, since $\varepsilon$ is arbitrary we conclude  $\displaystyle\mathop{\lim\sup}_{n\rightarrow \infty}  \frac{\hbox{\rm log} |X_{n}|}{\hbox{\rm log} n}\geq \frac{1}{\alpha}$. 
\end{proof}
\vskip 3mm
\begin{proco} \textbf{4.6}  
 From Lemmas 4.11, 4.12 we have $\mathbb P_{\rho}-$a.e.,

$\displaystyle\mathop{\lim\sup}_{n\rightarrow \infty} \frac{\hbox{\rm log} M_{n}}{\hbox{\rm log} n}\geq \displaystyle\mathop{\lim\sup}_{n\rightarrow \infty} \frac{\hbox{\rm log} |X_{n}|}{\hbox{\rm log} n}=\frac{1}{\alpha}$.

\noindent On the other hand, for $n$ large and $\varepsilon>0$, Lemma 4.11 gives  $\displaystyle\frac{\hbox{\rm log} M_{n}}{\hbox{\rm log} n}\leq \frac{1}{\alpha}+2\varepsilon$, hence $\displaystyle\mathop{\lim\sup}_{n\rightarrow \infty} \frac{\hbox{\rm log} M_{n}}{\hbox{\rm log} n}=\frac{1}{\alpha}$.

\noindent Then, for a set  of $\rho \otimes \mathbb P$- probability 1 in $V\times H^{\mathbb N}$ we have 

$\frac{1}{\alpha}=\displaystyle\mathop{\lim\sup}_{n\rightarrow \infty} \frac{\hbox{\rm log} |X_{n}| (\omega)}{\hbox{\rm log} n}=\displaystyle\mathop{\lim\sup}_{n\rightarrow \infty} \frac{\hbox{\rm log} M_{n}}{\hbox{\rm log} n}$, 

\noindent hence for a subsequence $n_{k}(\omega)$, $\frac{1}{\alpha}=\displaystyle\mathop{\lim}_{k\rightarrow \infty} \frac{\hbox{\rm log} |X_{n_{k}} (\omega)|}{\hbox{\rm log} n_{k}}$.

\noindent On the other hand we have for any $x\in V$ : $|X_{n}-X_{n}^{x}| \leq |S_{n}| | X_{0}-x|$ and  $\displaystyle\mathop{\lim}_{n\rightarrow \infty} |S_{n}|=0$, $\mathbb P-$a.e.   

\noindent Also $|\hbox{\rm log} |X_{n}|-\hbox{\rm log} |X_{n}^{x}|| \leq |S_{n}| |X_{0}-x| \sup (|X_{n}|^{-1}, |X_{n}^{x}|^{-1})$, hence for any $x\in V$, $\mathbb P_{\rho}-$a.e. : 

$\displaystyle\mathop{\lim}_{n\rightarrow \infty} |\hbox{\rm log} |X_{n}|- \hbox{\rm log} |X_{n}^{x}||=0$.

\noindent It follows $\displaystyle\mathop{\lim}_{k\rightarrow \infty} \frac{\hbox{\rm log} |X^{x}_{n_{k}}|}{\hbox{\rm log} n_{k}}=\frac{1}{\alpha}$, and $\mathbb P-$a.e., $\displaystyle\mathop{\lim\sup}_{n\rightarrow \infty} \frac{\hbox{\rm log} |X_{n}^{x}|}{\hbox{\rm log} n}\geq \frac{1}{\alpha}$.

\noindent A similar argument shows that $\mathbb P-$a.e., $\displaystyle\mathop{\lim\sup}_{n\rightarrow \infty} \frac{\hbox{\rm log} |X_{n}^{x}|}{\hbox{\rm log} n} \leq \frac{1}{\alpha}$.

\noindent Furthermore, for any $n\geq 1$, $x\in V$ :

$|M_{n}^{x}-M_{n}| \leq \displaystyle\mathop{\sup}\{|S_{k}| ; 1\leq k\leq n\} |x-X_{0}|$

\noindent where the sequence on the right is $\mathbb P-$a.e. bounded. Then, it follows that, $\mathbb P-$a.e. : 

$\displaystyle\mathop{\lim\sup}_{n\rightarrow \infty} \frac{\hbox{\rm log} M_{n}^{x}}{\hbox{\rm log} n}=\frac{1}{\alpha}$. 
\end{proco}
\vskip 3mm
\subsection{\sl Proof of the normalized hitting time convergence}

The proof depends on the following lemma.
\begin{lem} For any $\lambda\geq 0$, we have 

$\hbox{\rm log} \mathbb E (\hbox{\rm exp}(-\lambda N^{s}(A)))=-\int (1-\hbox{\rm exp}(-\lambda 1_{A}(y))) \mathbb E (\hbox{\rm exp}(-\lambda \displaystyle\mathop{\Sigma}_{1}^{\infty} 1_{A} (S_{i} y))) d\Lambda_{1}(y)$. 

In particular we have 

$\mathbb Q_{\Lambda_{1}}\{N^{s} (A)=0\}=\hbox{\rm exp}(-\mathbb Q_{\Lambda_{1}} \{\displaystyle\mathop{\Sigma}_{1}^{\infty} 1_{A} (S_{i} y)=0, \ y\in A)\}$ and $\mathbb Q_{\Lambda_{1}} \{N^{s} (\partial A)=0\}=1$.
\end{lem}

\begin{proof} Corollary 4.3 says that the above logarithmic formula is valid if $1_{A}$ is replaced by an arbitrary function $f$ in $C_{+}^{c} (U'_{1})$. We observe that if $A\subset U'_{1}$ is compact, then $1_{A}$ is the decreasing limit of a sequence in $C_{+}^{c} (U'_{1})$, hence dominated convergence implies the validity of the formula in this case. Also, if $A$ is bounded, $\partial A$ is compact, hence the formula is valid for $\partial A$. Taking the limit in the formula  at $\lambda=\infty$ we get,

$\hbox{\rm log} \mathbb P_{\Lambda_{0}} \{N^{s}(\partial A)=0\}=-\mathbb Q_{\Lambda_{1}} \{\displaystyle\mathop{\Sigma}_{1}^{\infty} 1_{\partial A}(S_{i}y)=0, y\in \partial A\}$.

\noindent Since $\Lambda_{0} (\partial A)=0$ we have $\mathbb Q_{\Lambda_{1}}-$a.e., $\displaystyle\mathop{\Sigma}_{1}^{\infty} 1_{\overline{A}} (S_{i}y)=\displaystyle\mathop{\Sigma}_{1}^{\infty} 1_{A} (S_{i } y)$, hence for $A$ bounded,

 $\hbox{\rm log} \mathbb E (\hbox{\rm exp}(-\lambda N^{s}(A)))=-\int (1-\hbox{\rm exp}(-\lambda 1_{A} (y))) \mathbb E (\hbox{\rm exp}(-\lambda \displaystyle\mathop{\Sigma}_{1}^{\infty} 1_{A} (S_{i} y))) d\Lambda_{1}(y)$.

\noindent For $A$ unbounded we proceed by exhaustion in $U'_{1}$ with the sets $U_{t}\cap U'_{1}$. Using again the formula for general $A$ as above, we get 

$\mathbb Q_{\Lambda_{0}} \{ N^{s} (\partial A)=0\}=\hbox{\rm exp}(-\mathbb Q_{\Lambda_{1}}\{\displaystyle\mathop{\Sigma}_{1}^{\infty} 1_{\partial A}
 (S_{i}y), y\in \partial A\})$. 
 
 \noindent Since $\Lambda_{1}(\partial A)=0$, we have  $\mathbb Q_{\Lambda} \{N^{s} (\partial A)=0\}=1$.
 \end{proof}
 \vskip 3mm

\begin{proco} \textbf{4.7}     We have by definition, with $n(t)=[z t^{\alpha}]$  :

$\mathbb P_{\rho} \{t^{-\alpha} \tau_{tA}^{x} > z\}=\mathbb P_{\rho}\{N_{n(t)}^{s} (tA)=0\}$ 

\noindent and we know the weak convergence of $N_{n}^{s}$ to $N^{s}$. We observe that, for a sequence $a_{n}>0$ converging to $a>0$, and for any function $\varphi$ on $U'_{1}$ with $N^{s}$-negligible discontinuities, the convergence of $N'_{n} (\varphi)=\int \varphi (a_{n} v) dN^{s}(v)$ to $\int  \varphi (av) d N^{s}(v)$ is valid. 

\noindent Since $\Lambda_{0}(\partial A)=0$ and $\displaystyle\mathop{\lim}_{t\rightarrow \infty} t n(t)^{-1/\alpha}=z^{-1/\alpha}$ :

$\displaystyle\mathop{\lim}_{t\rightarrow \infty} \mathbb P_{\rho} \{N_{n (t)}^{s} (tA)=0\}=\mathbb Q_{\Lambda_{1}} \{N^{s}(z^{-1/\alpha}A)=0 \}$.

\noindent Using Lemma 4.13 and the expression of the Laplace transform of  $N^{s} (A)$, we get

$\displaystyle\mathop{\lim}_{t\rightarrow \infty} \mathbb P_{\rho} \{t^{-\alpha} \tau_{tA}^{x} >z\}=\hbox{\rm exp}( - c z \theta (A))$. 

\noindent In this formula, we can replace $\mathbb P_{\rho}$ by $\mathbb P$, since  for any $x\in V$, the point process $N_{n}^{s}$ converges also weakly  to $N^{s}$ under $\mathbb P$. This follows from the use of $\varepsilon$-H\"older functions for $\varepsilon<\inf (1,\alpha)$, Lemma 4.8 and the inequality

$|\varphi (S_{n}x)-\varphi(S_{n}y)|\leq [\varphi]_{\varepsilon} |S_{n}|^{\varepsilon} |x-y|^{\varepsilon}$,

\noindent since $\displaystyle\mathop{\Sigma}_{1}^{\infty} |S_{n}|^{\varepsilon}<\infty$,  $\mathbb Q-$a.e. and $\int |y|^{\varepsilon} d\rho (y)<\infty$.

\noindent The proof of positivity for $\theta(A)$ is the same as in Proposition 2.4 with $A$ instead of supp$(\Lambda_{1})$ and $\theta(A)$ instead of $\theta$.
\end{proco}

\section{ Convergence to stable laws}
 The convergence to stable laws of the normalized sums $\displaystyle\mathop{\Sigma}_{i=1}^{n} X_{i}$ under  (c-e) was shown in (\cite{10}, \cite{15}) where explicit formulae for the corresponding characteristic functions were given and non degeneracy of the limiting laws was proved.  It was observed  that these formulae involved the asymptotic tail $\Lambda$ of $\rho$, as well as the occupation measure $\pi_{v}^{\omega}=\displaystyle\mathop{\Sigma}_{0}^{\infty} \varepsilon_{S_{i}(\omega) v}$. A similar situation occured in the dynamical context of \cite{11}, where the limiting law was expressed in terms of an induced transformation.  We observe that the connection with convergence to stable laws for $\displaystyle\mathop{\Sigma}_{i=1}^{n} X_{i}$, where $(X_{i})_{i \in \mathbb N}$ is a stationary process, and point process theory had been already developed in \cite{8}. For an  analysis of the involved properties in this setting see \cite{24}. For another approach, not using spectral gap properties, see chapter 4 of the recent book \cite{5}. Here we give new proofs of the results given  in (\cite{10}, \cite{15}), following and completing the point process approach of \cite{7} in the case of affine stochastic recursions. In particular we get also a direct proof of the convergence for the related  space point process $N^{s}=\displaystyle\mathop{\Sigma}_{i=1}^{n} \varepsilon_{u_{n}^{-1} X_{i}}$, via a detailed analysis of Laplace functionals and without use of the cluster process.

\subsection{\sl  On the space exceedances process}
\vskip 3mm
\noindent We give here a direct proof of the convergence of $N_{n}^{s}$, already shown in Corollary 4.3 above and we deduce the convergence of the characteristic function for the random variable $N_{n}^{s} (f)$, for $f$ compactly supported. We make  use of the mixing property in Proposition 3.5 for Lipschitz functions depending only on $v\in V$.
\vskip 3mm
\begin{thm}  
 Let $f$ be a complex valued compactly supported Lipchitz function on $V\setminus\{0\}$ which satisfies $Re(f)\geq0$. Then we have
$$-\hbox{\rm log}\displaystyle\mathop{\lim}_{n\rightarrow \infty} \mathbb E_{\rho}\hbox{\rm exp}(-N_{n}^{s}(f))=\mathbb E_{\Lambda_{0}} [(\hbox{\rm exp} f(v)-1) \hbox{\rm exp}(-\pi_{v}^{\omega} (f))].$$
\end{thm}

\noindent The proof depends on two lemmas where notations for $r_{n}$, $k_{n}$ explained above are used. For $i\leq j$ we write $C_{n}(i,j)=\hbox{\rm exp}(-\displaystyle\mathop{\Sigma}_{k=i}^{j} f(u_{n}^{-1} X_{k}))-1$, and we note the equality

\centerline{$C_{n}(1,r_{n})=\displaystyle\mathop{\Sigma}_{i=1}^{r_{n}}  [C_{n}(i,r_{n})-C_{n}(i+1, r_{n})]$}

\noindent where $C_{n} (r_{n}+1, r_{n})=0$ and $r_{n}$ is a sequence as in Proposition 2.6. We note also that

\centerline{$|C_{n}(i,j)|\leq 2,\ |C_{n}(i,j)-C_{n}(i+1, j)|\leq 2.$}

\noindent We are going to compare $C_{n}(1, r_{n})$ to $C_{n,k}(1,r_{n})$ for $k$ large, where 

\centerline{$C_{n,k} (1,r_{n})=\displaystyle\mathop{\Sigma}_{i=1}^{r_{n}}  [C_{n}(i,i+k)-C_{n}(i+1, i+k)],$}

\noindent we write $\Delta_{n,k}$ for  their difference, $\varepsilon_{n}=r_{n} \mathbb P_{\rho} \{|X|>u_{n}\}$ and we assume that  $\hbox{\rm supp}(f)\subset U'_{\delta}$ with $\delta>0$.

\noindent Using anticlustering we will obtain the approximation of the small quantity $C_{n}(1,r_{n})$ by $C_{n,k} (1, r_{n})$ up to $\varepsilon_{n}$. Then taking $k_{n}=[n r_{n}^{-1}]$ large and using the definition of $\Lambda$, we will get the limiting form of $k_{n} \mathbb E_{\rho} [C_{n,k}(1, r_{n})]$, hence of $k_{n} \mathbb E_{\rho} [C_{n} (1, r_{n})]$. We need the two following lemmas. 
\vskip 3mm
\begin{lem}  
 $\displaystyle\mathop{\lim}_{k\rightarrow \infty} \ \displaystyle\mathop{\lim\sup}_{n\rightarrow \infty} \varepsilon_{n}^{-1}  \mathbb E_{\rho} (|\Delta_{n,k}|)=0$.
\end{lem}

\begin{proof}
  We observe that 

\centerline{$1+C_{n}(i,r_{n})=(1+C_{n}(i+1, r_{n}))\hbox{\rm exp}(-f(u_{n}^{-1}X_{i})).$}

\noindent Also $C_{n}(i+1, r_{n})-C_{n}(i+1, i+k)=(\hbox{\rm exp}(-\displaystyle\mathop{\Sigma}_{j=i+1}^{i+k} f(X_{j}))) (\hbox{\rm exp}(-\displaystyle\mathop{\Sigma}_{j=i+k+1}^{r_{n}} f(X_{j}))-1)$.

\noindent Hence we have
$$\Delta_{n,k}=\displaystyle\mathop{\Sigma}_{i=1}^{r_{n}} (\hbox{\rm exp}(-f(u_{n}^{-1} X_{i}))-1) [C_{n}(i+1, r_{n})-C_{n}(i+1, i+k)]=\Delta'_{n,k} +\Delta''_{n,k}$$
where $\Delta'_{n,k}$ (resp. $\Delta''_{n,k}$) is the above sum with index $i$ restricted to $[1,r_{n}-k]$ (resp. $]r_{n}-k, r_{n}]$). As observed above, the expression under $\Sigma$ is bounded by 4 and vanishes unless  $|X_{i}|>\delta u_{n}$ for some $i\in [1, r_{n}-k]$ and $M_{k+i+1}^{r_{n}} >\delta u_{n}$. Then we get using stationarity,

\centerline{$\mathbb E_{\rho} (|\Delta'_{n,k}|)\leq 4 r_{n} \mathbb P_{\rho} \{|X_{0}|> \delta u_{n}, M_{k+1}^{r_{n}}> \delta  u_{n}\}.$}

\noindent Since the process $(X_{k})_{k\in \mathbb Z_{+}}$ satisfies anticlustering, it follows $\displaystyle\mathop{\lim}_{k\rightarrow \infty}\ \displaystyle\mathop{\lim\sup}_{n\rightarrow \infty} \varepsilon_{n}^{-1} \mathbb E_{\rho}(|\Delta'_{n,k}|)=0$. Also, stationarity implies
$$\mathbb E_{\rho} (|\Delta''_{n,k}|) \leq 4 \displaystyle\mathop{\Sigma}_{i=r_{n}-k+1}^{r_{n}}  \mathbb P_{\rho}\{|X_{i}| > \delta u_{n}\}= 4k \mathbb P_{\rho} \{|X|> \delta u_{n}\}.$$
Since $\rho$ is homogeneous at infinity and $\displaystyle\mathop{\lim}_{n\rightarrow \infty} r_{n}^{-1} k=0$, we get $\displaystyle\mathop{\lim}_{k\rightarrow \infty}\ \displaystyle\mathop{\lim\sup}_{n\rightarrow \infty} \varepsilon_{n}^{-1} \mathbb E_{\rho}(|\Delta''_{n,k}|)=0$, hence the required assertion. 
\end{proof}

\begin{lem}  We have the following convergences. 

 1) For any $k\geq 1$
$\displaystyle\mathop{\lim}_{n\rightarrow \infty} k_{n} \mathbb E_{\rho} [C_{n,k} (1, r_{n})]=  \mathbb E_{\Lambda_{0}} [(\hbox{\rm exp} f(v)-1) \hbox{\rm exp}(-\displaystyle\mathop{\Sigma}_{i=0}^{k} f(S_{i}v))]$

2) $\displaystyle\mathop{\lim}_{k\rightarrow \infty}\ \displaystyle\mathop{\lim\sup}_{n\rightarrow \infty} k_{n} \mathbb E_{\rho} (|\Delta_{n,k}|)=0$
\end{lem}

\begin{proof}
 1) Using stationarity we have

\centerline{$k_{n} \mathbb E_{\rho} [C_{n,k} (1, r_{n})]=k_{n}  r_{n} \mathbb E_{\rho} [C_{n} (1, k+1)-C_{n}(2, k+1)]=$}

\centerline{$k_{n} r_{n} \mathbb E_{\rho}[\hbox{\rm exp}(-\displaystyle\mathop{\Sigma}_{j=0}^{k} f(u_{n}^{-1} X_{j}))-\hbox{\rm exp}(-\displaystyle\mathop{\Sigma}_{j=1}^{k} f (u_{n}^{-1} X_{j}))].$}

\noindent The function $f^{(k)}$ on $(V\setminus\{0\})^{k+1}$ given by 

\centerline{$f^{(k)} (x_{0}, x_{1}\cdots, x_{k})=\hbox{\rm exp}( -\displaystyle\mathop{\Sigma}_{j=0}^{k} f(x_{j}))-\hbox{\rm exp}( -\displaystyle\mathop{\Sigma}_{j=1}^{k} f(x_{j}))=-(\hbox{\rm exp}(\ f (x_{0}))-1) \hbox{\rm exp}(-\displaystyle\mathop{\Sigma}_{j=0}^{k} f(x_{j}))$}

\noindent is bounded, uniformly continuous on $(U'_{\delta})^{k+1}$ and $\displaystyle\mathop{\lim}_{n\rightarrow \infty} n^{-1} k_{n} r_{n}=1$. Hence, Proposition 2.4 implies
$$\displaystyle\mathop{\lim}_{n\rightarrow \infty} k_{n} \mathbb E_{\rho} [C_{n,k} (1, r_{n})]= -\mathbb E_{\Lambda_{0}} [(\hbox{\rm exp} f(v)-1) \hbox{\rm exp}(-\displaystyle\mathop{\Sigma}_{j=0}^{k} f(S_{j}v))].$$
2) We have the equality, 
$k_{n} \mathbb E_{\rho} (|\Delta_{n,k}|)=k_{n} r_{n} \mathbb P_{\rho} \{|X|> u_{n}\} \varepsilon_{n}^{-1} \mathbb E_{\rho} (|\Delta_{n,k}|).$
\vskip 2mm
\noindent Then, using Lemma 5.2, the relation $\displaystyle\mathop{\lim}_{n\rightarrow \infty} n^{-1} k_{n} r_{n} =1$ and the homogeneity at infinity of $\rho$, assertion 2 follows.
\end{proof}
\vskip 3mm
\begin{proth} \textbf{5.1}  
 With $r_{n}$ as in Proposition 4.2 above, the multiple mixing property in Proposition 3.5  for functions depending only of $v\in V$gives
$$\displaystyle\mathop{\lim}_{n\rightarrow \infty} [\mathbb E_{\rho} (\hbox{\rm exp}(-N_{n}^{s} (f)))-(\mathbb E_{\rho} (1+C_{n} (1, r_{n})))^{k_{n}}]=0,$$
hence it suffices to study the sequence $(1+\mathbb E_{\rho} (C_{n} (1, r_{n})))^{k_{n}}$. Since $Ref\geq 0$ and supp$(f)\subset U'_{\delta}$ we have :
$$\mathbb E_{\rho}(|C_{n}
(1,r_{n})|)\leq \mathbb E_{\rho} (|1-\hbox{\rm exp}(-\displaystyle\mathop{\Sigma}_{i=1}^{r_{n}} f(u_{n}^{-1} X_{i}))|)\leq  \mathbb E_{\rho} (\displaystyle\mathop{\Sigma}_{i=1}^{r_{n}} |f(u_{n}^{-1} X_{i})|)\leq r_{n} |f|_{\infty} \mathbb P_{\rho} \{|X_{0}|> \delta u_{n}\}.$$
The  last inequality implies the $\mathbb L^{1}$-convergence to zero of $\displaystyle\mathop{\Sigma}_{i=1}^{r_{n}} f(u_{n}^{-1} X_{i})$. Then  the first  one  gives $\displaystyle\mathop{\lim}_{n\rightarrow \infty} \mathbb E_{\rho} (|C_{n}(1, r_{n})|)=0$.

\noindent It follows that the behaviour of the sequence $[1+\mathbb E_{\rho} (C_{n} (1, r_{n}))]^{k_{n}}$ for $n$ large is determined by the behaviour of $k_{n} \mathbb E_{\rho} (C_{n} (1, r_{n}))$. We have for $k\geq 1$,
$$k_{n} \mathbb E_{\rho} (C_{n} (1, r_{n}))=k_{n} \mathbb E_{\rho} (C_{n,k} (1, r_{n}))+k_{n} \mathbb E_{\rho} (\Delta_{n,k}).$$
Since supp$(f) \subset U'_{\delta}$ and $\displaystyle\mathop{\lim}_{j\rightarrow \infty} |S_{j} v|=0$ $\mathbb Q-$a.e., the series $\displaystyle\mathop{\Sigma}_{j=1}^{\infty} f(S_{j} v)$ converges $\mathbb Q-$a.e. Since $Ref\geq 0$, it follows 
$\displaystyle\mathop{\lim}_{k\rightarrow \infty} \mathbb E (\hbox{\rm exp}(-\displaystyle\mathop{\Sigma}_{j=1}^{k} f(S_{j}v)))= \hbox{\rm exp}(-\pi_{v}^{\omega} (f)).$

\noindent Then dominated convergence and Lemma 5.3 imply 
$$\displaystyle\mathop{\lim}_{k\rightarrow \infty} \ \displaystyle\mathop{\lim}_{n\rightarrow \infty} k_{n} \mathbb E_{\rho} (C_{n,k} (1, r_{n}))= -\mathbb E_{\Lambda_{0}} [(\hbox{\rm exp} f(v)-1) \hbox{\rm exp}(-\pi_{v}^{\omega}(f))].$$
This equality and the second assertion in Lemma 5.3 give the result. 
\end{proth}
\vskip 3mm
\begin{cor}  
 Let $m>0$, $\delta>0$, $\gamma\geq 0$, and let $f$ be a $\mathbb R^{m}$-valued continuous function on $V\setminus\{0\}$ which satisfies the conditions 

1) $f$ is locally Lipschitz

2) $f(v)=0$ for $|v|<\delta$

3) $\displaystyle\mathop{\sup}_{v\in V} |v|^{-\gamma} |f(v)|=c_{\gamma}<\infty$

\noindent Then we have , for any $u\in \mathbb R^{m}$, 

\centerline{$\hbox{\rm log}\displaystyle\mathop{\lim}_{n\rightarrow \infty} \mathbb E_{\rho} (\hbox{\rm exp}(-i<u, N_{n}^{s} (f)>)=-\mathbb E_{\Lambda_{0}} [(\hbox{\rm exp}( \ i<u, f(v)>)-1) \hbox{\rm exp}(-i<u, \pi_{v}^{\omega} (f)>)]$}
\end{cor}

\begin{proof}  
 We consider the $\mathbb R^{m}$-valued random variable $Y_{n}=N_{n}^{s} (f)$. For $a\geq 1$, let $\theta_{a} (v)$ be the function from $V\setminus \{0\}$ to $[0,1]$ defined by
$$\theta_{a} (v)=1\ \hbox{\rm for}\ |v|\leq a,\ \theta_{a}(v)=a+1-|v|\ \hbox{\rm for}\ |v| \in [a, a+1],\ \theta_{a}(v)=0\ \hbox{\rm for}\ |v| \geq a+1.$$
Then $\theta_{a}$ is Lipschitz, hence $f \theta_{a}$ is Lipschitz and compactly supported. Then Theorem 5.1 gives, in logarithmic form

\noindent $\displaystyle\mathop{\lim}_{n\rightarrow \infty} \mathbb E_{\rho} (\hbox{\rm exp}\ i<u, N_{n}^{s} (f \theta_{a})>)= -\mathbb E_{\Lambda_{0}}[(\hbox{\rm exp}(-i<u,  f \theta_{a} (v)>)-1)\hbox{\rm exp}(  \pi_{v}^{\omega}(i<u, f \theta_{a} (v)>))]=\Phi_{a} (u)$

\noindent Since $\Lambda (U'_{\delta}) <\infty$, the function $u\rightarrow \Phi_{a} (u)$ is continuous on $\mathbb R^{m}$. It follows that the sequence of random variables $Y_{n}^{a}=N_{n}^{s} (f\theta_{a})$ converges in law to the random variable $Y_{\infty}^{a}$ which has characteristic function $\Phi_{a}$. On the other hand we have $\displaystyle\mathop{\lim}_{a\rightarrow \infty} \theta_{a}=1$, hence by dominated convergence we get  in logarithmic form,
$$\displaystyle\mathop{\lim}_{a\rightarrow \infty} \Phi_{a}(u)=-\mathbb E_{\Lambda_{0}} [(\hbox{\rm exp}(-i<u, f(v)>)-1) \hbox{\rm exp}(i<u, \pi_{v}^{\omega} (f)>)]=\Phi (u).$$
We recall that, for $v$ fixed, the series $\displaystyle\mathop{\Sigma}_{j=0}^{\infty} f(S_{j}v)$ converges $\mathbb Q-$a.e. to a finite sum, hence the function $u\rightarrow \Phi (u)$ is continuous. In other words, $Y_{\infty}^{a}$ converges in law $(a\rightarrow \infty)$ to the random variable $Y$ with characteristic function $\Phi$.

\noindent  If we choose $\beta \in ]0,1[$, $\gamma>0$ such that $\beta \gamma \in ]0,\alpha[$, then  for any $\varepsilon >0$,  Markov's inequality gives
$$\mathbb P_{\rho}\{|Y_{n}^{a}-Y_{n}|> \varepsilon\} \leq \varepsilon^{-\beta} \mathbb E_{\rho} [\displaystyle\mathop{\Sigma}_{j=1}^{n} (f(u_{n}^{-1} X_{j}) 1_{\{|X_{j}|>a u_{n}\}})^{\beta}],$$
$$\mathbb P_{\rho}\{|Y_{n}^{a}-Y_{n}|> \varepsilon\} \leq \varepsilon^{-\beta} c_{\gamma}^{\beta}  n \mathbb E_{\rho} [|u_{n}^{-1} X|^{\beta \gamma}  1_{\{|X|>a u_{n}\})}].$$
Using Corollary 2.2,   it follows 
$\displaystyle\mathop{\lim\sup}_{n\rightarrow \infty} \mathbb P_{\rho} \{|Y_{n}^{a}-Y_{n}|> \varepsilon\} \leq \varepsilon^{-\beta}c_{\gamma}^{\beta} \Lambda (W^{\beta \gamma} 1_{\{W>a\}}))$, where $W(x)=|x|$.

\noindent Since $0<\beta \gamma<\alpha$, we get
$\displaystyle\mathop{\lim}_{a\rightarrow \infty}\ \displaystyle\mathop{\lim\sup}_{n\rightarrow \infty} \mathbb P_{\rho}\{|Y_{n}^{a}-Y_{n}|>\varepsilon\}=0$

\noindent Since $\varepsilon>0$ is arbitrary, the convergence in law of $Y_{n}$ to the random variable $Y$ follows, hence the corollary. 
\end{proof}

\noindent In order to prepare the study of limits for the sums $T_{n}=\displaystyle\mathop{\Sigma}_{j=1}^{n} X_{j}$ if $0<\alpha<2$, we write for $a>0$  $\psi_{a} (v)=v (1-\varphi_{a}(v))$, where 

$\varphi_{a} (v)=1$ if $|v|\leq a$, $\varphi_{a}(v)=2-a^{-1}|v|$ if $a\leq |v|\leq 2a$, $\varphi_{a}(v)=0$ if $|v|>2a$. 

\noindent Hence $0\leq \varphi_{a}\leq 1_{[0,2a]}$ and $k(\varphi_{a})\leq a^{-1}$. In particular we will use below the function $\varphi_{1}$ corresponding to $a=1$. Then a consequence of Corollary 5.4, if $\gamma=1$ is the following

\begin{cor}  
 The sequence of $V$-valued random variables $N_{n}^{s} (\psi_{a})$ converges in law to the random variable with characteristic function which logarithm is 
 
 $ \mathbb E_{\Lambda_{0}} [(\hbox{\rm exp}(-i<u, \psi_{a}(v)>)-1)\hbox{\rm exp}( i<u, \pi_{v}^{\omega} (\psi_{a})>)]$. 
\end{cor}

\subsection{\sl  Convergence to stable laws for $T_{n}=\displaystyle\mathop{\Sigma}_{i=1}^{n} X_{i}$}
 In this subsection we write $\psi(v)=v$ and we study the convergence of $N_{n}^{s} (\psi)=u_{n}^{-1}$ $T_{n}$ towards a stable law, relying on the weak convergence of $N^{s}_{n}$ studied in the above subsection. We need here the last part of the spectral gap result in Proposition 3.4 for the operator $P$.

\noindent We have the 
\vskip 2mm
\begin{thm}  
 Let $0<\alpha<2$. Then there exists a sequence $d_{n}$ in $V$ such that the sequence of random variables $n^{-1/\alpha} (T_{n}-d_{n})$ converges in law to a non degenerate stable law.

\noindent If $0<\alpha<1$, we have $d_{n}=0$.

\noindent If $1<\alpha<2$, we have $d_{n}=n \mathbb E_{\rho}(X)$

\noindent If $\alpha=1$, we have $d_{n}=n$ $\mathbb E_{\rho} [X \varphi_{1} (X)]$. 
\end{thm}

\vskip 2mm
\noindent Explicit expressions for the characteristic functions of the limits are given in the proofs. Non degeneracy of the limit laws are proved in \cite{10} and \cite{15}. For the proofs, we follow the approach of \cite{7} and we need two lemmas corresponding to the cases $0<\alpha<1$ and $1\leq \alpha<2$.

\noindent In the proofs below we use the normalization $u_{n}=(\alpha^{-1} cn)^{1/\alpha}$ instead of $n^{1/\alpha}$ as in the theorem.
\vskip 3mm
\begin{lem} 
Assume  $0<\alpha<1$. Then for any $u\in V$ and with the notation of Corollary 5.5, $T^{a}=N^{s}(\psi_{a})$ converges in law $(a\rightarrow 0)$ to  $T$ with characteristic function $\Phi(u)$ given by

$-\hbox{\rm log} \Phi (u)=  c^{-1} \mathbb E_{\Lambda_{0}} [(\hbox{\rm exp}(-i<u,v>)-1) \hbox{\rm exp}( i<u, \displaystyle\mathop{\Sigma}_{j=0}^{\infty} S_{j}v> )]$.

\noindent Also, for any $\delta>0$ we have 

$\displaystyle\mathop{\lim}_{a\rightarrow 0}\ \displaystyle\mathop{\lim\sup}_{n\rightarrow \infty} \mathbb P_{\rho} \left\{ |\displaystyle\mathop{\Sigma}_{j=1}^{n} u_{n}^{-1} X_{j} \varphi_{a}(u_{n}^{-1} X_{j})| > \delta\right\}=0$.
\end{lem}

\begin{proof} 
 Using dominated convergence, the first part follows from Corollary 5.5.

\noindent On the other hand, Markov's inequality gives 

$\mathbb P_{\rho} \left\{| \displaystyle\mathop{\Sigma}_{j=1}^{n} u_{n}^{-1} X_{j} \varphi_{a} (u^{-1}X_{j})|>\delta\right\} \leq n \delta^{-1} u_{n}^{-1} \mathbb E_{\rho} (|X| 1_{\{|X|<2au_{n}\}})$

\noindent The homogeneity at infinity of $\rho$ and Karamata's lemma (see \cite{30} p.26) gives that the right hand side is equivalent to 

$\delta^{-1} n^{1-\alpha^{-1}} \alpha (1-\alpha)^{-1} (2a u_{n}) \mathbb P_{\rho} \{|X|>2au_{n}\}$,

\noindent i.e. to $\delta^{-1}$ $a^{1-\alpha}$, up to a coefficient independant of $n$. Since $1-\alpha>0$ the result follows. 
 \end{proof}

\begin{lem}  
 Assume $1\leq \alpha<2$ and write $\overline{\psi}_{a} (v)=v \varphi_{a} (v)$ where $\varphi_{a}$ is defined in the proof of Corollary 5.4. Then we have the convergence 

$\displaystyle\mathop{\lim}_{a\rightarrow 0}\ \displaystyle\mathop{\lim\sup}_{n\rightarrow \infty} \mathbb E_{\rho} (| N_{n}^{s} (\overline{\psi}_{a})- \mathbb E_{\rho} ( N_{n}^{s} (\overline{\psi}_{a})|^{2})=0$.
\end{lem}

\begin{proof}  
 It suffices to show that for any $u\in \mathbb S^{d-1}$ :
$$\displaystyle\mathop{\lim}_{a\rightarrow 0}\ \displaystyle\mathop{\lim\sup}_{n\rightarrow \infty} \mathbb E_{\rho} (|<u, N_{n}^{s}(\overline{\psi}_{a})>-\mathbb E_{\rho} (<u, N_{n}^{s} (\overline{\psi}_{a})>|^{2})=0.$$
 We write 
$f_{a,n} (v)=\overline{\psi}_{a} (u_{n}^{-1} v),\ \overline{\psi}_{a,n}=f_{a,n}-\rho (f_{a,n}).$
 
 \noindent Hence $|f_{a,n} (v)|\leq u_{n}^{-1} |v| 1_{\{|v|\leq 2a u_{n}\}}$, $k(f_{a,n}) \leq 3 u_{n}^{-1}$. We have the equality
 
 $\mathbb E_{\rho} (|<u, N_{n}(\overline{\psi}_{a})>- \mathbb E_{\rho} (<u, N_{n} (\overline{\psi}_{a})>)|^{2})=A_{n,a}+2 B_{n,a}$
 
 \noindent with
 
 $A_{n,a} =n \mathbb E_{\rho} (|<u, \overline{\psi}_{a,n} (X_{0})>|^{2}),\ B_{n,a}=\displaystyle\mathop{\Sigma}_{j=1}^{n} (n-j) \mathbb E_{\rho} (<u, \overline{\psi}_{a,n} (X_{0}) > <u, \overline{\psi}_{a,n} (X_{j}))$.
 
 \noindent Now the proof splits into two parts a) and b) corresponding to the studies of $A_{n,a}, B_{n,a}$.
 
 \noindent a) We have, using the above estimation of $f_{a,n}$,
 $$n \mathbb E_{\rho} (|<u, \overline{\psi}_{a,n} (X_{0})>|^{2}) \leq n \mathbb E_{\rho} (|f_{a,n} (X_{0})|^{2})\leq n \mathbb E_{\rho}(u_{n}^{-2} |X_{0}|^{2} 1_{\{|X_{0}|<2 a u_{n}\}}).$$
  Then Karamata's lemma implies that, for $n$ large, the right hand side is equivalent to $n^{1-2\alpha^{-1}} (2a u_{n})^{2} ((2a)^{\alpha} n)^{-1}$, i.e. to $a^{2-\alpha}$. Hence, since $\alpha\in ]0,2[$, we get 
 $\displaystyle\mathop{\lim}_{a\rightarrow 0}\ \displaystyle\mathop{\lim\sup}_{n\rightarrow \infty}  A_{n,a}=0$
 uniformly in $u\in \mathbb S^{d-1}$. 
 \vskip 2mm
 b) Markov property for the process $(X_{i})_{i\geq 0}$ implies for $i\geq 1$, 
 
 $\mathbb E_{\rho} (<u, \overline{\psi}_{a,n} (X_{0}) > <u, \overline{\psi}_{a,n} (X_{i})>)= \mathbb E_{\rho} (<u, \overline{\psi}_{a,n} (X_{0}) > <u, P^{i} \overline{\psi}_{a,n} (X_{0}) >)$.
 
 \noindent First we consider the case $\alpha \in ]1,2[$ and we apply Proposition 3.4 to $P$ acting on the Banach space $\mathcal{H} =\mathcal{H}_{\chi,\varepsilon,\kappa}$ with $\chi \in ]1,\alpha[$, $\varepsilon=1$ and $\kappa$ choosen according to Proposition 3.4. We observe that for $h\in \mathcal{H}$ we have $|h(v)|\leq \| h\| (1+|v|)$. Since $\overline{\psi}_{a,n} \in \mathcal{H}$, we have 
 
 $\mathbb E_{\rho} (<u, \overline{\psi}_{a,n} (X_{0}) > <u,  P^{i} \overline{\psi}_{a,n} (X_{0})>)= \mathbb E_{\rho} (<u, f_{a,n} (X_{0}) > <u, U^{i} f_{a,n} (X_{0}) >)$,
 
 \noindent where we have used the decomposition $P^{i}=\rho \otimes 1+U^{i}$. The Schwarz inequality allows us to bound the right hand side by the square root of   
 $\mathbb E_{\rho} (|f_{a,n} (X_{0}) |^{2} ) \mathbb E_{\rho} (|U^{i} f_{a,n} (X_{0})|^{2} 1_{\{|X_{0}| < 2 an\}})$.
 
 \noindent Since $|U^{i} f_{a,n} (v)| \leq (1+|v|) \|U^{i}\| \ \|f_{a,n}\|$, the quantity $\mathbb E_{\rho} (<u,f_{a,n} (X_{0}) > <u, U^{i} f_{a,n} (X_{0})>)$ is bounded by  
 $\|U^{i}\|\ \|f_{a,n}\| [u_{n}^{-2} \mathbb E_{\rho} (|X_{0}|^{2} 1_{\{|X_{0}|<2au_{n}\}})]^{1/2}$ $[\mathbb E_{\rho} (1+|X_{0}|)^{2} 1_{\{|X_{0}|<2au_{n}\}})]^{1/2}$.
 
 \noindent Then Karamata's lemma implies that, up to a coefficient independant of $n$, the above expression is bounded by  
  $\| U^{i}\|\ \|f_{a,n}\| [n^{-1/2} a^{1-\alpha/2}]\ [1+n^{\alpha^{-1}-1/2} a^{1-\alpha/2}]$.
 
 \noindent Since $\|f_{a,n}\| \leq n u_{n}^{-1}$,  it follows that $B_{n,a}$, uniformly in $u\in \mathbb S^{d-1}$ and up to a coefficient, is bounded by 
 
 $n (\displaystyle\mathop{\Sigma}_{i=0}^{\infty} \|U^{i}\|) n^{-1} [n^{1/2-\alpha^{-1}} a^{1-\alpha/2}+ a^{2-\alpha}]=\displaystyle\mathop{\Sigma}_{i=0}^{\infty} \|U^{i}\|\  [a^{2-\alpha}+ a^{1-\alpha/2} n^{1/2-\alpha^{-1}}]$.
 
 \noindent Since $r (U)<1$ we have $\displaystyle\mathop{\Sigma}_{i=0}^{\infty} \|U^{i}\|<\infty$, hence $\displaystyle\mathop{\lim\sup}_{n\rightarrow \infty} B_{n,a}$ is bounded by $a^{2-\alpha}$, up to a coefficient independant of $n$. Since $1<\alpha<2$, and in view of the two above convergences the lemma follows.
 
 \noindent If $\alpha=1$, we need  to use the Banach space $\mathcal H'=\mathcal H_{\chi, \varepsilon,\kappa}$ with $0<\varepsilon<\chi<1$, $\kappa=0$, considered in Proposition 3.4. We use also the inequality $\|f\|_{a,n}\leq c_{1} a^{1-\chi} n^{-\varepsilon}$ with $c_{1}>0$, shown below. We note that for $h\in \mathcal H'$, we have $|h(v)| \leq \|h\| (1+|v|^{\varepsilon})$ in particular and  up to a constant independant of $n$ and $a$, we have  
$$|U^{i} f_{a,n}(v)|\leq \|U^{i}\| \ \|f_{a,n}\| (1+|v|^{\varepsilon}) \leq \|U^{i}\| (a^{1-\chi} n^{-\varepsilon}) (1+|v|^{\varepsilon}).$$ 
Hence we can bound $\mathbb E_{\rho} (<u, f_{a,n} (X_{0}) > <u, U^{i} f_{a,n} (X_{0}) >)$ by 
 
 $2 c_{1} \|U^{i}\| (a^{1-\chi} n^{-\varepsilon}) [n^{-2} \mathbb E_{\rho} (|X_{0}|^{2}) 1_{\{|X_{0}|<2na\}})]^{1/2} [\mathbb E_{\rho} (1+|X_{0}|^{2\varepsilon}) 1_{\{|X_{0}| < 2na\}}]^{1/2}$,
 
 \noindent which can be estimated using Karamata's lemma,  for $\varepsilon>1/2$, by
 
 $2 c_{2} \|U^{i}\| (a^{1-\chi} n^{-\varepsilon}) (a n^{-1})^{1/2} (n a)^{\varepsilon-1/2}=c_{2} a^{1-\chi+\varepsilon} n^{-1}$.
 
 \noindent It follows that $B_{n,a}$ can be estimated by $2c_{2}\displaystyle\mathop{\Sigma}_{i=0}^{\infty} \|U^{i}\| a^{1-\chi+\varepsilon}$. Since using Proposition 3.4 we have $r(U)<1$,  if $1/2<\varepsilon<\chi<1$,  it follows $\displaystyle\mathop{\lim}_{a\rightarrow 0}\ \displaystyle\mathop{\lim\sup}_{n\rightarrow \infty} B_{n,a}=0$.
 \end{proof}
 \vskip 3mm
\begin{proth} \textbf{5.6}  
 For $\alpha\in ]0,1[$ the proof follows from Lemma 5.7. We observe that dominated convergence implies the continuity of $\Phi$ at zero, hence $\Phi$ is a characteristic function. From Lemma 5.7 we know that if $Y_{n}=N_{n}^{s}(1)$, $Y_{n}^{a}=N_{n}^{s} (1-\varphi_{a})$, 
 
 1) For any $a>0$, $Y_{n}^{a}$ converges in law $(n\rightarrow \infty)$ to $T^{a}$
 
 2) $T^{a}$ converges in law $(a\rightarrow 0)$ to $T$
 
 3) For any $\varepsilon>0$, we have
$\displaystyle\mathop{\lim}_{a\rightarrow 0}\ \displaystyle\mathop{\lim\sup}_{n\rightarrow \infty} \mathbb P \{|Y_{n}-Y_{n}^{a}|>\varepsilon\}=0.$

 \noindent It follows that the sequence $Y_{n}=\displaystyle\mathop{\Sigma}_{i=1}^{n} u_{n}^{-1} X_{i}$ converges in law $(n\rightarrow \infty
)$ to the random variable $T$ with characteristic function $\Phi$. 

\noindent For $1<\alpha<2$, we write

$Y_{n}=N_{n}^{s}  (\psi)-\mathbb E_{\rho} (N_{n}^{s} (\psi))=u_{n}^{-1} \displaystyle\mathop{\Sigma}_{j=1}^{n} X_{j}-\mathbb E_{\rho}(X), \ \ Y_{n}^{a}=N_{n}^{s} (\psi_{a})-\mathbb E_{\rho} (N_{n}^{s} (\psi_{a}))$,

\noindent so that $Y_{n}^{a}-Y_{n}=N_{n}^{s} (\overline{\psi}_{a,n})-\mathbb E_{\rho} (N_{n}^{s} (\overline{\psi}_{a,n}))$. Then, for any $\varepsilon>0$,  Lemma 5.8 gives,
$$\displaystyle\mathop{\lim}_{a\rightarrow 0}\ \displaystyle\mathop{\lim\sup}_{n\rightarrow \infty} \mathbb P_{\rho} \{|Y_{n}^{a}-Y_{n}|>\varepsilon\}=0$$
Furthermore, the sequence $N_{n}^{s} (\psi_{a})$ converges in law $(n\rightarrow \infty)$ to $T^{a}$ and $\mathbb E_{\rho}(N_{n}^{s} (\psi_{a}))=n \mathbb E_{\rho} [u_{n}^{-1} X (1-\varphi_{a}) (u_{n}^{-1} X)]$ converges to the value $b(a)$ of $\Lambda$ on the function $v\rightarrow v(1-\varphi_{a} (v))$, as follows from $\alpha>1$ and the homogeneity at infinity of $\rho$. Hence the sequence $Y_{n}^{a}$ converges in law $(n\rightarrow \infty)$ to $T^{a}-b(a)=Y^{a}$. Finally $Y^{a}$ converges in law $(a\rightarrow 0)$ to the random variable $T$ with characteristic function $\Phi$ defined in logarithmic form by

 $-\Phi (u)= \mathbb E_{\Lambda_{0}} [(\hbox{\rm exp}(-i<u,v>)-1+i<u,v>) \hbox{\rm exp}( i<u, \displaystyle\mathop{\Sigma}_{j=0}^{\infty} S_{j}v>)]$ 

 $+i \mathbb E_{\Lambda_{0}} [<u,v>(\hbox{\rm exp}(  i<u, \displaystyle\mathop{\Sigma}_{j=0}^{\infty} S_{j}v>)-1)].$

\noindent This follows of Theorem 5.1, of dominated convergence $(a\rightarrow 0)$ and of the following inequalities

$|\hbox{\rm exp}(-i<u, \psi_{a}(v)>)-1+i<u, \psi_{a}(v)>| \leq \inf (2+|u|\ |v|, 4|v|^{2}|u|^{2}),$

$|<u,\psi_{a}(v)>\mathbb E (\hbox{\rm exp}( i<u,\displaystyle\mathop{\Sigma}_{j=1}^{\infty} \psi_{a} (S_{j}v)>) -1)| \leq \inf (|u|\ |v|, 2|u|^{2} |v|^{2} \displaystyle\mathop{\Sigma}_{j=1}^{\infty} \mathbb E |S_{j})$,

\noindent where $\alpha>1$ gives $\displaystyle\mathop{\Sigma}_{j=1}^{\infty} \mathbb E |S_{j}|<\infty$. Continuity of $\Phi$ at zero follows also from the above inequalities.

\noindent Then, as in \cite{7}, we deduce the convergence in law of the sequence $Y_{n}$ to $T$.

\noindent If $\alpha=1$, we write $Y_{n}=\displaystyle\mathop{\Sigma}_{j=1}^{n}  n^{-1} X_{j}-\mathbb E_{\rho}(\varphi_{1} (n^{-1}X))$ and

\centerline{$Y_{n}^{a}=\displaystyle\mathop{\Sigma}_{j=1}^{n}  n^{-1}X_{i}(1-\varphi_{a} (n^{-1}X_{j}))-b_{n}(a)=N_{n}^{s} (\psi_{a})-b_{n}(a)$}

\noindent where $b_{n}(a)=\mathbb E_{\rho} [X(\varphi_{1}-\varphi_{a}) (n^{-1}X)]$. With the new notations, the above inequalities are still valid. The homogeneity at infinity of $\rho$ gives now 

$\displaystyle\mathop{\lim}_{n\rightarrow \infty} b_{n} (a)=c^{-1} \mathbb E_{\Lambda_{0}} (v(\varphi_{1}-\varphi_{a})) =b(a)$.

\noindent It follows that the sequence $Y_{n}^{a}$ converges in law $(n\rightarrow \infty)$ to the random variable $T_{1}^{a}$ with characteristic function given in logarithmic form by
 $$ -\mathbb E_{\Lambda_{0}} [(\hbox{\rm exp}(-i<u, \psi_{a} (v)>)-1) (\hbox{\rm exp}( i(<u, \pi_{v}^{\omega}(\psi_{a})>)-i <u, b(a)>)]$$
 We insert the expression $i<u, v>(\varphi_{1}-\varphi_{a} (v))$ with the adequate sign in each of the above factors inside the expectation $\mathbb E_{\Lambda_{0}}$. Then dominated convergence $(a\rightarrow 0)$ shows that $T_{1}^{a}-b(a)$ converges in law to the random variable $T$ with characteristic function given in logarithmic form by
$$\Phi(u)= - \mathbb E_{\Lambda_{0}} [A(u,v)+B(u,v)]$$ 
with $A(u,v)=(\hbox{\rm exp}(-i<u,v>)-1+i<u,v>)\varphi_{1}(v)$,

$ B(u,v)=i<u,v> \varphi_{1} (v) (\hbox{\rm exp}( i<u, \displaystyle\mathop{\Sigma}_{j=1}^{\infty} S_{j} v>)-1)$.

\noindent As in (\cite{10}, \cite{15}), the stability of the limiting laws follow from the formula for $\Phi(u)$. If $0 <\alpha<2$, $\alpha\neq 1$ the formula for $\Phi(u)$ shows that for any $n\in \mathbb N$ we have $\Phi^{n} (u)=\Phi (n^{1/\alpha} u)$, hence  $T$ has a stable law of index $\alpha$.

\noindent If $\alpha=1$, we  have with
$\gamma_{n}=c \ \mathbb E_{\Lambda_{0}} [v (\varphi_{1}(n^{-1} v)-\varphi_{1}(v))]$,
$\Phi^{n} (u)=\Phi (n\ u)
\hbox{\rm exp}(- in <u,\gamma_{n}>).$

\noindent This implies that $T$ follows a stable law with index 1. 
\end{proth}
\vskip 2mm
\begin{remark}The idea used above in the proof of the limiting form of $N_{n}^{s}$ can be used  in other similar situations, using only anticlustering and the definition of $\Lambda$. This is the case for the Laplace functional of the cluster process $C$ stated in Proposition 2.6.

\noindent There we write, as in the proof of Theorem 5.1 

$C'_{n}(k,\ell)=\hbox{\rm exp}(-\displaystyle\mathop{\Sigma}_{k}^{\ell} f(u_{n}^{-1} X_{i})) 1_{\{M_{k,\ell}>u_{n}\}})$

$C'_{n}(1,r_{n}=\hbox{\rm exp}(-\displaystyle\mathop{\Sigma}_{1}^{r_{n}} f(u_{n}^{-1} X_{i})) 1_{\{M_{r_{n}}>u_{n}\}})$

$C'_{n,k}(1,r_{n})=\hbox{\rm exp}(-\displaystyle\mathop{\Sigma}_{1}^{r_{n}} [C'_{n} (i, i+k)- C'_{n}(i+1, i+k)])$

\noindent Then the approximation of $C'(1, r_{n})$ by $C'_{n,k}(1, r_{n})$ up to $\varepsilon_{n}$ is still valid. Furthermore, the definition of $\Lambda_{0}$ shows that for $k_{n}$ large $k_{n} \mathbb E_{\rho} [C'_{n} (1, r_{n})]$ is close to 

$I=-\mathbb E_{\Lambda_{0}} ([\hbox{\rm exp}(-\pi_{v}^{\omega} (f) (\hbox{\rm exp}(f(v)) (1_{\{\displaystyle\mathop{\sup}_{i>0}|S_{i}v|>1\}})-1_{\{\displaystyle\mathop{\sup}_{i\geq0}|S_{i}v|>1\}})]$

\noindent To conclude we write using the notation of 2.4

$\mathbb E_{\rho} [\hbox{\rm exp}(-C_{n}]=\mathbb E_{\rho} \{C'_{n}(1, r_{n})/ M_{r_{n}}>u_{n}\}=(\frac{n}{r_{n}} \mathbb E_{\rho} [C'_{n} (1,r_{n})])(\frac{r_{n}}{n \mathbb P_{\rho}\{M_{r_{n}}>u_{n}\}})$.

\noindent From above the first factor converges to $I$. The second factor is asymptotic to 

$(\theta n \mathbb P_{\rho} \{|X_{0}|>u_{n}\})^{-1}$ since $\theta_{n}^{-1}=r_{n} \mathbb P_{\rho} \{|X_{0}|>u_{n}/M_{r_{n}}>u_{n}\}$ converges to $\theta^{-1}$. 

\noindent If \hbox{\rm supp}$(f)\subset U'_{1}, \theta^{-1}I$ is the sum of 

$-\theta^{-1} \mathbb E_{\Lambda_{1}} [\hbox{\rm exp}(-\displaystyle\mathop{\Sigma}_{1}^{\infty}f (S_{i}v)) (1_{\{\displaystyle\mathop{\sup}_{i\geq 1}|S_{i}v|\leq 1\}})]=1$ and of the last expression in Proposition 2.6.
\end{remark}

\section{Appendix : Condition \hbox{\rm (c-e)} is open if $d>1$}

 We denote by $T_{\mu}$ the closed subsemigroup of $G$ generated by supp$(\mu)$, where $\mu$ is a probability on $G$. We consider weak topologies for probability measures on $G$ and on $H$. We denote by $M^{1}(G)$ $(\hbox{\rm resp.} M^{1}(H))$ the set of probabilities on 
$G(\hbox{\rm resp.}\ H)$. We denote by $\mathcal W(H)$ the weak topology on $M^{1} (H)$ defined by the convergence on continuous compactly supported functions as well as  of the moments $\int (\gamma^{k} (g)+|b|^{k} (h)) d\lambda (h)$ for any $k\in \mathbb N$. An element $\gamma \in G$ is said to be proximal if it has a unique simple dominant real eigenvalue.

\begin{thm}

\noindent If $d>1$, condition \hbox{\rm (c-e)} is open in the weak topology $\mathcal W(H)$ on $M^{1}(H)$.
\end{thm}

\vskip 2mm
\noindent We will need the Proposition

\begin{prop}
Condition \hbox{\rm i-p} is open for the weak topology on $M^{1} (G)$.
\end{prop}

\begin{proof}   
Assume $\mu \in M^{1}(G)$, satisfies i-p and let $\mu_{n}\in M^{1}(G)$ be a sequence which converges weakly to $\mu$. Then supp$(\mu_{n})$ and $T_{\mu_{n}}$ are closed subsets of $G$ which  converges to supp$(\mu)$ and $T_{\mu}$ respectively. If $\gamma$ is a proximal element of $T_{\mu}$, then by perturbation theory there exists a neighbourhood  of $\gamma$ in $G$ which consists of proximal elements. Hence there exists $\gamma_{n}\in T_{\mu_{n}}$ which is also proximal.

\noindent On the other hand $T_{\mu_{n}}$ is irreducible for large $n$. Otherwise there exists a proper subspace $W^{n}\subset V$ with $T_{\mu_{n}} (W^{n})=W^{n}$. Let $W\subset V$ be the limit of a subsequence of $W^{n}$. Then, clearly $T_{\mu}(W)=W$, which contradicts the irreducibility of  $T_{\mu}$. 

\noindent In order to show the strong irreducibility of $T_{\mu_{n}}$ for $n$ large, we show the irreducibility of $Zc_{0} (T_{\mu_{n}})$, the connected component of the Zariski closure $Zc(T_{\mu_{n}})$ of $T_{\mu_{n}}$ (see \cite{25}). Since $T_{\mu_{n}}$ is irreducible, the Lie group $Zc_{0} (T_{\mu_{n}})$ is reductive and has finite  index in $Zc(T_{\mu_{n}})$. We decompose $V$ as the direct sum of its isotypic components $V_{i}^{(n)} (1\leq i\leq p_{n})$ under the action of $Zc_{0} (T_{\mu_{n}}) : V=\displaystyle\mathop{\oplus}_{i=1}^{p_{n}} V_{i}^{(n)}$. Since $Zc_{0} (T_{\mu_{n}})$ has finite index in $Zc (T_{\mu_{n}})$ we can assume, by taking a suitable power, that $\gamma_{n} \in Zc_{0} (T_{\mu_{n}})$. The uniqueness of the above decomposition of $V$ and the relation $\gamma_{n} v=\lambda_{n} v$, $v=\displaystyle\mathop{\Sigma}_{i=1}^{p_{n}} v_{i}$, $v_{i}\in V^{(n)}_{i}$, with $\lambda_{n}$ a simple dominant eigenvalue of $\gamma_{n}$ implies $\gamma_{n} v_{i}=\lambda_{n} v_{i}$ ; hence the proximality of $\gamma_{n}$ implies that $v$ belongs to a unique $V_{i}^{(n)}$, to $V_{1}^{(n)}$ say. Also the irreducibility of $T_{\mu_{n}}$ implies that $T_{\mu_{n}}$ permutes the subspaces $V_{i}^{(n)} (1\leq i\leq p_{n})$. Since $V_{1}^{(n)}$ is isotypic and $\gamma_{n}$ is proximal, the subspace $V_{1}^{(n)}$ is $T_{\mu_{n}}$-irreducible. The same is valid for any $V_{i}^{(n)}= g(V_{1}^{(n)})$ since $g \gamma_{n} g^{-1}$ is  also proximal, for $g\in T_{\mu_{n}}$. Assume $Zc_{0} (T_{\mu_{n}})$ is not irreducible for $n$ large ; then it follows that $p_{n}\in ]1,d]$ and $r_{n}=dim \ V_{1}^{(n)} \in [1,d[$. It follows that we can assume $p_{n}=p$ and $r_{n}=r$ for $n$ large with $p>1$, $r<d$. Hence, taking convergent subsequences of $V_{i}^{(n)} (1\leq i\leq p)$ we obtain proper subspaces $V_{i} (1\leq i\leq p)$ which are permuted by $T_{\mu}$ ; the irreducibility of $T_{\mu}$ implies that their sum  is $V$, hence we have $V=\displaystyle\mathop{\oplus}_{1}^{p} V_{i}$, which contradicts the strong irreducibility of $T_{\mu}$. Hence $T_{\mu_{n}}$ satisfies condition i-p for $n$ large. 
\end{proof}
\vskip 3mm
\begin{proth} \textbf{6.1}

\noindent Let $\lambda_{n}\in M^{1}(H)$ be a sequence which converges to $\lambda \in M^{1} (H)$ in the weak topology $\mathcal W(H)$ and let us denote by $\mu_{n}$ the projection of $\lambda_{n}$ on $G$. We verify the stability of conditions 1, 2 in (c-e) if $d>1$,  since condition 3 follows of the definition of $\mathcal W(H)$ and condition 4 is a direct consequence of Proposition 6.2. 
We denote by $\mathbb S^{d-1}_{\infty}$ the sphere at infinity of $V$ and we observe that the group $H$ acts continuously on the compactification $(V) \cup \mathbb S^{d-1}_{\infty}$ endowed with the visual topology.

1) Assume that supp$(\lambda_{n})$ has a fixed point $x_{n}\in V$ for  $n$ large. Since the closed subset supp$(\lambda_{n})$ converges to supp$(\lambda)$, we can find a convergent subsequence of $x_{n}$ to a point $x$ in $(V) \cup (\mathbb S^{d-1}_{\infty})$,   such that $x$ is supp$(\lambda)$-invariant. If $x \in V$ we have a contradiction since supp$(\lambda)$ has no fixed point  in $V$. If $x\in \mathbb S^{d-1}_{\infty}$, we have also a contradiction since  condition i-p implies that the projective action of supp$(\mu)$ has no fixed point, if $d>1$.

2) Using Lemma 6.4, since finiteness of moments for $\mu_{n}$ is valid, we get that for $\mu_{n}$ and for any $s\geq 0$, the corresponding operator $P^{s}$ has a spectral gap on the relevant H\"older space on $\mathbb S_{\infty}^{d-1}$ (see \cite{13}). The moment condition implies that perturbation theory is valid for the operators $P^{s}$. Hence the spectral radius $k(s)$ varies continuously. In particular, since we have $k(s)>1$ for $\mu$ and $s>\alpha$, and $L(\mu)<0$, the same is valid for $\mu_{n}$ with $n$ large. Hence there exists $\alpha_{n}>0$ close to $\alpha$ such that $k(\alpha_{n})=1$ and $L(\mu_{n})<0$. 
\end{proth}

\textbf{Acknowledgment}

\noindent This research, did not receive any specific grant from funding agencies in the public, commercial, or not-for-profit sectors.

\end{document}